\documentclass[a4paper,reqno, 11pt]{amsart}
\usepackage{fullpage}

\usepackage{color}
\usepackage[capitalize,nameinlink,noabbrev,nosort]{cleveref}
\usepackage{graphicx}
\usepackage{tikz-cd}
\usepackage{mathtools,dsfont}

\usepackage{enumerate}

\usepackage{caption}
\usepackage{subcaption}

\usepackage{graphicx,color,marginnote}
\usepackage{amsmath}
\usepackage[utf8]{inputenc}
\usepackage[T1]{fontenc}
\usepackage[noadjust]{cite}
\usepackage[english]{babel}
\usepackage[normalem]{ulem} 

\usepackage{amsfonts}
\usepackage{amssymb}
\usepackage{bbm}
\usepackage{epstopdf}
\usepackage{color}
\numberwithin{equation}{section}

\newcommand{\cT}{\mathcal T}

\newcommand{\C}{\mathbb{C}}
\newcommand{\R}{\mathbb{R}}
\newcommand{\cO}{\mathcal{O}}

\newcommand{\tb}{\bullet}
\newcommand{\tw}{\circ}

\usepackage{amsthm}
\theoremstyle{plain}
\newtheorem{theorem}{Theorem}[section]

\newtheorem{corollary}[theorem]{Corollary}

\newtheorem*{assumption*}{Assumption}

\newtheorem{lemma}[theorem]{Lemma}
\newtheorem{definition}[theorem]{Definition}
\newtheorem{proposition}[theorem]{Proposition}

\theoremstyle{remark}
\newtheorem{remark}[theorem]{Remark}

\newcommand{\old}[1]{}

\renewcommand{\Im}{\operatorname{Im}}
\renewcommand{\Re}{\operatorname{Re}}

\newcommand{\G}{\mathcal{G}}
\newcommand{\T}{\mathcal{T}}
\newcommand{\Or}{\mathcal{O}}

\newcommand{\Hex}{\operatorname{Hex}}

\renewcommand{\i}{\mathrm{i}}
\newcommand{\e}{\mathrm{e}}

\def\cF{\mathcal{F}}

\usepackage{bm}
\usepackage{color}

\graphicspath{ {../../Images/} }

\title[title]{Perfect t-embeddings and Lozenge Tilings}

\author[Tomas Berggren]{Tomas Berggren$^\mathrm{a}$}

\author[Matthew Nicoletti]{Matthew Nicoletti$^\mathrm{b}$}

\author[Marianna Russkikh]{Marianna Russkikh$^\mathrm{c}$}

\thanks{\textsc{${}^\mathrm{A}$ Massachusetts Institute of Technology, Department of Mathematics, 77 Massachusetts Avenue, Cambridge, Massachusetts, 02139, United States of America}}

\thanks{\textsc{${}^\mathrm{B}$ UC Berkeley, Department of Statistics, Berkeley, CA 94720, United States of America}}
\thanks{\textsc{${}^\mathrm{C}$  University of Notre Dame, Department of Mathematics, 255 Hurley Bldg,
Notre Dame, IN 46556, United States of America}}

\thanks{\texttt{tomasb@mit.edu}, \texttt{mnicoletti@berkeley.edu}, \texttt{mrusskik@nd.edu}}

\begin{document}

\maketitle

\begin{abstract}
We construct perfect t-embeddings for regular hexagons of the hexagonal lattice, providing the first example, and hence proving existence, for graphs with an outer face of degree greater than four. The construction is in terms of the inverse Kasteleyn matrix and relies only on symmetries of the graph. Using known formulas for the inverse Kasteleyn matrix, we derive exact contour integral formulas for these embeddings and their origami maps. Through steepest descent analysis, we establish scaling limits, proving convergence of origami maps to a maximal surface in the Minkowski space~$\mathbb{R}^{2,1}$, and we verify structural rigidity conditions, leading to a new proof of convergence of height fluctuations to the Gaussian free field. 

\end{abstract}

\tableofcontents

\section{Introduction}

\subsection{Overview}
Dimer models are probability measures on the set of 
perfect matchings of (bipartite planar) graphs. 
These models exhibit so-called limit shape phenomena. More precisely, there are stepped surfaces, or \emph{height functions}~$h$ defined on the faces of the graph, which are in 1-1 correspondence with perfect matchings on the bipartite planar graph. 
These height functions, normalized appropriately, concentrate near some surface (known as the \emph{limit shape}), at the large scale~\cite{CKP00}. 
In this paper, we are interested in fluctuations of the height function. In the scaling limit, it is expected (and proven rigorously in several cases, e.g.~\cite{BF14, Pet14, Pet15, BL21, BK18, BG19, BG18, Ken00, Ken00_GFF, Rus18, Rus20, ARVP_21, BLR_19, BLR_20}) 
that the height function fluctuations converge to a conformally invariant limit. In particular, due to a conjecture of Kenyon-Okounkov~\cite{KO07}, it is expected that on a sequence of growing subgraphs of a 
bi-periodic lattice the 
fluctuations of the height function~$h - \mathbb{E}[h]$ converge to a \emph{Gaussian free field} (GFF) in some conformal structure, a Gaussian field with covariance structure given by the Green's function of a Laplace operator. The GFF associated with a sequence of graphs possesses the following nontrivial features. First, its domain of definition is only a subset of the domain of the height function, known as the \emph{liquid region}, or \emph{rough region}, which is apriori unknown: In the~\emph{frozen facets}, which form the complement of the rough region, fluctuations~$h - \mathbb{E}[h] = 0$ up to events of exponentially decaying probability. Second, the \emph{complex structure} of this GFF, which is used to define the associated Green's function, depends in a nontrivial way on the boundary conditions of the sequence of bipartite graphs. We refer to this complex structure as the \emph{Kenyon-Okounkov complex structure}.


\emph{Perfect t-embeddings} and a corresponding mapping called the \emph{origami map} (which are the subject of this work) were recently introduced 
to give a new description of the complex structure inducing the limiting GFF.
First introduced in~\cite{KLRR22} under the name of 
\emph{Coulomb gauges}, \emph{t-embeddings} are (dual) graph embeddings of edge-weighted planar bipartite graphs which capture the statistics of the associated dimer model and are preserved under the elementary graph transformations. Soon after, discrete holomorphic functions on t-embeddings were defined and studied in~\cite{CLR1}. \emph{Perfect} t-embeddings 
were then introduced and studied in the follow up-work~\cite{CLR2}. 
The main result of~\cite{CLR2} states the following. 
Suppose the sequence of \emph{perfect t-embeddings}~{$\T_n, \, n = 1,2, \dots$} satisfies \textbf{two regularity assumptions} and the graphs of associated origami maps converge to a space-like \textbf{maximal surface} $S$ in the Minkowski space~$\mathbb{R}^{2,1}$.
    Then (gradients of) the associated dimer model height fluctuations converge to (gradients of) the Gaussian free field whose conformal structure is obtained from the metric of the surface~$S \subset \mathbb{R}^{2, 1}$. 
    
    It is conjectured in~\cite{CLR2} that perfect t-embeddings of ``nice'' planar bipartite graphs exist in general, and furthermore that they are unique up to a natural group of symmetries. For graphs with 
 a degree four outer face, existence and uniqueness are known via~\emph{shuffling algorithms}, see~\cite{KLRR22, CR21, BNR23}, and otherwise they remain open questions. In fact, there are no known techniques for constructing or analyzing perfect t-embeddings on graphs with outer face of degree larger than four; this motivates our work. In addition, checking all assumptions of the main theorem of~\cite{CLR2} is a non-trivial task, so far done only for the uniform dimer model on the special 
graph known as the Aztec diamond, see~\cite{BNR23}. 


\begin{figure}
 \begin{center}
\includegraphics[width=0.4\textwidth]{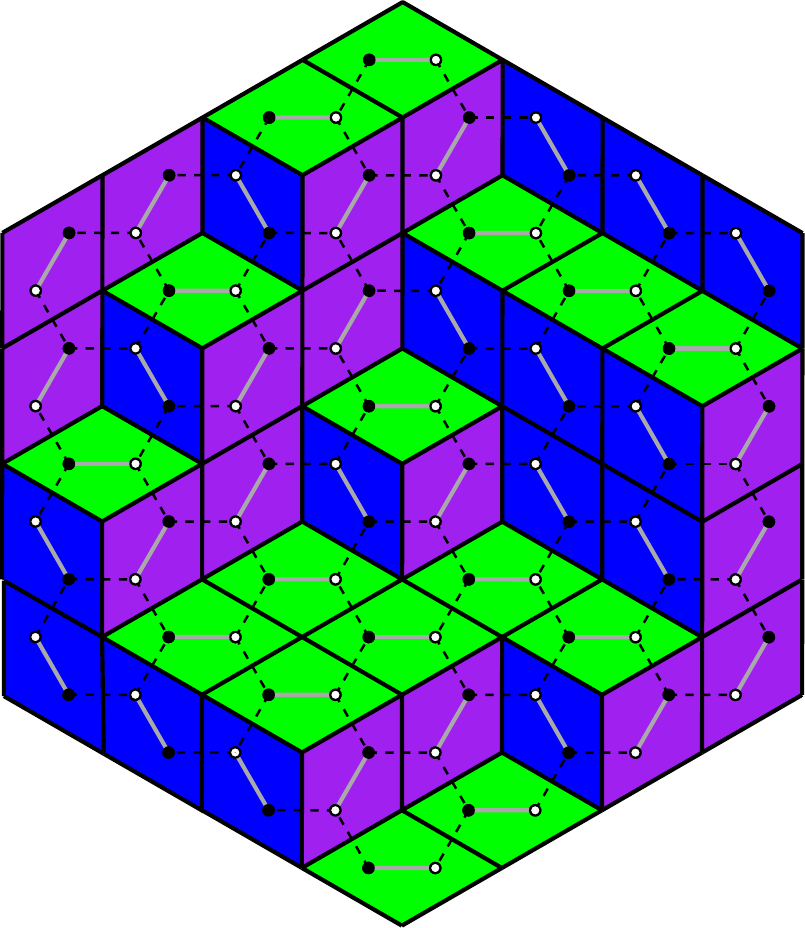}
  \caption{A dimer configuration (shown in grey) on a hexagon-shaped subgraph of the hexagonal lattice together with the corresponding lozenge tiling of the dual graph. $A=4.$}\label{fig:dimers_loz}
 \end{center}
\end{figure}

In this work, we construct 
perfect t-embeddings of homogeneous hexagon-shaped subgraphs of the hexagonal lattice, whose dimer model is equivalent to a uniformly random \emph{lozenge tiling} of an~$A \times A \times A$ regular hexagon cut out of the triangular lattice, for~$A \in \mathbb{Z}_{>0}$ (see Figure~\ref{fig:dimers_loz}). More precisely, we construct perfect t-embeddings of a \emph{reduced hexagon} (see Figure~\ref{fig:hex_reduction}), a graph with an outer face of degree~$6$ which has the same dimer statistics on the interior. This construction provides the first example of a sequence of perfect t-embeddings of graphs where the degree of the outer face is larger than four. The construction itself does not rely on a shuffling algorithm or exact formulas: It only uses the symmetries possessed by the inverse of the \emph{Kasteleyn matrix}, which is nothing but an adjacency matrix in this setup. 

Then, we use known formulas from~\cite{Pet14} for the inverse Kasteleyn matrix to obtain
exact contour integral formulas for these perfect t-embeddings and their origami maps. Analyzing these expressions using the steepest descent method, we obtain scaling limits of the perfect t-embeddings and the origami maps;  in particular, we prove that the graphs of the origami maps converge to a maximal surface in~$\mathbb{R}^{2, 1}$, see Figure~\ref{fig:hex_temb_scale}, and that the perfect t-embedding fulfill the regularity assumptions of \cite{CLR2} by proving the \emph{structural rigidity} conditions introduced in~\cite{BNR23}. In addition, our formulas provide a clear link between the Kenyon-Okounkov conformal structure and the t-embeddings, via the \emph{action function} which appears in the steepest descent analysis of our double contour integral formulas.


\subsection{Main results}
\label{subsec:results}

\begin{figure}
 \begin{center}
\includegraphics[width=0.47\textwidth]{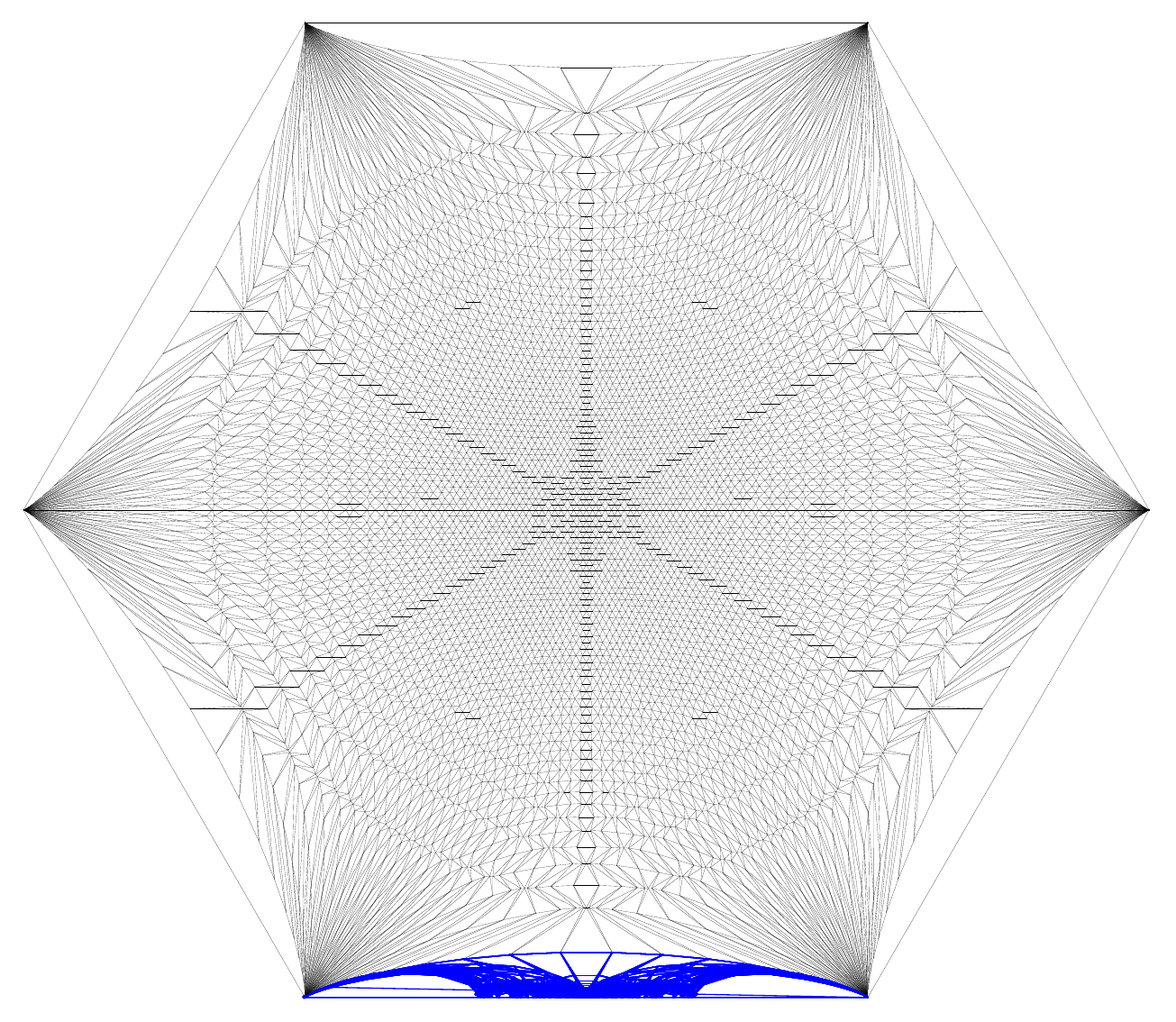}
\includegraphics[width=0.5\textwidth]{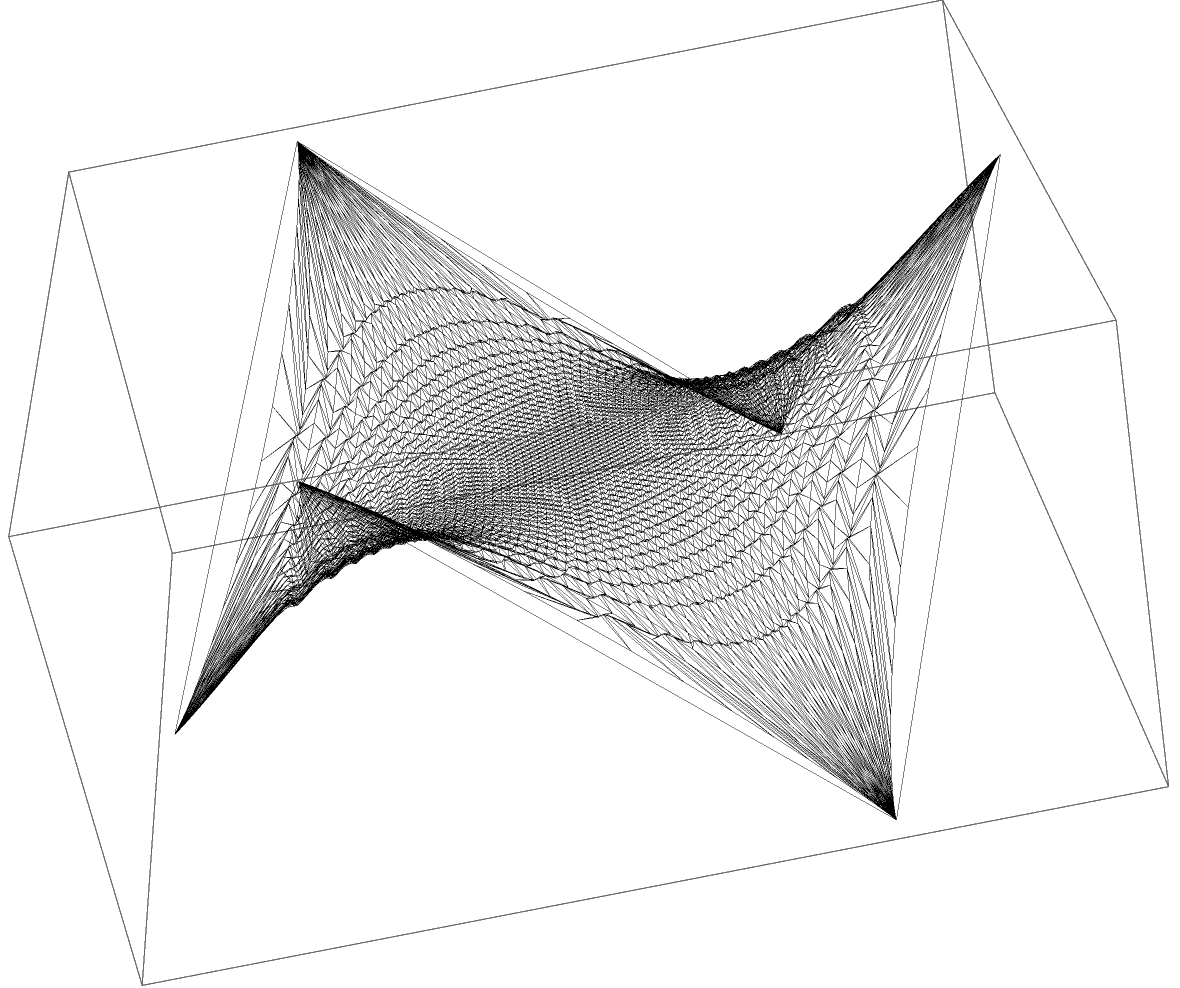}
  \caption{{\bf Left:} Perfect t-embedding (black) and the origami map (blue) of the uniformly weighted hexagon of size~$50$. {\bf Right:} The graph of (the real pat of) the origami map of the uniformly weighted hexagon of size~$50$.}\label{fig:hex_temb_scale}
 \end{center}
\end{figure}

In order to state our results, we will briefly remind the reader of the definitions of perfect t-embeddings and corresponding origami maps, and refer to Section~\ref{sec:bg} for more detailed definitions. A t-embedding~$\mathcal{T}$ of an edge-weighted finite planar bipartite graph~$\mathcal{G}$ is an embedding into~$\mathbb{C}$ of the \emph{augmented dual graph}~$\mathcal{G}^*$ of~$\mathcal{G}$, which is the dual graph of~$\mathcal{G}$ augmented with an extra boundary vertex for each edge of the outer face, see Figure~\ref{fig:temb}. One way to define a t-embedding is via its associated~\emph{Coulomb gauge functions}, as follows.

First, the \emph{Kasteleyn matrix}~$K = K_{\mathcal{G}}$ is a matrix with rows indexed by black vertices and columns by white vertices, with~$K(b,w) = \pm \text{weight}(bw)$ if and only if~$b w$ is an edge in~$\mathcal{G}$, and~$0$ otherwise. Signs are chosen in a particular way to satisfy the \emph{Kasteleyn condition} (described in Section~\ref{sec:bg}); for the reduced hexagon shown in Figure~\ref{fig:hex_reduction}, the signs are~$+$ on solid edges and~$-$ on dashed edges.

Next, \emph{Coulomb gauges} are complex valued functions~$\cF^\tb$ and~$\cF^\tw$ on black and white vertices, respectively, which are in the kernel of~$K^{t}$ and~$K$, respectively, at interior vertices of~$\mathcal{G}$, i.e.~$\sum_{w' \sim b} K(b, w')\cF^\tw(w') = 0$ if~$b$ is not on the outer face, and similarly for~$\cF^\tb$. In general, a pair of Coulomb gauges can be used to define a \emph{t-realisation}~$\mathcal{T}$ (see Section~\ref{sec:bg} for details), which satisfies all properties of a t-embedding except for properness of the embedding, as follows. The image~$d\cT(bw^*)$ of the oriented dual edge~$bw^*$, traversed with~$w$ on the left, is defined by
\begin{equation}\label{eqn:dT_intro}
d\cT(bw^*) \coloneqq \cF^\tb(b)K(b,w)\cF^\tw(w) .
\end{equation}
Then, the associated origami map is a piecewise-smooth complex-valued function on the union of faces of~$\mathcal{T}$; in fact, inside of each face, it is an isometry. It can be restricted to the vertices and edges of~$\mathcal{T}$, and the image of a directed dual edge can be computed as
\begin{equation}\label{eqn:dO_intro}
d\cO(bw^*) \coloneqq \cF^\tb(b)K(b,w)\overline{\cF^\tw(w)}.
\end{equation}
The images of edges of the embedding under~$\mathcal{O}$ is what is depicted in blue in Figure~\ref{fig:hex_temb_scale}, left. Starting from~$\cF^\tw$ and~$\cF^\tb$, if one chooses constants of integration and sums up values of~$d \T$ and~$d \Or$ using~\eqref{eqn:dT_intro} and~\eqref{eqn:dO_intro}, the value of the t-embedding and origami map can be computed at each vertex of~$\mathcal{G}^*$. This is a well-defined procedure since~$\cF^\tw$ and~$\cF^\tb$ are in the kernel of~$K$ and~$K^t$, respectively, in the bulk. A t-embedding is a proper embedding of~$\mathcal{G}^*$ which can be obtained from some pair of Coulomb gauges in this way.

Finally, a \emph{perfect t-embedding} is a t-embedding for which
\begin{enumerate}
\item the boundary polygon is tangential to the unit circle, and
\item  the interior dual edges adjacent to a boundary 
vertex are on angle bisectors.
\end{enumerate}
See Figure~\ref{fig:temb} for an example of a perfect t-embedding.

We now move on to stating our results. We construct and analyze a particular sequence of perfect t-embeddings~$\mathcal{T}_A$ of the~$A \times A \times A$ \emph{reduced hexagon}~$H_A'$ (see Figure~\ref{fig:hex_reduction} and Section~\ref{subsubsec:reduced_hex} for its definition) for positive even integers~$A \in 2 \mathbb{Z}_{> 0}$. The domain covered by~$\T_A$ (the union of its faces) is itself a regular hexagon of sidelength~$1$, which we will denote by~$\Hex$ (see Figure~\ref{fig:hex_temb_scale}, left, or Figure~\ref{fig:hex_temb} for examples). Of course, the boundary polygon~$\partial \Hex$ is not tangential to the unit disc, but rather the disc of radius~$\frac{\sqrt{3}}{2}$; however, this is not important because it is only a question of rescaling to match the definition exactly. We also compute exact formulas for their corresponding origami maps,~$\mathcal{O}_A$, and use these to analyze the origami maps asymptotically.

Our first step is to 
write
the Coulomb gauge functions~$\cF^\tb$ and~$\cF^\tw$ corresponding to a perfect t-embedding~$\mathcal{T}_A$ in terms of certain entries of the inverse Kasteleyn of~$H_A'$.  The signs we use for the Kasteleyn matrix of~$H_A'$ are indicated in Figure~\ref{fig:hex_reduction} and described in detail in Section~\ref{subsubsec:reduced_hex}. We denote the inverse Kasteleyn matrix of the reduced hexagon~$H_A'$ by~$R$. The six boundary vertices~$b_1,b_2,b_3$ and~$w_1,w_2,w_3$ are also indicated in Figure~\ref{fig:hex_reduction}.

\begin{theorem}[Proposition~\ref{prop:T_A} in the text]\label{thm:CG_intro}
Define the Coulomb gauge functions on black and white vertices of~$H_A'$, respectively, by the formulas
\begin{align}\label{eqn:Fform_intro}
    \mathcal{F}^\bullet(b) := \frac{1}{\Delta(A)} \left( \i e^{\i \pi/6}  R(w_2, b) +  R(w_3, b) - \i e^{-\i \pi/6}  R(w_1, b) \right),
    \end{align}
    and 
    \begin{align}\label{eqn:Gform_intro}
    \mathcal{F}^\circ(w) := -R(w, b_1) + e^{\i \pi / 3} R(w, b_2) + e^{-\i \pi / 3} R(w, b_3),
    \end{align}
for an appropriate constant~$\Delta(A) > 0$. Then, these Coulomb gauge functions define a perfect t-embedding~$\mathcal{T}_A$ of~$H_A'$, such that the boundary polygon of~$\mathcal{T}_A((H_A')^*)$ is a unit sidelength regular hexagon.
\end{theorem}

\begin{remark}
The formulas above were obtained from two ingredients: (1) A guess for the boundary polygon of the corresponding perfect t-embedding, and (2) exploiting the symmetry of the hexagon to obtain the gauge functions which satisfy the required properties of a perfect t-embedding with that boundary polygon. See Remark~\ref{rmk:guess_rmk} for further discussion.
\end{remark}

Exact formulas for the entries of the inverse Kasteleyn matrix of the regular hexagon are given in~\cite{Pet14}, and we show that this inverse is closely related to the inverse Kasteleyn matrix~$R$ for the reduced graph~$H_A'$. However, we do not use the exact formulas to prove Theorem~\ref{thm:CG_intro}; we only use the symmetries of the graph~$H_A'$. 

However, we do use the exact formulas together with~\eqref{eqn:Fform_intro} and~\eqref{eqn:Gform_intro} to obtain exact formulas for the t-embeddings and the origami maps, both of which turn out to be amenable to asymptotic analysis. We provide an exact expression giving the embedded position~$\mathcal{T}_A(x, n) \in \mathbb{C}$ of a face~$(x, n)$ of~$H_A'$ (we explain precisely how we coordinatize the faces of~$H_A'$ in Section~\ref{subsubsec:reduced_hex}).

\begin{theorem}[Theorem~\ref{thm:exact_hex} in the text]\label{thm:T_intro}
Let~$\T_A$,~$\Or_A$ denote the perfect t-embedding and the origami map, respectively, of the size~$A$ hexagon, as described above. Then

    \begin{multline}\label{eqn:T_exact_hex_intro}
        \mathcal{T}_A(x, n) =  \mathrm{constant} 
        \\
        +
        \frac{(2 A-1)!}{\Delta(A)}     \frac{1}{(2 \pi \i)^2} \int_{\{0,1,\dots,A-1\}} \int_{\{\infty\}}  \frac{1}{z_1 - z_2} \;\left(
            \frac{(z_1-x+1)_{2 A-n}}{(z_2-x+1)_{2 A-n}}  
        - \frac{(z_1+1)_{2 A-1}}{(z_2+1)_{2 A-1}}  
        \right) 
           \\
        \times    \frac{ (-2A-z_2)_{A} (-z_2)_{A}}{(-2A-z_1)_{A} (-z_1)_{A}} f_A(z_2) g_A(z_1) dz_2 dz_1.
    \end{multline}
    The origami map $\mathcal{O}(x, n)$ evaluated at the corresponding face is given by
    \begin{multline}\label{eqn:O_exact_hex_intro}
        \mathcal{O}_A(x, n) =  \mathrm{constant}
        \\
        +
        \frac{(2 A-1)!}{\Delta(A)}     \frac{1}{(2 \pi \i)^2} \int_{\{0,1,\dots,A-1\}} \int_{\{\infty\}}  \frac{1}{z_1 - z_2} \;\left(
            \frac{(z_1-x+1)_{2 A-n}}{(z_2-x+1)_{2 A-n}}  
        - \frac{(z_1+1)_{2 A- 1}}{(z_2+1)_{2 A-1}}  
        \right) 
           \\
        \times    \frac{ (-2 A-z_2)_{A} (-z_2)_{A}}{(-2 A-z_1)_{A} (-z_1)_{A}} f_A(z_2) \overline{g_A}(z_1) dz_2 dz_1.
    \end{multline}
    In the formulas above, $f_A$ and $g_A$ are holomorphic functions defined by \eqref{eqn:gdef} and \eqref{eqn:fdef}, and $\overline{g_A}(z_1) \coloneqq \overline{g_A(\overline{z_1})}$, and~$\Delta(A)$ is the same constant as in Theorem~\ref{thm:CG_intro}.
\end{theorem}
Up to an exponentially decaying error, the additive constants in the expressions above are independent of~$A$ (and are the constants of integration for~$\mathcal{T}$ and~$\mathcal{O}$). The functions~$f_A$ and~$g_A$ are holomorphic in the upper half plane and approximate (in an appropriate sense) the rational functions~$f$ and~$g$ given in~\eqref{eqn:f_def1} and~\eqref{eqn:g_def1} below. We obtain~\eqref{eqn:T_exact_hex_intro} and~\eqref{eqn:O_exact_hex_intro} by computing integral formulas for~\eqref{eqn:Fform_intro} and~\eqref{eqn:Gform_intro} using the formulas of~\cite{Pet14}, and then summing up increments~\eqref{eqn:dT_intro} and~\eqref{eqn:dO_intro}. The explicit computation is straightforward and algebraic in nature. However, the identification of functions~$f_A$ and~$g_A$ and their role in the asymptotic behavior of~$\cT_A$ and~$\cO_A$ is not trivial and required apriori knowledge of the what the pre and post-limit formulas should look like. In particular, the functions~$f_A$ and~$g_A$ come from prelimit formulas for \eqref{eqn:Fform_intro} and~\eqref{eqn:Gform_intro} stated in Proposition~\ref{prop:contour_int}, and approximate the functions~$f$ and~$g$ defined in~\eqref{eqn:f_def1} and~\eqref{eqn:g_def1} below, but this approximation is not immediately obvious from the prelimit formulas.

Our next theorem is about the scaling limits of t-embeddings and origami maps. We are able to obtain such scaling limits by an asymptotic analysis of the formulas in Theorem~\ref{thm:T_intro}. For the statement, we first define rational functions~$f$ and~$g$ by
\begin{align}\label{eqn:f_def1}
   f(z) &= \frac{2}{3} \left( \i e^{\i \pi/6} \frac{1}{z - 1} - \frac{1}{2}\frac{1}{z - (-\frac{1}{2})}
    -\i e^{-\i \pi/6} \frac{1}{z + 2} \right),\\
    g(z) &= -1 +  \frac{1}{z} e^{\i \pi/3}  -\frac{1}{z+1}  e^{-\i \pi/3}, \label{eqn:g_def1} 
\end{align}
and set~$\overline{g}(z) \coloneqq \overline{g(\overline{z})}$. We also need the \emph{critical point map}~$\zeta(\chi, \eta)$ from Definition~\ref{def:critical_point} in the text. Restricted to the rough region, it turns out that this is nothing but the uniformizing map~$\mathcal{D} \rightarrow \mathbb{H}$ from~\cite{Pet15}, where~$\mathcal{D}$ denotes the liquid region, or rough region, of the hexagon equipped with the Kenyon-Okounkov conformal structure, and~$\mathbb{H} \subset \mathbb{C}$ is the upper half plane. Finally, we use~$\mathfrak{H}$ to denote the continuous domain in~$\mathbb{R}^2$ approximated by the set of vertices in~$\frac{1}{A} H_A$ for large~$A$.

\begin{theorem}[Theorem~\ref{thm:TLIM1} in the text]\label{thm:TLIM_intro}
    Fix a compact subset $K$ in the rescaled hexagon~$\mathfrak{H}$ which is bounded away from the arctic curve.
    Then, uniformly for $(\chi, \eta) \in K$, letting~$\zeta = \zeta(\chi, \eta)$, we have the asymptotic behavior, as $A\to\infty$,
\begin{align}\label{eqn:TLIM_intro}
\mathcal{T}_A(\lfloor \chi A \rfloor, \lfloor \eta A \rfloor) 
= e^{-\i \frac{2 \pi}{3}} -\frac{1}{2\pi\i}\int_{\gamma_{\zeta}}f(z) g(z) \; dz + o(1)
\end{align}
and
\begin{align}\label{eqn:OLIM_intro}
    \mathcal{O}_A(\lfloor \chi A \rfloor, \lfloor \eta A \rfloor) 
    = -\frac{1}{2} - \int_{\gamma_{\zeta}}f(z) \overline{g}(z) \; dz + o(1) .
\end{align}
For~$\zeta \in \mathbb{H}$, the curve~$\gamma_{\zeta}$ above moves from~$\overline{\zeta}$ to~$\zeta$, crossing~$\mathbb{R}$ in~$(-\infty,-2)$. The $o(1)$ error terms can be taken uniformly $O(\frac{1}{\sqrt{A}})$ for~$(\chi, \eta) \in K$.
\end{theorem}

To prove the above theorem we use the steepest descent method. Our steepest descent analysis is very similar to the analyses in~\cite{Pet14, Pet15}.
We have included a complete steepest descent analysis in order to be clear and self contained, and simply because our formulas are slightly different from those analyzed in~\cite{Pet14, Pet15}.

In what follows, we denote by~$\mathbb{R}^{2,1}$ the manifold~$\mathbb{C} \times \mathbb{R}$ equipped with the (non positive-definite) metric~$|d z|^2 - d\vartheta^2$. Define the surface~$S_{\Hex} \subset \mathbb{R}^{2,1}$ as the unique space-like surface of zero mean curvature in~$\mathbb{R}^{2,1}$ with boundary contour given by the union of closed line segments in~$\mathbb{C} \times \mathbb{R}$ connecting the points~$(-e^{\i k  \pi/3}, (-1)^k \frac{1}{2})  $,~$k=1,\dots,6$. Towards the goal of verifying the assumptions of Proposition~\ref{prop:rigidityCLRthm} below, we show the following.

\begin{proposition}[Lemma \ref{lem:diff} in the text]\label{prop:Shexintro}
The limit in~\eqref{eqn:OLIM_intro} is real valued, and as~$\zeta $ ranges over~$\mathbb{H}$, the map 
   \begin{equation}\label{eqn:WP}
   \zeta \mapsto 
\left(e^{-\i \frac{2 \pi}{3}} -\frac{1}{2\pi\i}\int_{\gamma_{\zeta}}f(z) g(z) \; dz , -\frac{1}{2} - \int_{\gamma_{\zeta}}f(z) \overline{g}(z) \; dz \right)
\end{equation}
   is a conformal parameterization of the surface~$S_{\Hex} \subset \C \times \mathbb{R}$.
\end{proposition}

In other words, if we consider~$S_{\Hex}$ as a Riemann surface via its Riemannian metric (recall it is space-like), then~\eqref{eqn:WP} defines a biholomorphic map~$\mathbb{H} \rightarrow S_{\Hex}$. The previous proposition implies that the graphs of the origami maps, when viewed as a function of the complex coordinate on~$\Hex$ (the image of the embedding), converge to~$S_{\Hex}$ (this is the content of Corollary~\ref{cor:orig}). See Figure~\ref{fig:hex_temb_scale} for the graph of~$\Re[\cO_A]$ above~$\cT_A$ for~$A = 50$, which approximates~$S_{\Hex}$.

\begin{remark}\label{rmk:maximal}
  In fact,~\eqref{eqn:WP} is a form of the \emph{Weierstrass parameterization} of a zero mean-curvature surface. This parameterization emerges naturally from the saddle point analysis of~$\mathcal{T}$ and~$\mathcal{O}$ in the proof of Theorem~\ref{thm:TLIM_intro}, and a similar parameterization also emerged naturally in the saddle point analysis of perfect t-embeddings in~\cite{BNR23}. Furthermore, we expect such a parameterization to describe scaling limits for t-embeddings more generally. 
\end{remark}

\begin{remark}
    Theorem~\ref{thm:TLIM_intro} is valid in frozen regions as well as the rough region, and in particular, it implies that each frozen region is collapsed to a boundary vertex of~$S_{\Hex}$ under~$(\mathcal{T}_A,\mathcal{O}_A)$ as~$A \rightarrow \infty$. Indeed, under the map~$\zeta$, frozen regions are mapped to different intervals of the real line minus~$\{-2,-1,-\frac{1}{2}, 0, 1\}$, which each map to a point under the limiting expressions~\eqref{eqn:TLIM_intro} and~\eqref{eqn:OLIM_intro}.
\end{remark}

\begin{remark}\label{rmk:limshape_and_surface}
We emphasize that the maximal surface~$S_{\Hex}$ is \emph{not} the same as the limit shape, even though it shares many similarities: e.g. its boundary conditions resemble those of the limit shape (with the right convention for the height function). In particular, we emphasize that the surface~$S_{\Hex}$ does not have flat facets. However, it would be interesting to investigate the relationship between the two: For example, both the limit shape and the maximal surface~$S_{\Hex}$ are described by a variational problem, and ultimately are described by harmonic functions in the variable~$\zeta$. 
Understanding a more precise connection between these two objects could be valuable.
\end{remark}


Finally, we prove that the \textbf{two regularity assumptions} required for the application of the theorem of~\cite{CLR2} are satisfied in this example. In particular, we show that a stronger pair of conditions, referred to in~\cite{BNR23} as \emph{structural rigidity} conditions, are satisfied in this setup as well, see Section~\ref{subsec:ef_Lip}.

\begin{theorem}[Rigidity Condition]\label{thm:exp_fat_lip}
The embeddings~$\T_A$ are \emph{structurally rigid} in the bulk: For each compact subset~$\mathcal{K} \subset \Hex$, there exists~$C = C_{\mathcal{K} } > 0$, such that for large enough~$A$, the edge lengths satisfy the uniform bound
$$\frac{1}{C} \frac{1}{A}  < |d\T_A(e^*)| < C \frac{1}{A},$$
and the angles of~$\T$ inside~$\mathcal{K} $ are bounded away from~$0$ and~$\pi$.
\end{theorem}
We obtain these facts from an asymptotic expansion of the associated Coulomb gauge functions given in Lemma~\ref{lem:expansions}. In particular, we only utilize the analytic properties of the rational functions~$f$ and~$g$ which appear in the expansion. We expect such an asymptotic expansion to hold in greater generality, provided that perfect t-embeddings exist.

As we explain in the text (see Corollary~\ref{cor:orig}, Proposition~\ref{prop:rigidityCLRthm} and the end of Section~\ref{subsec:ef_Lip}), due to Theorem~\ref{thm:TLIM_intro} and Theorem~\ref{thm:exp_fat_lip}, we may apply~\cite[Theorem 1.4]{CLR2} to obtain the following corollary of our results.

\begin{corollary}\label{cor:fluct_proof}
    The gradient of the centered dimer model height function on the~$A \times A \times A$ hexagon converges to the gradient of the Gaussian free field in the conformal structure of the maximal surface~$S_
    {\Hex} \subset \mathbb{R}^{2,1}$ obtained as the limit of surfaces~$(\T_A, \Or_A)$, via the identification of the rescaled hexagon graph~$\frac{1}{A} H_A$ with~$S_{\Hex}$ described by Theorem~\ref{thm:TLIM_intro}.
\end{corollary}

Convergence of height fluctuations on the~$A \times A \times A$ hexagon to a GFF is a special case of the result of Petrov~\cite{Pet15}. In that work, the conformal structure of the GFF is the Kenyon-Okounkov conformal structure, whereas the corollary above describes the conformal structure via the maximal surface~$S_{\Hex}$. Thus, we expect that these two conformal structures coincide. As we explain at the end of Section~\ref{subsec:limits}, it is not difficult to directly see that the Kenyon-Okounkov conformal structure on the rough region~$\mathcal{D}$ is equivalent to the one on~$\mathcal{D}$ induced by the mapping to the maximal surface~$S_{\Hex}$ given by~\eqref{eqn:TLIM_intro} and~\eqref{eqn:OLIM_intro}. Indeed, the former complex structure is defined by declaring the diffeomorphism~$\zeta : \mathcal{D} \rightarrow \mathbb{H}$ to be a biholomorphic map. The latter is defined by declaring the map from~$ \mathcal{D}$ to~$S_{\Hex}$ defined by~\eqref{eqn:TLIM_intro}  and~\eqref{eqn:OLIM_intro} to be biholomorphic (where~$S_{\Hex}$ is viewed as a Riemann surface via its metric). However, the latter map is nothing but a composition of the former with the conformal parameterization of Proposition~\ref{prop:Shexintro}, which means the two complex structures are the same. Thus we have directly checked that the GFF in Corollary~\ref{cor:fluct_proof} is the same as the object described in~\cite{Pet15}, i.e. these two results are consistent. 

Finally, we would like to point out that our work hints at a general scheme for analyzing perfect t-embeddings, provided that they exist and that sufficiently precise asymptotics for a finite number of columns (and rows) of the inverse Kasteleyn can be obtained. More precisely, suppose that a perfect t-embedding exists. Suppose, in addition, that we have an asymptotic expansion of the gauge functions of the form given in Lemma~\ref{lem:expansions} (let us not discuss the precise expansion needed, but point out that it can be obtained from asymptotics of a finite number of columns/rows of the inverse Kasteleyn, see~\cite[Section 3.4]{KLRR22}).
Finally, suppose that the rational functions~$f$ and~$g$ appearing in the expansions (as in Lemma~\ref{lem:expansions}) have certain restrictions on their poles and zeros which we will not describe in detail here (though in particular all poles should be on the real line and zeros should be in the upper half plane). Then, (1) the convergence of surfaces~$(\mathcal{T}, \mathcal{O})$ to a maximal surface in~$\mathbb{R}^{2,1}$ can be deduced, and (2) Theorem~\ref{thm:exp_fat_lip} can be deduced directly from such expansions, which implies the two regularity assumptions required for an application of~\cite[Theorem 1.4]{CLR2} (as we do to prove Corollary~\ref{cor:fluct_proof} above).

\subsection{Outline}
\label{subsec:outline}
 Section~\ref{sec:bg} contains some background on perfect t-embeddings and origami maps. 
 In Section~\ref{sec:p_emb_of_hex} we define the reduced hexagon and introduce Coulomb gauge functions in terms of the inverse of the initial Kasteleyn matrix. These Coulomb gauge functions provide both perfect t-embeddings~$\T_n$ and corresponding origami maps~$\Or_n$ of the reduced hexagon. 
 Section~\ref{subsec:prelim} is dedicated to deriving exact contour integral formulas for~$\T_n$ and~$\Or_n$. To obtain them we use known double integral formulas for the inverse Kasteleyn matrix and the construction of Coulomb gauge functions introduced in Section~\ref{sec:p_emb_of_hex}. 
 In Section~\ref{sec:scaling_limit} we study the continuous analogues of the functions~$\T_n$ and~$\Or_n$, show that they give a maximal surface~$S$ in the Minkowski space~$\mathbb{R}^{2,1}$, and check that the conformal structure obtained from the metric of this surface coincides with the Kenyon-Okounkov conformal structure.
 In Section~\ref{subsec:ef_Lip} using the exact formulas obtained in Section~\ref{subsec:prelim} we check that perfect t-embeddings~$\T_n$ satisfy the rigidity condition.
 Finally, in Section~\ref{sec:steep_descent} we perform a saddle point analysis that allows us to show the convergence of the graphs of the origami maps to a space-like maximal surface~$S$.

\addtocontents{toc}{\protect\setcounter{tocdepth}{1}}
\subsection*{Acknowledgements}  The authors are grateful to Alexei Borodin for valuable discussions. 
TB~was partially supported by the Knut and Alice Wallenberg Foundation grant KAW~2019.0523 and by A. Borodin's Simons Investigator grant. 
Part of this research was performed while all authors were visiting the Institute for Pure and Applied Mathematics (IPAM), which is supported by the National Science Foundation (Grant No. DMS-1925919).
\addtocontents{toc}{\protect\setcounter{tocdepth}{2}}


\section{Background on perfect t-embeddings and origami maps}
\label{sec:bg}
In this section we introduce some notation and remind definitions and basic facts of (perfect) t-embeddings and their origami maps. For more details we refer an interested reader to~\cite{KLRR22,CLR1,CLR2}. 

Let~$\G$ be a bipartite, finite graph, with~$V(\G)=B\cup W$. Denote the positive weights on edges of~$\G$ by~$\chi_e$. The probability measure on the set~$\mathcal{M}$ of dimer configurations is given by
\[\mathbb{P}[m]=\frac{1}{Z}\prod_{e\in m}\chi_e, \quad\text{ where } Z=\sum_{m\in\mathcal{M}}\prod_{e\in m}\chi_e \text{ is a partition function.}\]
Recall that two weight functions~$\chi_e$ and~$\tilde\chi_e$ on edges of a bipartite graph are gauge equivalent if there exists a pair of functions ~$\mathcal{F}_\mathbb{R}^\bullet: B\to \mathbb{R}_{>0}$ and~$\mathcal{F}_\mathbb{R}^\circ: W\to \mathbb{R}_{>0}$ such that
\[\tilde\chi (bw)=\mathcal{F}_\mathbb{R}^\bullet(b) \chi (bw) \mathcal{F}_\mathbb{R}^\circ(w).\]
Note that gauge equivalent weight functions define the same probability measure on dimer configurations of the graph.
 Let~$K_\R:\R^W\to\R^B$ be a \emph{real-valued} Kasteleyn matrix, i.e., a matrix with entries given by~${K_\R(b,w)=\pm\chi_{(bw)}}$ such that {the `$\pm$'} signs satisfy the Kasteleyn sign condition. Recall that the Kasteleyn sign condition is: for each simple face the alternating product of Kasteleyn signs along this face is equal to $(-1)^{k+1}$, where $2k$ is the degree of this face. Such a choice of signs allow to present the partition function as a determinant of the Kasteleyn matrix: $Z=|\det K_{\mathbb{R}}|.$ Moreover, it is known that all local statistics of the dimer measure can be written in terms of the inverse Kasteleyn matrix.
We refer an interested reader to~\cite{Ken09, Gor21} for a general introduction to the dimer model.

\begin{figure}
 \begin{center}
\includegraphics[width=0.97\textwidth]{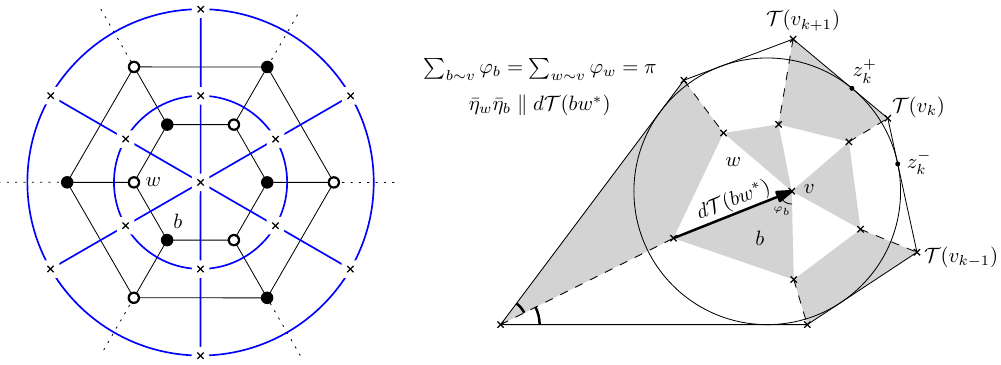}
  \caption{{\bf Left:} The augmented dual graph shown in blue. {\bf Right:} A perfect t-embedding of the augmented dual graph shown on the left. Edges shown in dashed lie on bisectors of the angles of the boundary tangental polygon.}\label{fig:temb}
 \end{center}
\end{figure}

We are interested in embeddings of augmented dual graph. To obtain an augmented dual graph~$\G^*$ from $\G$ one should connect all boundary vertices of~$\G$ to the infinity and take the dual graph, see Figure~\ref{fig:temb}. We call vertices of the set~$B$ (resp.~$W$) black (resp. white) vertices and the corresponding faces of~$\G^*$ black (resp. white) faces. 

\begin{definition}[Following \cite{CLR1,CLR2}]\label{def:temb_origami}
Let~$(\G, \chi)$ be a weighted, finite, bipartite, planar graph. Then
\begin{itemize}
\item[$\bullet$] a proper embedding~$\T$ of augmented dual graph $\G^*$ is called a t-embedding if 
 all edges are line segments, the edge weights of $\G$ given by length of dual edges are gauge equivalent to initial weights~$\chi_e$, and around each inner vertex of~$\T(\G^*)$ the sum of angles adjacent to white faces is the same as the sum of angles adjacent to black ones and both of them are~$\pi$. For each adjacent vertices $w\in W$ and $b\in B$ denote by $d\T(bw^*)$ the complex number $\T (v') - \T (v)$, where $vv'$ is an oriented edge $(bw^*)\in E(\G^*)$ with $b$ to the right, see Figure~\ref{fig:temb}.
 \item[$\bullet$] given t-embedding $\T:\G^*\to \mathbb{C}$ we call a function $\eta: B\cup W \to  \mathbb{T}$ an origami square root function if $d\T(bw^*)\in\bar\eta_b\bar\eta_w\mathbb{R}$. Note that the existence of an origami square root function follows from the angle condition of the t-embedding;
 \item[$\bullet$] an origami map $\Or: \G^*\to \mathbb{C}$ associated with a t-embedding $\T(\G^*)$ is a primitive of the differential form
 \begin{align}\label{eq:def_O}
 d\Or(bw^*):=\eta^2_w d\T(bw^*)=\overline{\eta_b^2d\T(bw^*)}.
 \end{align}
 Note that~\eqref{eq:def_O} defines $\Or$ uniquely up to translations;
 \item[$\bullet$] a t-embedding is a perfect t-embedding if the boundary polygon is tangental (to a unit circle) and all inner edges adjacent to boundary vertices lie on bisectors of corresponding angles of the boundary tangental polygon, see Figure~\ref{fig:temb}.
 \end{itemize}
  \end{definition}

\begin{remark}
For all $b\in B$ and $w\in W$ define $K_\T(b,w):=d\T(bw^*)$ if $b\sim w$ and $K_\T(b,w)=0$ otherwise. Then the angle condition around each vertex of the t-embedding implies that $K_\T$ is a Kasteleyn matrix with complex-valued Kasteleyn signs.
\end{remark}

In~\cite{CLR2} it is shown that if one constructs a sequence of perfect t-embeddings of edge-weighted bipartite graphs and proves that they satisfy certain properties asymptotically, it allows one to prove the conformal invariance of the dimer model in the scaling limit. We restate the main theorem of~\cite[Theorem  1.4]{CLR2} in terms of a stronger set of sufficient conditions, since this set of sufficient conditions are the content of our main results in Sections~\ref{subsec:limits} and~\ref{subsec:ef_Lip}.

In the following we use the notation~$\mathbb{R}^{2,1}$ for the Minkowski space, which is~$\mathbb{C} \times \mathbb{R}$ equipped with the metric~$|dz|^2- d \vartheta^2$. We also view the origami map~$\mathcal{O}$ of a t-embedding as a piecewise-linear complex-valued function, which is defined on the union of the faces of~$\mathcal{T}$; from this point of view,~$\mathcal{O}(z)$ is defined on a subset of the plane~$\mathbb{C}$. Indeed, this can be achieved by noting that the boundary of each face of~$\mathcal{T}$, which is a polygon in~$\mathbb{C}$, is mapped by~$\mathcal{O}$ to a congruent polygon, so that~$\mathcal{O}$ can be extended to the interior of each face as an isometry.

For simplicity we restrict to the setting where we have a sequence of t-embeddings~$\T_n$ which all cover the same domain~$\Omega \subset \mathbb{C}$. Now we are ready to state a version of~\cite[Theorem  1.4]{CLR2}, which the authors used in~\cite{BNR23}.

\begin{proposition}\label{prop:rigidityCLRthm}
   Suppose~$\{\mathcal{T}_n\}_{n \geq 1}$ is a sequence of perfect t-embeddings of edge-weighted bipartite graphs~$\mathcal{G}_n$, such that for each~$n$ the union of faces of~$\mathcal{T}_n$ is the simply connected domain~$\Omega \subset \mathbb{C}$.
   Suppose~$\Omega$ is convex, and in addition that the following conditions hold.
   \begin{itemize}
       \item Graphs of origami maps~$\mathcal{O}_n(z)$ converge to a maximal surface~$S$ in the Minkowski space~$\mathbb{R}^{2,1}$: Uniformly on compact subsets of the interior of~$\Omega$, $\Or_n(z) \rightarrow \vartheta(z)$, and the graph of~$\vartheta : \Omega \rightarrow \mathbb{R}$ is a surface~$S$ with vanishing mean curvature (with respect to the Minkowski metric of~$\mathbb{R}^{2,1}$).
       \item The perfect t-embeddings satisfy the following \emph{structural rigidity} conditions:
       For each compact subset~$K$ of the interior of~$\Omega$, there is a constant~$C = C_K$ such that for each~$ n \geq 1$, the edges of~$\mathcal{T}_n$ inside of~$K$ satisfy 
       $$\frac{1}{C n} \leq |d \T_n(e^*)| \leq \frac{C}{n}$$
       and an~$\epsilon > 0$ such that angles of~$\mathcal{T}_n$ inside of~$K$ are contained in~$[\epsilon, \pi-\epsilon]$.
   \end{itemize}
   Then the gradients of dimer model height fluctuations converge to the gradients of the Gaussian free field in the induced metric on~$S \subset \mathbb{R}^{2,1}$.
\end{proposition}

\begin{proof}
   Corollary 5.13 and Proposition 5.14 of~\cite{BNR23} imply that we can apply~\cite[Theorem  1.4]{CLR2}, from which the proposition follows. 
\end{proof}

Following~\cite{KLRR22} and \cite[Section 4.1]{CLR2} let us call a pair of functions ~$\mathcal{F}^\bullet: B\to \mathbb{C}$ and~$\mathcal{F}^\circ: W\to \mathbb{C}$ \emph{Coulomb gauge functions} if
\begin{equation}
\label{eq:Coulomb-def}
\begin{array}{ll}
[K_\R^\top \cF^\tb](w)=0 &\text{for all~$w\in W\smallsetminus\partial W$,}\\[2pt]
[K_\R\cF^\tw](b)=0 &\text{for all~$b\in B\smallsetminus\partial B$,}
\end{array}
\end{equation}
where $\partial B$ and $\partial W$ are boundary black and white vertices of~$\G$.
Given a pair of Coulomb gauge functions $(\cF^\tb,\,\cF^\tw)$ one can define a \emph{t-realisation}~${\cT=\cT_{(\cF^\tb,\,\cF^\tw)}}$ together with the associated origami map~${\cO=\cO_{(\cF^\tb,\,\cF^\tw)}}$ by setting
\begin{equation}
\label{eq:TO-def-via-F}
\begin{array}{rcl}
d\cT(bw^*)&\!:=\!& \cF^\tb(b)K_\R(b,w)\cF^\tw(w),\\[2pt]
d\cO(bw^*)&\!:=\!&\cF^\tb(b)K_\R(b,w)\overline{\cF^\tw(w)}\,.
\end{array}
\end{equation}
Note that Kasteleyn sign condition implies the angle condition only modulo $2\pi$. Indeed, let~${v\in\G^*}$ correspond to a face $b_1, w_1, b_2, w_2, \ldots, b_k, w_k$ of~$\G$ (labeled counterclockwise), define $ b_{k+1} := b_1$, then
\[\prod_{i=1}^k \frac{d\T(b_iw_i^*)}{d\T(b_{i+1}w_i^*)}
=\prod_{i=1}^k\frac{\cF^\tb(b_i) K_\mathbb{R}(b_iw_i)}{\cF^\tb(b_{i+1}) K_\mathbb{R}(b_{i+1}w_i)}
=\prod_{i=1}^k\frac{K_\mathbb{R}(b_iw_i)}{K_\mathbb{R}(b_{i+1}w_i)}
=(-1)^{k+1} X_v, 
\quad \text{ with  } X_v>0.\]
Note that $\frac{d\T(b_iw_i^*)}{d\T(b_{i+1}w_i^*)}=-e^{i\varphi_{w_i}}r$, where $\varphi_{w_i}$ is an angle of the white face $w_i$ adjacent to the vertex~$v$ and $r\in\mathbb{R}_+$. Hence we obtain
\[\prod_{i=1}^k \frac{d\T(b_iw_i^*)}{d\T(b_{i+1}w_i^*)} = (-1)^ke^{i (\sum_{w\sim v} \varphi_w)}X_v.\]
This implies that 
$\sum_{w\sim v} \varphi_w=\pi \mod 2\pi.$
 Therefore a priori ${\cT_{(\cF^\tb,\,\cF^\tw)}}$ might have overlaps. However if a t-realisation is a proper embedding, then it is a t-embedding.

\begin{remark} In order to make Definition~\ref{def:temb_origami} in agreement with \eqref{eq:TO-def-via-F} one should choose the origami root function such that
\[
 \cF^\tb(b)\in \bar\eta_b\mathbb{R} \quad \text{and}\quad
\cF^\tw(w)\in \bar\eta_w\mathbb{R}\, ,
\]
for all black vertices~$b$ and white vertices~$w$ of~$\G$.
\end{remark}

\begin{remark}\label{rmk:lambda} Note that  $(\cF^\tb,\,\cF^\tw) \to (\lambda\cF^\tb,\, \lambda^{-1}\cF^\tw)$, with~$\lambda\in\mathbb{T}$ does not change the t-realisation and rotates the corresponding origami map:
\[ \cT_{(\cF^\tb,\,\cF^\tw)}=\cT_{(\lambda\cF^\tb,\,\lambda^{-1}\cF^\tw)}
\quad \text{ and }\quad
 \cO_{(\cF^\tb,\,\cF^\tw)}=\lambda^2\cO_{(\lambda\cF^\tb,\,\lambda^{-1}\cF^\tw)}.
\]
\end{remark}

It is known, see~\cite{CLR2}, that if a pair of Coulomb gauge functions satisfy the boundary conditions of a perfect t-embedding then it defines a proper embedding and therefore it defines a perfect t-embedding. The following theorem is equivalent to~\cite[Theorem 4.1]{CLR2}.

\begin{theorem}\label{thm:proper_emb}
Denote by~$v_1, \ldots , v_{2n}$ the outer vertices of~$\G^*$ labeled counterclockwise. We also denote by~$v_{\operatorname{in},k}$ the unique inner vertex of~$\G^*$ that is adjacent to~$v_k$. 
A pair of Coulomb gauge functions~$(\cF^\tb,\,\cF^\tw)$ defines a perfect t-embedding if and only if t-realization~${\cT=\cT_{(\cF^\tb,\,\cF^\tw)}}$ satisfies the following properties:
\begin{itemize}
\item[(a)]  The boundary vertices $\cT(v_1), \ldots, \cT(v_{2n})$ form a tangental (not necessary convex) polygon~$P$ with a radius~$1$ inscribed circle;
\item[(b)] The points~$\cT(v_{\operatorname{in},k})$ lie inside the polygon~$P$ on bisectors of corresponding angles of~$P$. 
\end{itemize}
\end{theorem}

\begin{proof} Note that properties (a)--(b) describe boundary conditions of a perfect t-embedding, therefore each perfect t-embedding~$\cT$ satisfies these properties. It remains to check that each t-realization satisfying (a) and (b) is a perfect t-embedding.

Let us show that the statement of the theorem is equivalent to~\cite[Theorem 4.1]{CLR2}. In other words, we want to check that for each t-realization satisfying properties (a)--(b) one can choose~$\lambda$ from Remark~\ref{rmk:lambda} and the additive constants of integration in the definition~(\refeq{eq:TO-def-via-F}) such that the properties~(i)--(iii) from~\cite[Theorem 4.1]{CLR2} hold.

Note that~$\cT$ is defined up to translations, so we can assume that~$P$ is tangental to the unit circle with the center at the origin. In particular, by this translation we obtain that the index of the polyline $\cT(v_1) \ldots  \cT(v_{2n})\cT(v_1)$  with respect to the origin is~$1$, since~$P$ is a polygon containing the origin. Therefore~$\cT_{(\cF^\tb,\,\cF^\tw)}$ satisfies property~(iii) of~\cite[Theorem 4.1]{CLR2}.

Property (b) implies that the origami map of the boundary polygon~$P$ lies on a line, so we can rotate the origami map (by choosing a proper~$\lambda$ in Remark~\ref{rmk:lambda}) such that all boundary points ~{$\cO(v_1), \ldots , \cO(v_{2n})\in\mathbb{R}$}. We claim that one can choose the translation of the origami map such that for all $k\in\{1,\ldots,2n\}$ the point $(\cT(v_k), \cO(v_k))$ lies on the hyperboloid~${\{(z,\vartheta)\in \mathbb{C}\times\mathbb{R}\simeq\mathbb{R}^{2,1}: |z|^2-|\vartheta|^2=1\}}$. Indeed, note that all tangent points on~$P$ maps to the same point under the origami map, by applying the translation we can assume that it maps to the origin. Then $|\cO(v_k)|=|\cT(v_k)-z_k^{\pm}|$, where~$z_k^{\pm}$ is a tangency point of the unit circle with the line $\cT(v_{k\pm1})\cT(v_{k})$. Now, we just use that $|\cT(v_k)|$ is the distance between the center of the inscribed circle and the vertex of the tangental polygon,  $|\cO(v_k)|$ is the length of the corresponding tangent and the radius of the inscribed circle is~$1$ to see that 
\[|\cT(v_k)|^2-|\cO(v_k)|^2=1.\]
Hence~$\cT_{(\cF^\tb,\,\cF^\tw)}$ satisfies property~(i) of~\cite[Theorem 4.1]{CLR2}.

Finally, the inequalities 
\begin{align*}
&\operatorname{Im}[(\cT(v_{k+1})-\cT(v_{k}))\overline{\cT(v_k) -\cT(v_{\operatorname{in},k})}]>0\\
-&\operatorname{Im}[(\cT(v_{k-1})-\cT(v_{k}))\overline{\cT(v_k) -\cT(v_{\operatorname{in},k})}]>0
\end{align*}
are equivalent to the property that the points~$\cT(v_{\operatorname{in},k})$ lie inside the polygon~$P$. And therefore~$\cT_{(\cF^\tb,\,\cF^\tw)}$ satisfies property~(ii) of~\cite[Theorem 4.1]{CLR2} as well.
\end{proof}


\section{Perfect t-embeddings of uniformly weighted hexagon}
\label{sec:p_emb_of_hex}

The goal of this section is to find a Coulomb gauge defining a perfect t-embedding of uniformly weighted hexagon in terms of the initial real-valued inverse Kasteleyn matrix~$K^{-1}_\mathbb{R}$. We assume that the side length of the hexagon, which we call $A$ throughout the paper, is even, which will make some choices simpler and more symmetric. The odd-length case can be done in a similar way.
\begin{figure}
 \begin{center}
\includegraphics[width=0.45\textwidth]{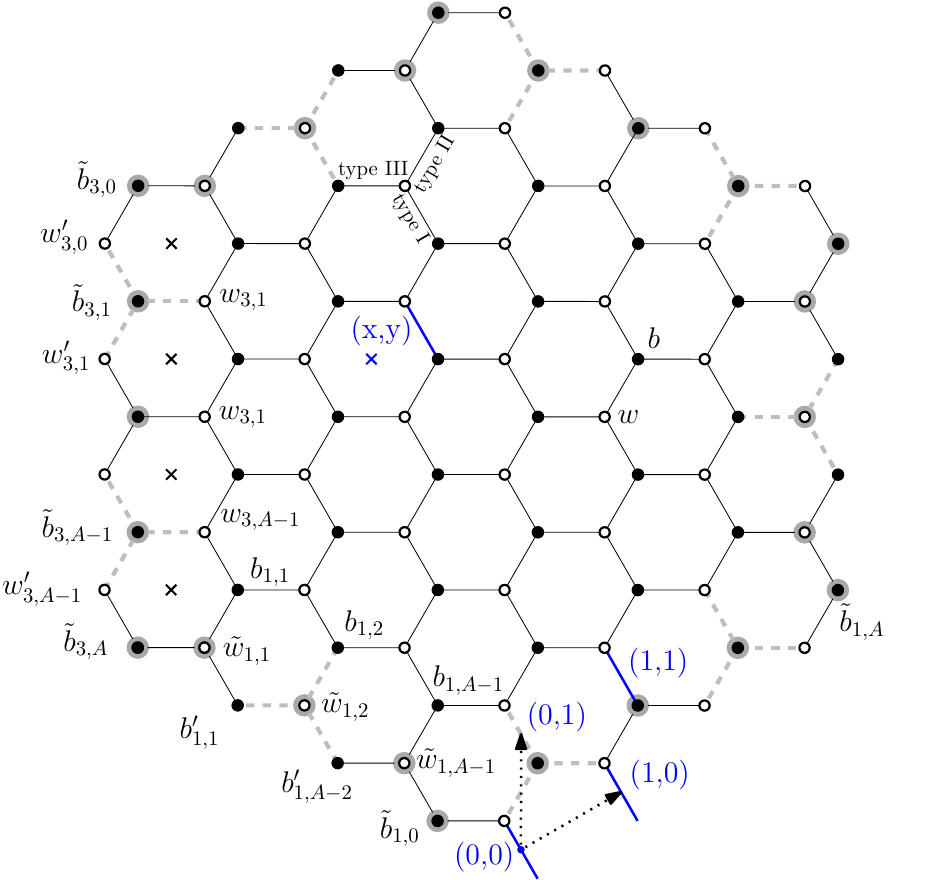}
$\quad$
\includegraphics[width=0.49\textwidth]{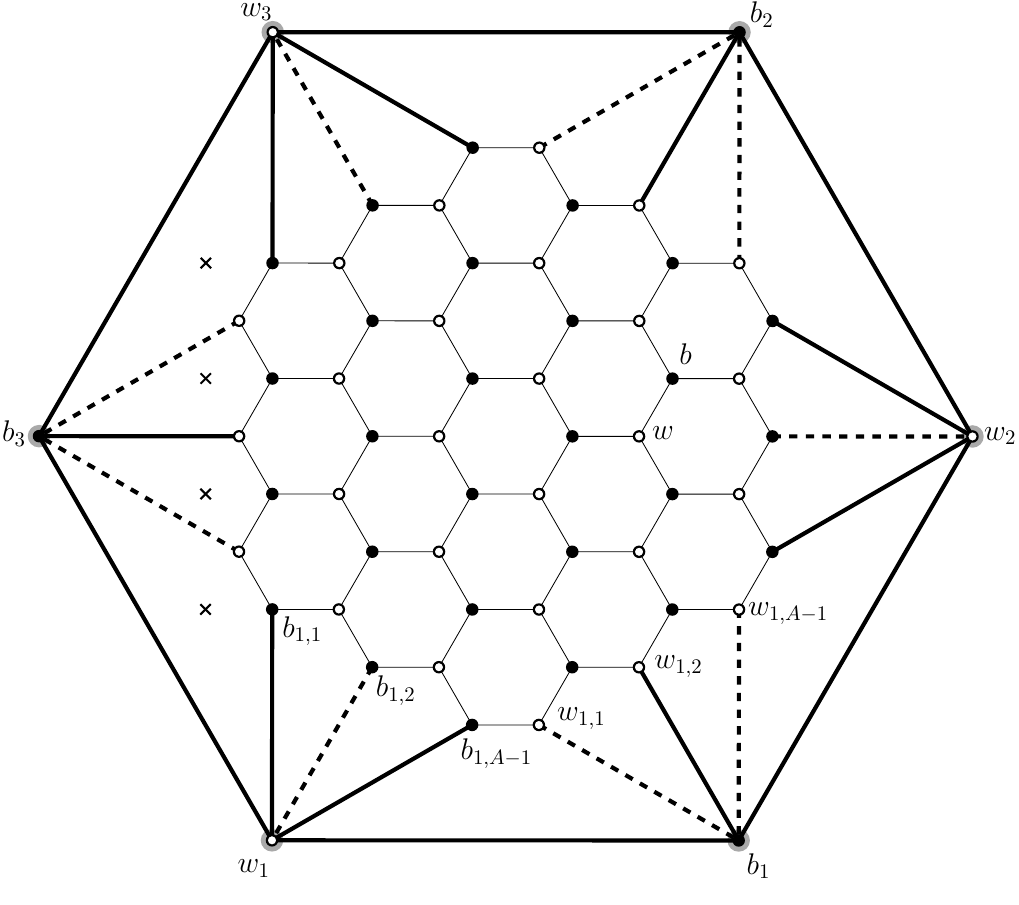}
  \caption{{\bf Left:} A hexagon~$H_4$ of size $4$ with marked (in gray) boundary vertices contracted in the reduction process. After a gage transformation described in the reduction process all dashed edges have Kasteleyn sign $-1$, all other edges have sign~$1$. Coordinates on the set of edges are shown in blue. Note that the two bottom blue edges with the coordinates~$(0,0)$ and~$(1,0)$ are not a part of the hexagon. {\bf Right:} The reduced hexagon~$H'_4$, obtained by the reduction process from $H_4$. Kasteleyn signs on dashed edges are~$-1$, on all other edges are~$1$.
  }\label{fig:hex_reduction}
 \end{center}
\end{figure}

\subsection{Reduced hexagon}
\label{subsubsec:reduced_hex}
Similar as for uniformly weighted Aztec diamond~\cite{CR21, BNR23} we want to construct a perfect t-embedding of a certain \emph{reduced hexagon} which looks like the graph in the right of Figure~\ref{fig:hex_reduction}. In this section we describe the reduction procedure and study the Kasteleyn matrix of the reduced graph.

Let us introduce a coordinate system on the honeycomb lattice. Note that there are three types of edges on the hexagonal lattice given the direction, we call them type~I, type~II and type~III as shown on Figure~\ref{fig:hex_reduction} (i.e. all horizontal edges are of type~III). In fact for our purposes it is enough to specify the coordinates on the set of edges of type~I. In order to do this we introduce a coordinate system on the midpoints of edges of type~I of the embedding of the hexagonal lattice as shown on Figure~\ref{fig:hex_reduction}. We identify the coordinates $e(x,y)$ of each edge of type~I with the coordinates 
$(x,y)$ of the midpoint of the corresponding edge~$e$.  Now we introduce coordinates on the set of vertices and faces of the $A \times A \times A$  hexagon~$H_{A}$ in the following way. Note that each vertex of the hexagonal lattice is adjacent to a unique edge of type~I, so we identify the coordinates of each vertex with the coordinates of the corresponding adjacent edge of type~I. We will sometimes denote by~$b(x,y)$ the black vertex at edge~$e(x, y)$, and by~$w(x, y)$ the white vertex at edge~$e(x,y)$. Note also that each edge of type~I has exactly two adjacent faces, we identify the coordinates of the face on the left with the coordinates of the edge. To fix the origin, we say that the bottom boundary white vertex of each hexagon~$H_{A}$ has coordinates~$(0,0)$, see Figure~\ref{fig:hex_reduction}.



 Let~$K$ be the Kasteleyn matrix of~$H_A$, where $K(b,w)=1$ for $b\sim w$ and $K(b,w)=0$ otherwise. To obtain the reduced hexagon~$H'_{A}$ from the hexagon~$H_{A}$ one should make the following sequence of moves:
\begin{itemize}
\item[$\bullet$] apply a gauge transform to modify (only) Kasteleyn signs on (some) edges adjacent to boundary vertices. More precisely, multiply by $-1$ all weights on edges adjacent to boundary 
black vertices of~$H_{A}$ with coordinates of the form $(2k,1)$, $(2k-1,2A-2k)$, $(-A,2k+1)$ with $k$ integer;
and to boundary degree-three white vertices of~$H_{A}$ with coordinates of the form $(A-1,2k)$, $(-2k,2A-1)$, $(-2k,2k)$ with $k$ integer; 
\item[$\bullet$] contract black vertices $(k,1)$, $(k,2A-1-k)$, $(-A,k)$ of~$H_{A}$ with~$k$ integer;
\item[$\bullet$] contract degree three white vertices $(A-1,k)$, $(k,2A-1)$, $(-k,k)$ of~$H_{A}$ with~$k$ integer;  
\end{itemize}
\begin{remark}
 Note that reduced hexagon has degree four faces at the boundary. 
The first step in the reduction is needed to get valid Kasteleyn signs on boundary faces of $H'_{A}$.
\end{remark}

Note that inner vertices of augmented dual $(H'_{A})^*$ are in natural correspondence with inner faces of $H_{A}$ and therefore can be indexed by $(x,y)\in\mathbb{Z}^2$ in the same way. 

Let us introduce some notation on the set of boundary vertices and vertices adjacent to boundary ones. Denote the black vertices along the bottom right boundary of the hexagon~$H_{A}$ by $\widetilde{b}_{1,0},\, \widetilde{b}_{1,1}, \ldots, \widetilde{b}_{1,A}$, note that by the reduction process we contract all these  vertices to one black vertex, denote the corresponding black boundary vertex of~$H'_{A}$ by $b_1$, see Figure~\ref{fig:hex_reduction}. 
We emphasize that the labeling does not follow the coordinates.
Denote white degree-two vertices along the same part of the boundary by $w'_{1,0}, \, w'_{1,1}, \ldots, w'_{1,A-1}$. For each $k\in\{1, \ldots, A-1\}$ denote the inner white vertex adjacent to $\widetilde{b}_{1,k}$ by $w_{1,k}$. Let us now introduce the notation on vertices along bottom left boundary. Denote degree-three white vertices along the left bottom boundary by $\widetilde{w}_{1,1},\, \widetilde{w}_{1,2}, \ldots, \widetilde{w}_{1,A-1}$ as shown on Figure~\ref{fig:hex_reduction}. Note that this is exactly the set of vertices contracted in the reduction process, so we denote the corresponding white vertex of~$H'_{A}$ by~$w_1$. For each $k\in\{1, \ldots, A-1\}$ denote the inner black vertex adjacent to $\widetilde{w}_{1,k}$ by~$b_{1,k}$. Finally, denote all boundary non-corner black vertices along this part of the boundary by $b'_{1,1}, \, b'_{1,2}, \ldots, b'_{1,A-2}$. Similarly, for $j=2,3$ we introduce vertices $w_j,b_j\in H'_{A}$ and sets of contracted vertices ${\widetilde{w}_{j,k}, \widetilde{b}_{j,k}\in H_{A}}$, degree-two boundary vertices $w'_{j,k}, b'_{j,k}\in H_{A}$ and inner vertices $w_{j,k}, b_{j,k}\in H_{A}$ adjacent to the boundary ones. To simplify the notation, we denote inner vertices of~$H'_{A}$ adjacent to the boundary ones in the same way as the corresponding vertices denoted in $H_{A}$, see Figure~\ref{fig:hex_reduction}.

Let $K_{\operatorname{reduced}}$ be the Kasteleyn matrix of the reduced hexagon~$H'_A$ obtained by the reduction process from~$K$. Denote by~$R(w, b)$ entries of the inverse Kasteleyn matrix of the reduced hexagon~$H'_{A}$. Then the entries of the matrix~$R$ coincide up to a sign with the corresponding entries~$K^{-1}(w,b)$ of the inverse Kasteleyn matrix of the original graph~$H_{A}$.

\begin{lemma}\label{lem:reduced_K}
Let~$w_r, b_r \in V(H'_{A})$ be white and black vertices of the reduced graph. We identify~$w_r$ with a white vertex $w$ of the original graph $H_{A}$ in the following way: If~$w_r$ is in the interior, it is identified with a vertex of the original graph in the natural way, and if~${w_r}$ is a contracted vertex, we identify it with any vertex at the endpoint of the corresponding string of contracted vertices in the original graph.  Similarly we identify each black vertex~$b_r$ of the reduced graph with a black vertex~$b$ in the original graph. Then 
\[R({w_r}, {b_r}) = \pm K^{-1}(w, b).\]
More precisely, if none of $w_r, b_r \in V(H'_{A})$ are boundary vertices, then 
\[R({w_r}, {b_r}) = K^{-1}(w, b).\]
If both are boundary vertices, then for $j,k\in\{1,2,3\}$
\[R({w_j}, {b_k}) = K^{-1}(\widetilde{w}_{j,1}, \widetilde{b}_{k,0}).\]
If $w_j$ is a boundary vertex and $b$ is an inner vertex, then for $j\in\{1,2,3\}$
\[R({w_j}, b) = (-1)^{k+1}K^{-1}(\widetilde{w}_{j,k}, b)\]
and if $w$ is an inner vertex and $b_r$ is a boundary vertex, then for $r\in\{1,2,3\}$
\[R(w, {b_r}) = (-1)^{j}K^{-1}(w, \widetilde{b}_{r, j}).\]

\end{lemma}

\begin{remark}
For example, in the above lemma, we identify the vertex~$w_1$ of~$H'_{A}$ with any of~$A-1$ white vertices~$\widetilde{w}_{1,1}, \ldots, \widetilde{w}_{1,A-1}$; and we identify the vertex~$b_1$ of~$H'_{A}$ with any of~$A+1$ black vertices~$\widetilde{b}_{1,0}, \ldots, \widetilde{b}_{1,A}$, see Figure~\ref{fig:hex_reduction}.
\end{remark}

\begin{proof}
First suppose~$b_r$ is in the interior, then it corresponds to the unique vertex~$b$ of the original graph. The claim follows from the following observation: If~$b$ is fixed vertex of~$H_{A}$, and we vary~$w$ along a boundary corresponding to a contracted white vertices, then the signs of~$K^{-1}(w, b)$ alternate.  
More precisely, define $\widetilde{w}_{1,0}:=w'_{3,A-1}$ and $\widetilde{w}_{1,A}:=w'_{1,0}$, then for all~$k\in \{0, \ldots, A\}$ we have
\begin{equation}\label{eqn:kast_1}
K^{-1}(\widetilde{w}_{1,1}, b) = (-1)^{k+1} K^{-1}(\widetilde{w}_{1,k}, b).
\end{equation}
Note that this follows from the fact that~$ K K^{-1}(b, b') = 0$ for each black vertex~$b'$ along the bottom left boundary, i.e. $b' \in \{\widetilde{b}_{3,A}, b'_{1,1},\dots, b'_{1,A-2}, \widetilde{b}_{1,0}\}$.

Denote the Kasteleyn matrix of~$H_{A}$ obtained by the gauge described in the reduction process by~$K_{\operatorname{gauge}}$. Then for~${b\not\in \{b'_{1,1}, \ldots, b'_{1,A-2}\}}$ and~$k\in \{1, \ldots, A-1\}$ we have
\begin{equation}\label{eqn:kast_2}
K_{\operatorname{gauge}}^{-1}(\widetilde{w}_{1,1}, b) = K_{\operatorname{gauge}}^{-1}(\widetilde{w}_{1,k}, b) =  K^{-1}(\widetilde{w}_{1,1}, b), 
\end{equation}
and for~${w\not\in \{w'_{1,0}, \ldots, w'_{1,A-1}\}}$ and~$k\in \{0, \ldots, A\}$ we have
\begin{equation}\label{eqn:kast_2'}
K_{\operatorname{gauge}}^{-1}(w,\widetilde{b}_{1,0}) 
= K_{\operatorname{gauge}}^{-1}(w,\widetilde{b}_{1,k}) 
= K^{-1}(w,\widetilde{b}_{1,0}).
\end{equation}

Note also, that for each inner or boundary degree-two white vertex $w$, we have 
\begin{equation}\label{eqn:kast_3}
K_{\operatorname{gauge}}^{-1}({w}, b) = K^{-1}(w, b).
\end{equation}

Using the identification described in the statement of the lemma, for each pair of white and black vertices~$w_r, b_r \in H'_{A}$ define matrix $R_{\operatorname{reduced}}$ by
\[R_{\operatorname{reduced}}({w_r}, {b_r}) := K_{\operatorname{gauge}}^{-1}(w, b).\]
Due~(\refeq{eqn:kast_2}) and vertices identification process the matrix~$R_{\operatorname{reduced}}$ is well-defined. Since the entries of~$K_{\operatorname{gauge}}$ lead to the gauge we use for the reduced hexagon's Kasteleyn matrix~$K_{\operatorname{reduced}}$, we obtain
\begin{equation}\label{eqn:kast_4} 
R_{\operatorname{reduced}} K_{\operatorname{reduced}} (w_r , w'_r) = \delta_{w_r w'_r}.
\end{equation}
Hence the matrix $R_{\operatorname{reduced}}$ is simply the inverse Kasteleyn matrix of the reduced hexagon~$H'_{A}$, i.e. coincide with the matrix $R$ from the statement of the lemma.

To finish the proof let us show that~(\ref{eqn:kast_4}) indeed holds. For any non-boundary point $w'_r$ we have
\[
\sum_{b\sim w'_r} K_{\operatorname{reduced}}(b, w'_r) R_{\operatorname{reduced}}(w_r, b)=
\sum_{b\sim w'_r} K_{\operatorname{gauge}}(b, w'_r) K_{\operatorname{gauge}}^{-1}(w_r, b)=
\delta_{w_r w'_r}.
\] 
Now without loss of generality assume that $w'_r=w_1$. Define $b_{1,0}, b_{1,A} \in H'_{A}$ by $b_{1,0}:=b_3$ and $b_{1,A}:=b_1$, then we have 
\[
\sum_{b\sim w'_r} K_{\operatorname{reduced}}(b, w'_r) R_{\operatorname{reduced}}(w_r, b)=
\sum_{k=0}^{A} K_{\operatorname{reduced}}(b_{1,k}, w_1) R_{\operatorname{reduced}}(w_r, b_{1,k}).
\] 
This sum can be rewritten as 
\begin{align*}
K_{\operatorname{gauge}}(\widetilde{b}_{3,A}, \widetilde{w}_{1,1}) 
K^{-1}_{\operatorname{gauge}}(w, \widetilde{b}_{3,A})&+\\
 \sum_{k=1}^{A-1} K_{\operatorname{gauge}}(b_{1,k}, \widetilde{w}_{1,k}) 
&K^{-1}_{\operatorname{gauge}}(w, b_{1,k}) +\\
& K_{\operatorname{gauge}}(\widetilde{b}_{1,0}, \widetilde{w}_{1,A-1}) 
K^{-1}_{\operatorname{gauge}}(w, \widetilde{b}_{1,0}).
\end{align*}
Note that 
\[
 \sum_{k=1}^{A-2} 
 \underbrace{
 \left(K_{\operatorname{gauge}}(b'_{1,k}, \widetilde{w}_{1,k}) + 
 K_{\operatorname{gauge}}(b'_{1,k}, \widetilde{w}_{1,k+1}) \right)}_{=(-1+1)=0}
 K^{-1}_{\operatorname{gauge}}(w, b'_{1,k})=0,
\]
therefore 
\begin{align*}
\sum_{k=0}^{A} K_{\operatorname{reduced}}(b_{1,k}, w_1)& R_{\operatorname{reduced}}(w_r, b_{1,k})=\\
&\sum_{k=1}^{A-1}
 \left(
 \sum_{b\sim \widetilde{w}_{1,k}} 
 K_{\operatorname{gauge}}(b, \widetilde{w}_{1,k}) 
K^{-1}_{\operatorname{gauge}}(w, b)
 \right)=\sum_{k=1}^{A-1} \delta_{w\widetilde{w}_{1,k}}=\delta_{w_rw_1}.
\end{align*}

If~$b_r$ is on the boundary, so~$b_r$ itself is a contracted vertex, then a similar argument works.
\end{proof}

\subsection{Perfect t-embedding of reduced hexagon}
\label{subsubsec:temb_hex} In this section we introduce Coulomb gauge functions~$\mathcal{F}^\bullet: B\to \mathbb{C}$ and~$\mathcal{F}^\circ: W\to \mathbb{C}$ on reduced hexagon~$H'_A$ and show that the corresponding t-realization $\T=\T_{(\mathcal{F}^\bullet, \mathcal{F}^\circ)}$ is a perfect t-embedding of the augmented dual $(H'_A)^*.$
Recall that we suppose the combinatorial side length of the hexagon to be an even number. Similar formulas work for an odd side length but the signs will change in formulas for Coulomb gauge functions~$\mathcal{F}^\bullet$ and~$\mathcal{F}^\circ$.

Let $R(w, b)$ be entries of the inverse Kasteleyn matrix of the reduced hexagon~$H'_A$, as in the previous section. Define function~$\widetilde{\mathcal{F}}^\bullet: B\to \mathbb{C}$ by
\begin{align}\label{eqn:Fform}
\widetilde{\mathcal{F}}^\bullet(b) :=  e^{\i 2\pi/3}  R(w_2, b) +  R(w_3, b) + e^{-\i 2\pi/3}  R(w_1, b),
\end{align}
and function $\mathcal{F}^\circ: W\to \mathbb{C}$ by
\begin{align}\label{eqn:Gform}
\mathcal{F}^\circ(w) := -R(w, b_1) + e^{\i \pi / 3} R(w, b_2) + e^{-\i \pi / 3} R(w, b_3).
\end{align}
By definition of $R(w,b)$, the functions~$\widetilde{\mathcal{F}}^\bullet$ and~$\mathcal{F}^\circ$  satisfy~(\ref{eq:Coulomb-def}) for $K_\R=K_{\operatorname{reduced}}$, therefore the pair $(\widetilde{\mathcal{F}}^\bullet, \mathcal{F}^\circ)$ defines a Coulomb gauge and we can consider corresponding t-realisation~${\cT=\cT_{(\widetilde{\mathcal{F}}^\bullet,\,\cF^\tw)}}$ and origami map~${\cO=\cO_{(\widetilde{\mathcal{F}}^\bullet,\,\cF^\tw)}}$.

Note that by symmetry of Kasteleyn weights on~$H'_A$ we have
\begin{align}
R(w_1,b_1)=&R(w_2,b_1)=R(w_2,b_2)=R(w_3,b_2)=R(w_3,b_3)=R(w_1,b_3)=\alpha, \label{eqn:alpha_eq}\\
&R(w_1,b_2)=R(w_3,b_1)=R(w_2,b_3)=\beta, \label{eqn:beta}
\end{align}
for some real~$\alpha$ and~$\beta.$ The explicit computation of the values $\alpha$ and $\beta$ are given in Remark~\ref{b_ne_a}.


Using~\eqref{eqn:Fform}--\eqref{eqn:beta}, one can check that 
\[|\widetilde{\mathcal{F}}^\bullet(b_1)|=|\widetilde{\mathcal{F}}^\bullet(b_2)|=|\widetilde{\mathcal{F}}^\bullet(b_3)|=
|\mathcal{F}^\circ(w_1)|=|\mathcal{F}^\circ(w_2)|=|\mathcal{F}^\circ(w_3)|=|\beta-\alpha|.\]

\begin{lemma}\label{lem:alpha_beta}
Let $\alpha$ and $\beta$ be defined by~\eqref{eqn:alpha_eq} and~\eqref{eqn:beta}. Then $\alpha$ is strictly positive and $\beta$ is strictly negative.
In particular, 
\[ \Delta(A):=|\beta-\alpha| =\alpha - \beta > 0 \]
\end{lemma}

\begin{proof}
Note that there exists a dimer cover of~$H'_A$ containing the edge~$(w_1b_1)$, therefore 
\[\mathbb{P}[w_1b_1]>0.\]  
The positivity of $\alpha$ now follows from the fact that all local statistics can be written in terms of the inverse Kasteleyn matrix, more precisely, we have
\[\mathbb{P}[w_1b_1]=K_{\operatorname{reduced}}(b_{1}, w_1)\cdot R(w_1, b_1)=1\cdot\alpha=\alpha.\]

Similarly, let us show that $\beta$ is strictly negative. Note that $H'_A$ with vertices~$w_1$ and~$b_2$ removed can be covered by dimers. Therefore
\[|\beta|= \frac{Z_{H'_A\smallsetminus \{w_1,b_2\}}}{Z_{H'_A}}> 0,\] 
where $Z_{H'_A\smallsetminus \{w_1,b_2\}}$ and $Z_{H'_A}$ are the partition functions of the graphs~$H'_A\smallsetminus \{w_1,b_2\}$ and~$H'_A$ correspondingly. We can also track a sign of $\beta$ by adding the edge $(w_1b_2)$ to the graph $H'_A$ with $K(b_2,w_1):=-1$. Note that the Kasteleyn weight on $(w_1b_2)$ has to be negative to satisfy Kasteleyn sign condition, since $b_2, w_1, b_3, w_3$ form a single face of degree~$4$ and $K(b_2,w_3)=K(b_3,w_3)=K(b_3,w_1)=1$. Now, note that  
\[\frac{Z_{H'_A\smallsetminus \{w_1,b_2\}}}{Z_{H'_A}} =
\mathbb{P}_{H'_A\cup \{(w_1 b_2)\}}[w_1b_2]\cdot\frac{Z_{H'_A\cup \{(w_1 b_2)\}}}{Z_{H'_A}}=
(-1)\cdot K^{-1}_{H'_A\cup \{(w_1 b_2)\}}(w_1,b_2)\cdot\frac{Z_{H'_A\cup \{(w_1 b_2)\}}}{Z_{H'_A}} .\]
This implies, that $K^{-1}_{H'_A\cup \{(w_1 b_2)\}}(w_1,b_2)$ is negative. To finish the proof, note that 
\[K^{-1}_{H'_A\cup \{(w_1 b_2)\}}(w_1,b_2)\cdot\frac{Z_{H'_A\cup \{(w_1 b_2)\}}}{Z_{H'_A}} =K^{-1}_{H'_A}(w_1,b_2)=R(w_1,b_2)=\beta.\]
Therefore $\beta$ is negative.
\end{proof}

\begin{proposition}\label{prop:T_A}
Let~${\cT=\cT_{(\widetilde{\mathcal{F}}^\bullet,\,\cF^\tw)}}$ be a t-realisation with gauge functions defined by
 formulas~\eqref{eqn:Fform} and~\eqref{eqn:Gform}. Then $\T$ is a perfect t-embedding of the reduced hexagon~$H'_A$, with the boundary polygon given by a regular hexagon with side length~$\Delta(A)$. 
\end{proposition} 

\begin{proof}
This proposition follows from~\cite[Theorem 4.1]{CLR2}. Due to Theorem~\ref{thm:proper_emb} it is enough to show that the boundary vertices of the augmented dual~$(H'_A)^*$ mapped by $\T$ to a regular hexagon~$P$, moreover the image of inner edges of~$(H'_A)^*$ adjacent to boundary vertices lie on the inner bisectors of the angles of the hexagon~$P$.  

\begin{figure}
 \begin{center}
\includegraphics[angle=-30,width=0.47\textwidth]{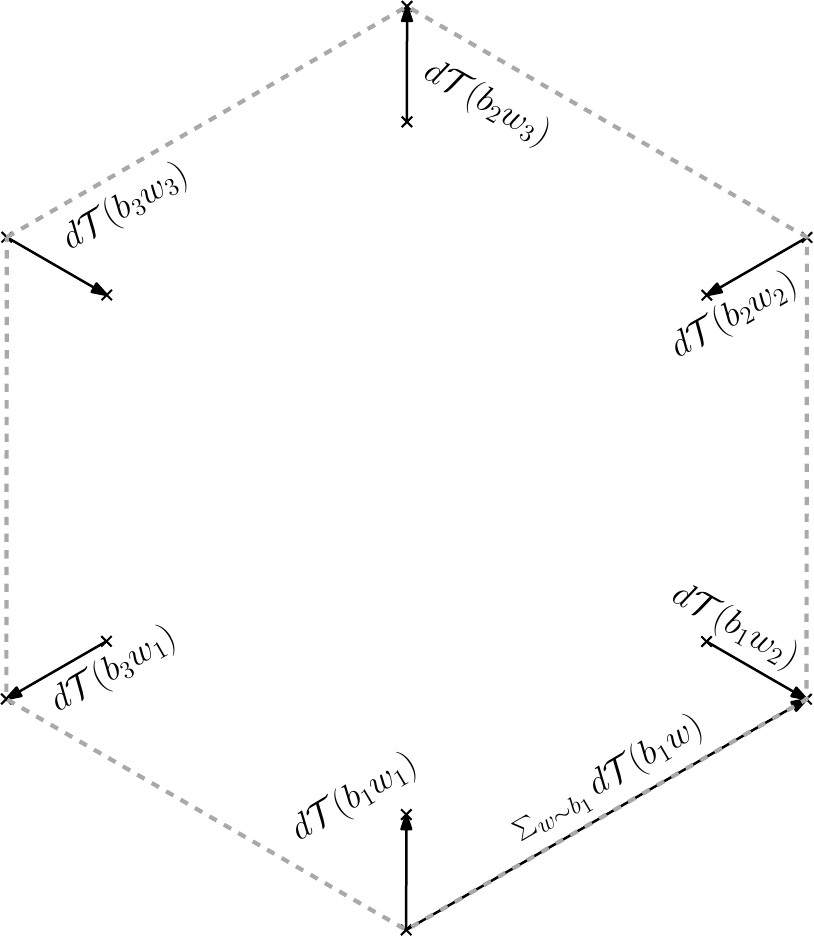}
$\quad$
\includegraphics[angle=-30,width=0.47\textwidth]{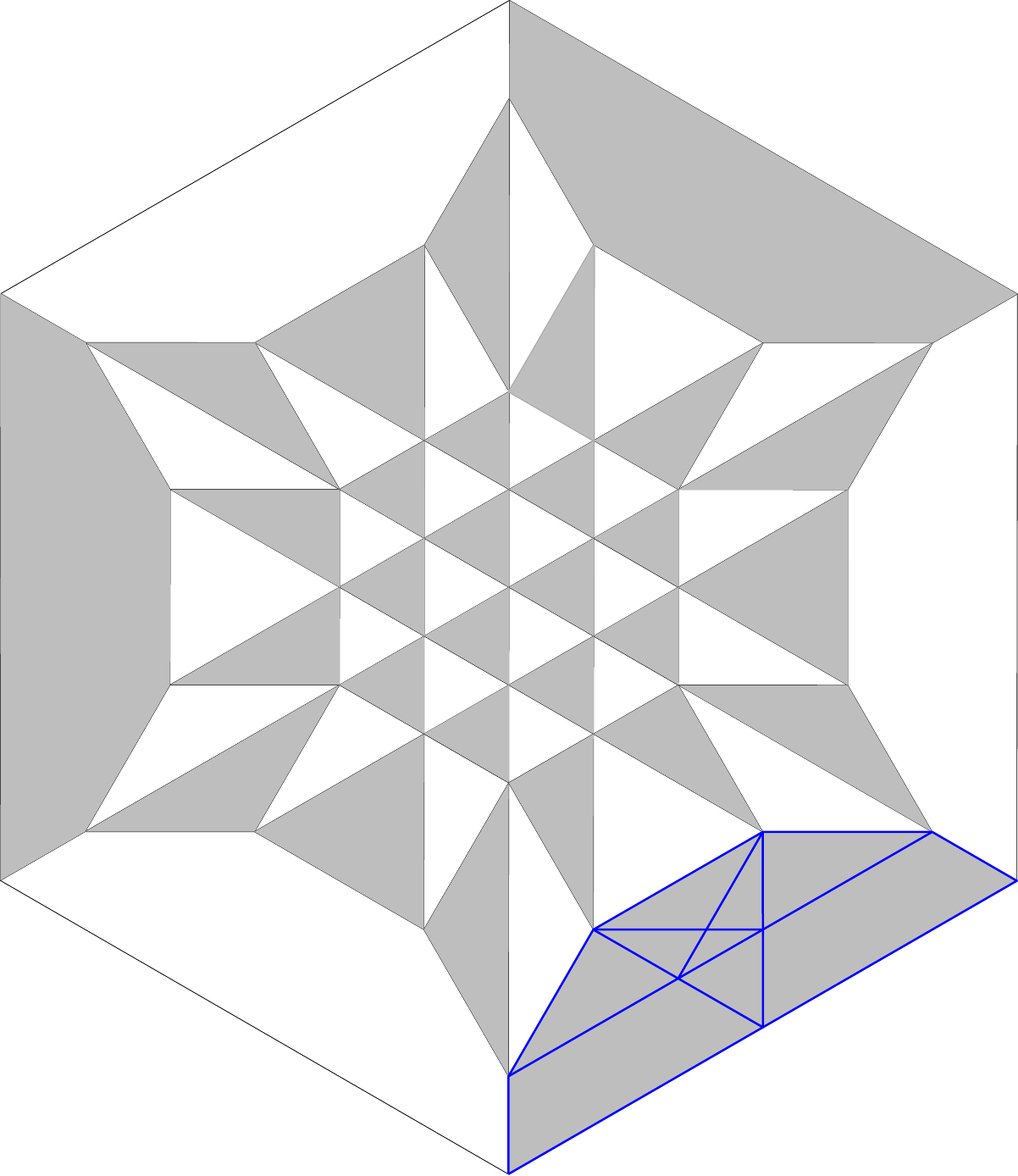}
  \caption{{\bf Left:} Edges of~$\T((H'_A)^*)$ between boundary faces and the boundary hexagon~$P$. {\bf Right:} 
  Perfect t-embedding of~$(H'_4)^*$ together with its origami map (shown in blue).}\label{fig:hex_temb}
 \end{center}
\end{figure}

Our first step is to prove that 
\[\e^{i \pi / 3}d\T(b_1w_1^*)=-d\T(b_1w_2^*),  \quad d\T(b_1w_1^*)\in \e^{i \pi / 3}\mathbb{R}_{>0}  \quad\text{ and }\quad
\sum_{w\sim b_1} d\T(b_1w^*) \in \mathbb{R}_{>0}.\]
Recall that $d\cT(bw^*)= \widetilde{\mathcal{F}}^\bullet(b)K_{\operatorname{reduced}}(b,w)\cF^\tw(w)$, therefore
\[d\T(b_1w_1^*)= \widetilde{\mathcal{F}}^\bullet(b_1)\cF^\tw(w_1) 
\quad\text{ and }\quad  
d\T(b_1w_2^*)= \widetilde{\mathcal{F}}^\bullet(b_1)\cF^\tw(w_2).\] Note that
\begin{align*}
\mathcal{F}^\circ(w_2) &= -R(w_2, b_1) + e^{\i \pi / 3} R(w_2, b_2) + e^{-\i \pi / 3} R(w_2, b_3)
\\
&=(e^{\i \pi / 3}-1)\alpha + e^{-\i \pi / 3}\beta
=-e^{\i \pi / 3} \big(      (e^{-\i \pi / 3}-1)\alpha  + e^{\i \pi / 3}\beta  \big)\\
&=-e^{\i \pi / 3}  \big(  -R(w_1, b_1) + e^{\i \pi / 3} R(w_1, b_2) + e^{-\i \pi / 3} R(w_1, b_3) \big)= -e^{\i \pi / 3} \mathcal{F}^\circ(w_1)
\end{align*}
which implies that $\e^{i \pi / 3}d\T(b_1w_1^*)=-d\T(b_1w_2^*)$. To see that $d\T(b_1w_1^*)\in \e^{i \pi / 3}\mathbb{R}$, note that 
\[\widetilde{\mathcal{F}}^\bullet(b_1) =  e^{\i 2\pi/3}  R(w_2, b_1) +  R(w_3, b_1) + e^{-\i 2\pi/3}  R(w_1, b_1)
= ( e^{\i 2\pi/3} + e^{-\i 2\pi/3} )\alpha+\beta=\beta-\alpha \in\mathbb{R}
\]
and $\mathcal{F}^\circ(w_1)=e^{\i \pi / 3}\widetilde{\mathcal{F}}^\bullet(b_1)$, therefore $\mathcal{F}^\circ(w_1)\widetilde{\mathcal{F}}^\bullet(b_1)=e^{\i \pi / 3}(\widetilde{\mathcal{F}}^\bullet(b_1))^2\in \e^{i \pi / 3}\mathbb{R}_{>0}$. Moreover, note that the boundary edge adjacent to the black boundary face~$b_1$, see Figure~\ref{fig:hex_temb}, is given by
\[\sum_{w\sim b_1} d\T(b_1w^*) = \widetilde{\mathcal{F}}^\bullet(b_1) \sum_{w\sim b_1} K_{\operatorname{reduced}}(b_1,w)\cF^\tw(w)=-\widetilde{\mathcal{F}}^\bullet(b_1)=\alpha-\beta\in\mathbb{R}_{>0},\]
recall that $\alpha-\beta>0$ by Lemma~\ref{lem:alpha_beta}.

Similarly, one can check that 
\begin{align*}
\e^{i \pi / 3}d\T(b_1w_2^*)&=-d\T(b_2w_2^*), \\
\e^{i \pi / 3}d\T(b_2w_2^*)&=-d\T(b_2w_3^*), \\
\e^{i \pi / 3}d\T(b_2w_3^*)&=-d\T(b_3w_3^*), \\
\e^{i \pi / 3}d\T(b_3w_3^*)&=-d\T(b_3w_1^*), \\
\e^{i \pi / 3}d\T(b_3w_1^*)&=-d\T(b_1w_1^*), 
\end{align*}
\[\sum_{w\sim b_2} d\T(b_2 w^*)=\sum_{b\sim w_1} d\T(b w_1^*)=(\alpha-\beta)e^{i2\pi/3},\]
\[\sum_{w\sim b_3} d\T(b_3 w^*)=\sum_{b\sim w_2} d\T(b w_2^*)=(\alpha-\beta)e^{-i2\pi/3},\]
\[ \quad \text{ and } \quad \sum_{b\sim w_3} d\T(b w_3^*)=\alpha-\beta.\]
This implies that the boundary vertices of the augmented dual~$(H'_A)^*$ are mapped by $\T$ to a regular hexagon~$P$ with side length~$\Delta(A)=|\alpha-\beta|$, and the image of inner edges of~$(H'_A)^*$ adjacent to boundary vertices lie on inner bisectors of the hexagon~$P$ as shown on Figure~\ref{fig:hex_temb}.
\end{proof}

\begin{remark}\label{rmk:guess_rmk}
We initially arrived at our gauge functions~\eqref{eqn:Fform} and~\eqref{eqn:Gform} in the following way. First, due to the definition of the Coulomb gauge functions we can conclude that for some coefficients $c^{\circ}_j, c^{\bullet}_j \in \mathbb{C}$, $j=1,2,3$ one has 
  \begin{align}\label{eq:guess_gauges}
\mathcal{F}^{\circ}(w) = \sum_{j=1}^3 c^{\circ}_j R(w, b_j) \quad\text{ and }\quad
\widetilde{\mathcal{F}}^\bullet(b) = \sum_{j=1}^3 c^{\bullet}_j R(w_j, b),
\end{align}
where $b_1, w_1, b_3, w_3, b_2, w_2$ are boundary vertices of the reduced hexagon~$H'_A$ and $R$ denotes the inverse Kasteleyn matrix of~$H'_{A}$. Indeed, this simply follows from~\eqref{eq:Coulomb-def} and the fact that the only function in the kernel of $K_{\operatorname{reduced}}$ is zero.

Next, we ansatz that the boundary polygon is given by a regular hexagon, as in Figure~\ref{fig:hex_temb}, left. 
Note that the boundary edge of the t-embedding adjacent to the black boundary face~$b_j \in H'_A$ for~$j\in\{1,2,3\}$ is given by
\[\sum_{w\sim b_j} d\T(b_jw^*) = \widetilde{\mathcal{F}}^\bullet(b_j) \sum_{w\sim b_j} K_{\operatorname{reduced}}(b_j,w)\cF^\tw(w)\]
and its origami map is given by
\[\sum_{w\sim b_j} d\cO(b_jw^*) = \widetilde{\mathcal{F}}^\bullet(b_j) \sum_{w\sim b_j} K_{\operatorname{reduced}}(b_j,w)\overline{\cF^\tw(w)}.\]
Due to Remark~\ref{rmk:lambda} we can assume that the coefficients in~\eqref{eq:guess_gauges} are chosen such that 
the origami map of the boundary edge adjacent to the black boundary face~$b_1$ is real and coincides with the edge of the t-embedding, in particular $\sum_{w\sim b_1} d\cO(b_1w^*)=\Delta(A)$, where $\Delta(A)$ is the side length of the regular hexagon. Assuming that the boundary polygon is given by a regular hexagon, the only additional condition for perfectness is the bisector condition, which in this case is equivalent to have the whole boundary mapped to the same segment under the origami map. More precisely, we have for~$j\in\{1,2,3\}$ 
\[\sum_{w\sim b_j} d\cO(b_jw^*)=\Delta(A)=- \sum_{b \sim w_j} d\cO(bw_j^*).\] 
The bisector condition for the regular hexagon together with~\eqref{eq:guess_gauges} imply that 
\[\frac{\sum_{w\sim b_1} d\T(b_1w^*) }{\sum_{w\sim b_1} d\cO(b_1w^*)}=
\frac{\sum_{w\sim b_1} K_{\operatorname{reduced}}(b_1,w)\cF^\tw(w)}{\sum_{w\sim b_1} K_{\operatorname{reduced}}(b_1,w)\overline{\cF^\tw(w)}}=\frac{c^{\circ}_1}{\overline{c^{\circ}_1}}=1.\]
 Similarly, 
 $\frac{c^{\circ}_2}{\overline{c^{\circ}_2}}= e^{\i 2\pi/3}$ and 
 $\frac{c^{\circ}_3}{\overline{c^{\circ}_3}}= e^{-\i 2\pi/3}$.
  Hence 
  \begin{align}\label{eq:c_circ}
  c^{\circ}_1\in\pm\mathbb{R}_{>0}, \quad 
  c^{\circ}_2\in \pm e^{\i \pi/3} \mathbb{R}_{>0} \quad\text{ and }\quad 
  c^{\circ}_3\in \pm e^{-\i \pi/3} \mathbb{R}_{>0}.\end{align}
  Note also, that 
 \begin{equation}\label{eqn:edge1}
 \sum_{w\sim b_1} d\T(b_1w^*) = 
 \widetilde{\mathcal{F}}^\bullet(b_1) \sum_{w\sim b_1} K_{\operatorname{reduced}}(b_1,w)\cF^\tw(w)
 =c^{\circ}_1\widetilde{\mathcal{F}}^\bullet(b_1)=\Delta(A)\in\mathbb{R}_{>0},
 \end{equation}
 therefore $\widetilde{\mathcal{F}}^\bullet(b_1)\in\pm\mathbb{R}_{>0}$ and has the same sign as $c^{\circ}_1$. Similarly, by looking at boundary edges of the t-embedding~$\sum_{w\sim b_2} d\T(b_2w^*)$ and~$\sum_{w\sim b_3} d\T(b_3w^*)$ we obtain
   \[
  \widetilde{\mathcal{F}}^\bullet(b_2)\in\pm e^{\i \pi/3}\mathbb{R}_{>0}\quad\text{ and }\quad
    \widetilde{\mathcal{F}}^\bullet(b_3)\in\pm e^{-\i \pi/3}\mathbb{R}_{>0}.
    \]
 Repeating the same argument for boundary edges adjacent to white faces, we get 
   \begin{align}\label{eq:c_bullet}
   c^{\bullet}_1\in\pm e^{-\i 2\pi/3}\mathbb{R}_{>0}, \quad 
   c^{\bullet}_2\in \pm e^{\i 2\pi/3} \mathbb{R}_{>0} \quad
   c^{\bullet}_3\in \pm  \mathbb{R}_{>0}\end{align}
  and
\[    
\mathcal{F}^{\circ}(w_1) \in \pm e^{-\i 2 \pi / 3} \mathbb{R}_{>0}, \quad    
   \mathcal{F}^{\circ}(w_2) \in \pm e^{\i 2 \pi / 3} \mathbb{R}_{>0} \quad
\mathcal{F}^{\circ}(w_3) \in \pm  \mathbb{R}_{>0}.
\]
Finally, we use the symmetry of the inverse Kasteleyn matrix described in~\eqref{eqn:alpha_eq} together with~\eqref{eq:guess_gauges},~\eqref{eq:c_bullet}, and knowing that  $\widetilde{\mathcal{F}}^\bullet(b_1)$ is real in order to conclude that $|c^{\bullet}_1|=|c^{\bullet}_2|$ and they have the same sign in~\eqref{eq:c_bullet}. Applying similar reasoning, we can determine that~$|c_3^\bullet| = |c_2^\bullet|$ and that~$c_3^\bullet$ has the same sign as~$c_1^\bullet$,~$c_2^\bullet$.

Similarly, we obtain that all coefficients~$c_j^{\circ}$ in~\eqref{eq:guess_gauges} have the same absolute values, and we can identify the signs in~\eqref{eq:c_circ} and~\eqref{eq:c_bullet} in such a way that~\eqref{eqn:edge1} and all other similar constraints are satisfied. 
This identifies~$\mathcal{F}^{\circ}(w)$ and~$\widetilde{\mathcal{F}}^\bullet(b)$ up to a positive multiplicative constant.
\end{remark}

\begin{remark} Note that if we renormalize one of the gauge functions by $\Delta(A)$, say, take
\begin{align}\label{eqn:Fform_ren}
    \mathcal{F}^\bullet(b) := \frac{1}{\Delta(A)} \left( e^{\i 2\pi/3}  R(w_2, b) +  R(w_3, b) + e^{-\i 2\pi/3}  R(w_1, b) \right)
    \end{align}
    and $ \mathcal{F}^\circ(w) $ defined by~\eqref{eqn:Gform},
then $\cT_{(\cF^\tb,\,\cF^\tw)}$ is a perfect t-embedding of $(H'_A)^*$ to the regular hexagon with the side length one.  We will use this normalization of the perfect t-embedding going forward.
\end{remark}



\section{Prelimit formulas}\label{subsec:prelim}
In this section we compute exact contour integral formulas for the perfect t-embedding  $\mathcal{T} = \mathcal{T}_A$ into a regular hexagon of sidelength $1$ and its origami map $\mathcal{O} = \mathcal{O}_A$. To get our formulas, we first obtain contour integral formulas for the gauge functions using the results of Section~\ref{subsubsec:temb_hex} together with known exact formulas for the inverse of the Kasteleyn matrix on the unreduced size~$A$ hexagon graph~$H_A$.


Our starting point is~\eqref{eqn:Gform} and~\eqref{eqn:Fform_ren} together with the exact formula for the inverse Kasteleyn matrix from \cite{Pet14}. We use the same Kasteleyn matrix and same coordinates for vertices used in that paper, which agrees with the coordinate system defined in Section~\ref{subsubsec:reduced_hex}, except that our definition of black and white vertices is swapped from the coloring of the vertices (which are triangles in the tiling formulation) in that paper. We record the result in our notation. To formulate the result, we will need the standard notation for~$n \in \mathbb{Z}$,~$p \in \mathbb{Z}_{\geq 0}$
$$(n)_p \coloneqq n (n+1) \cdots (n+p-1) \qquad (n)_0 \coloneqq 1$$
which is known as the Pochhammer symbol.

\begin{theorem}[\cite{Pet14} Theorem 6.1] \label{thm:Kast_exact}
Let~$K = K_A$ be the Kasteleyn matrix of the~$A\times A\times A$ hexagon~$H_A$ with all edge weights and Kasteleyn signs equal to~$1$: $K(b, w) \coloneqq \mathbf{1}_{b \sim w}$. 

For a white vertex at~$(y, m)$ and a black vertex at~$(x, n)$ (recall the coordinates defined in Section~\ref{subsubsec:reduced_hex}) we have
\begin{multline}\label{eqn:Kast_exact1}
K^{-1}((y, m), (x, n)) = (-1)^{y-x+m-n} \bigg(- \mathbf{1}_{m < n} \mathbf{1}_{y \leq x} \frac{(x - y + 1)_{n - m -1}}{(n-m-1)!}  \\
+ \frac{1}{(2 \pi \i)^2} \frac{(2A-n)!}{(2A-m-1)!} \int_{\{y, y+1, \dots, A-1\}} \int_{\{\infty\}}  \frac{(z_1-y+1)_{2A-m-1}}{(z_2-x)_{2 A-n+1}} \frac{(-2 A-z_2)_{A}}{(-2A-z_1)_{A}} \frac{(-z_2)_{A}}{(-z_1)_{A}} \frac{d z_2 d z_1}{z_2 - z_1} \bigg)
\end{multline}
The~$z_1$ contour contains the (potential) poles~$\{y, y+1, \dots, A-1\}$ and contains no other poles. The~$z_2$ contour~$\{\infty\}$ contains the $z_1$ contour and should be sufficiently large (i.e. it should contain all poles in the variable $z_2$ of the integrand).

\end{theorem}
\begin{remark}
    We can write the first term as a single integral via
    \begin{align}\label{eqn:Kast_diag}
     \mathbf{1}_{m < n} \mathbf{1}_{y \leq x} \frac{(x - y + 1)_{n - m -1}}{(n-m-1)!} 
    &=\mathbf{1}_{y \leq x} \frac{(2 A - n)!}{(2 A-m-1)!} \frac{1}{2 \pi \i} \int_{\{\infty\}} \frac{(z - y + 1)_{2 A - m - 1}}{(z - x)_{2 A - n + 1}} \; dz 
    \end{align}
    where $\{\infty\}$ is a contour which is large enough. See Equation (6.10) of~\cite{Pet14}. We also will make use of the identity in Lemma 6.2 of~\cite{Pet14}, which if $m < n$ can be rewritten as
    \begin{align}\label{eqn:Kast_diag2}
         \mathbf{1}_{y \leq x} \frac{(x - y + 1)_{n - m -1}}{(n-m-1)!} 
       = \frac{(2 A - n)!}{(2 A-m-1)!} \frac{1}{2 \pi \i} \int_{\{y,y+1,\dots, A-1\}} \frac{(z - y + 1)_{2 A - m - 1}}{(z - x)_{2 A - n + 1}} \; dz .
       \end{align} 
\end{remark}

\begin{remark}\label{b_ne_a}
Explicit computation using \eqref{eqn:Kast_exact1} yields that 
\begin{align}
R(w_1, b_1) &= K^{-1}(w(-1,1), b(-1,1)) =  \frac{(2 A -1)!}{(2 A)_A (A-1)!},  \label{eqn:alpha}\\
R(w_1, b_2) &= K^{-1}(w(-1,1), b(-1,2 A)) = - \frac{A}{2 A - 1} \frac{(2 A -1)!}{(2 A)_A (A-1)!}.  \label{eqn:beta_eq}
\end{align}
Therefore 
\begin{align}
\Delta(A) &=  \left(\frac{A}{2 A - 1} + 1\right) \frac{(2 A -1)!}{(2 A)_A (A-1)!}. \label{eqn:Delta_eq}
\end{align}
In particular, $\Delta(A)=\alpha-\beta>0$, where $\alpha$ and $\beta$ defined by~\eqref{eqn:alpha_eq}--\eqref{eqn:beta}.
\end{remark}

Now we provide exact inegral formulas for the gauge functions $\mathcal{F}^\circ$ and~$\mathcal{F}^\bullet$ defining the embedding $\mathcal{T}$ of the $A \times A \times A$ hexagon into the regular hexagon with side length $1$. Recall that 
\begin{align}\label{eqn:Fform2}
    \mathcal{F}^\bullet(b) = \frac{1}{\Delta(A)} \left( \i e^{\i \pi/6}  R(w_2, b) +  R(w_3, b) - \i e^{-\i \pi/6}  R(w_1, b) \right),
    \end{align}
    and 
    \begin{align}\label{eqn:Gform2}
    \mathcal{F}^\circ(w) = -R(w, b_1) + e^{\i \pi / 3} R(w, b_2) + e^{-\i \pi / 3} R(w, b_3).
    \end{align}

In order to state our formulas in the next proposition, we define the corresponding pair of holomorphic functions 
\begin{align}
   g_A(z_1) &= -1 +  \frac{A}{1+z_1} e^{\i \pi/3}  +\frac{A}{-A-z_1}  e^{-\i \pi/3},     \label{eqn:gdef} \\
   f_A(z_2) &= \bigg( \i e^{\i \pi/6}  \frac{1}{(2 A)_A (A-1)!}\frac{1}{z_2-(A-1)} \notag \\
   &\qquad \qquad +
   \frac{1}{2 \pi \i} \int_{\{0, 1,\dots,A-1\}} 
   \frac{d z_1}{z_2 - z_1} \frac{1}{(-2 A-z_1)_A (-z_1)_A}  \label{eqn:fdef}
   \\
   &\qquad \qquad \qquad \qquad -  \i e^{-\pi \i/6} \frac{1}{(2 A)_A (A-1)!}  \; \frac{1 }{z_2+2 A}  \bigg) . \notag 
\end{align}

\begin{proposition}[Contour integral formulas]\label{prop:contour_int}
    The functions~$\mathcal{F}^\bullet$ and~$\mathcal{F}^\circ$ from~\eqref{eqn:Fform2} and~\eqref{eqn:Gform2} are given, for a white vertex~$ w$ with coordinates~$(y, m)$ which is away from the boundary of the reduced hexagon $H_{A}'$, and a black vertex~$ b$ at position~$(x,n)$ away from the boundary of the reduced hexagon $H_{A}'$, by the following contour integral formulas:
    \begin{align}
    \mathcal{F}^\bullet( b) &= \frac{1}{\Delta(A)} (-1)^{x+n} \frac{(2A-n)!}{2 \pi \i} \int_{\{\infty\}}  \; \frac{ (-2A-z_2)_{A} (-z_2)_{A}}{(z_2-x)_{2A-n+1}}  f_A(z_2) \; dz_2 ,\label{eqn:Fexact}\\
    \mathcal{F}^\circ( w) &= (-1)^{y+m} \frac{1 }{2 \pi \i} \frac{(2A-1)!}{(2A-m-1)!} \int_{\{y,y+1,\dots,A-1\}}  \frac{(z_1-y+1)_{2A-m-1}}{(-2A-z_1)_{A} (-z_1)_{A}} g_A(z_1) \; dz_1  \label{eqn:Gexact} . 
    \end{align}

    \end{proposition}

    \begin{remark}

        The values~$\mathcal{F}^\circ(w_j)$ and~$\mathcal{F}^\bullet(b_j)$ for contracted boundary vertices~$w_j$ and~$b_j$,~$j=1,2,3$, in the reduced hexagon are (essentially) given in the proof of Proposition~\ref{prop:T_A}. In particular, they can be computed from the values of~$\alpha$ and~$\beta$, which are given in Remark~\ref{b_ne_a}. The formulas above either give the correct values at boundary vertices, or can be adjusted to give the correct boundary values. However, we will not need the exact boundary values in what follows, anyways.
    \end{remark}

\begin{proof}
We use the formula \eqref{eqn:Gform2} for $\mathcal{F}^{\circ}$. The function $R(\mathrm w, \mathrm b)$ is given by the formula for~$K^{-1}$ in Theorem \ref{thm:Kast_exact} together with Lemma \ref{lem:reduced_K}. Indeed, if the white vertex $w$ at $(y, m)$ is a vertex where we did not gauge by a~$-1$ during the reduction process described in Section~\ref{subsubsec:reduced_hex} (also c.f. Figure~\ref{fig:hex_reduction}), we may write 
\begin{align}\label{eqn:white_exact}
    \mathcal{F}^\circ(w) &=  -K^{-1} (w, b(-1, 1)) + e^{\i \pi / 3} K^{-1}(w, b(-1, 2 A)) + e^{-\i \pi / 3} K^{-1}(w, b(-A, 2 A)) .
\end{align}

We will first focus on the contributions coming from the double integral. Then we will use the identity \eqref{eqn:Kast_diag2} to show that the contributions coming from the single integral term in the expression for $K^{-1}$ cancel with the residue in the double integral at $z_1 = z_2$. We now compute \eqref{eqn:white_exact} for $w = w(y, m)$. The contributions from the double integral term of $K^{-1}$ involves a sum of three terms; we have
\begin{multline}\label{eqn:double_int_terms}
  (-1)^{y + m}  \frac{1}{2 \pi \i }\frac{1}{(2 A - m - 1)!} \int_{\{y,y+1,\dots, A-1\}}dz_1 \frac{(z_1 - y+1)_{2 A-m-1}}{(-2 A-z_1)_A (-z_1)_A} \\
 \times \left( 
-I_1(z_1) + e^{\i \pi/3} I_2(z_1) + 
e^{-\i \pi/3} I_3(z_1)  \right),
\end{multline}
where

\begin{align*}
    I_1(z_1) &= \frac{1}{2 \pi \i }\int_{\{\infty\}} dz_2 (-2 A-z_2)_A (-z_2)_A \frac{(2A-1)!}{(z_2+1)_{2 A}} \frac{1}{z_2 - z_1} , \\
    I_2(z_1) &= \frac{1}{2 \pi \i }\int_{\{\infty\}} dz_2 (-2 A-z_2)_A (-z_2)_A \frac{1}{(z_2+1)} \frac{1}{z_2 - z_1}, \\
    I_3(z_1) &= \frac{1}{2 \pi \i }\int_{\{\infty\}} dz_2 (-2 A-z_2)_A (-z_2)_A \frac{1}{(z_2+A)} \frac{1}{z_2 - z_1}.
\end{align*}

Taking the residue at $\infty$ yields 
\begin{align}\label{eqn:I1}
I_1(z_1) = (2 A-1)! .
\end{align}
 We may evaluate $I_2(z_1)$ by taking the residue at $z_2 = z_1$, call it $R_2(z_1)$, and at $z_2 = -1$. Adding these two contributions together yields 
\begin{align}\label{eqn:I2}
    I_2(z_1) = R_2(z_1)
+ (2 A-1)! \frac{A}{z+1} .
\end{align}
Finally, $I_3(z_1)$ will also contain a term from the $z_2 = z_1$ residue, and the residue at $z_2 = -A$ will yield a term of the form $(2 A-1)! \frac{A}{-z_1-A}$. Therefore, we have
\begin{align}\label{eqn:I3}
    I_3(z_1) = R_3(z_1)+
(2 A-1)! \frac{A}{-z_1-A} .
\end{align}

So, altogether
\begin{itemize}
\item By the identity \eqref{eqn:Kast_diag2}, integrating out (over $z_1$) the contribution~$R_2(z_1)$ from the residue at $z_2 = z_1$ in~$I_2(z_1)$ equals $\mathbf{1}_{y \leq x} \frac{(x - y + 1)_{n - m -1}}{(n-m-1)!} |_{x = -1, n = 2 A}$, which in turn equals 
$$\mathbf{1}_{m < n} \mathbf{1}_{y \leq x} \frac{(x - y + 1)_{n - m -1}}{(n-m-1)!}|_{x = -1, n = 2 A}$$
 since $m  < n = 2 A$ always holds. This cancels with the indicator term in the inverse Kasteleyn on the first line of~\eqref{eqn:Kast_exact1}.
 \item The integral of~$R_3(z_1)$, which is equal to $\mathbf{1}_{m < n} \mathbf{1}_{y \leq x} \frac{(x - y + 1)_{n - m -1}}{(n-m-1)!}|_{x = -A, n = 2 A} = 0$, vanishes because no white vertex $(y, m)$ in the reduced graph has $y \leq -A$.
\item The term $\mathbf{1}_{m < n} \mathbf{1}_{y \leq x} \frac{(x - y + 1)_{n - m -1}}{(n-m-1)!}$ in the formula \eqref{eqn:Kast_exact1} for $K^{-1}$ in the first and third terms of \eqref{eqn:white_exact} vanishes.
\end{itemize}
By the above and by Equations \eqref{eqn:I1}-\eqref{eqn:I3} we get
\begin{align*}
    \mathcal{F}^\circ(w) &= (-1)^{y + m}  \frac{1}{2 \pi \i }\frac{(2A-1)!}{(2A - m - 1)!} \int_{\{y,y+1,\dots, A-1\}}dz_1 \frac{(z_1 - y+1)_{2A-m-1}}{(-2A-z_1)_A (-z_1)_A} \\
    &\qquad \times \left( 
    -1 + e^{\i \pi/3} \frac{A}{z_1+1} + 
    e^{-\i \pi/3} \frac{A}{-z_1-A}  \right)
\end{align*}
which matches the formula in the proposition.

For $ \mathcal{F}^\bullet(b)$, the computation is similar; the starting point is the formula~\eqref{eqn:Fform2}. 

By using $w(-1,1), w(A-1,1), w(-1, 2 A -1)$ as the representatives for boundary vertices $w_1, w_2, w_3$, respectively, we end up with, for~$b$ the black vertex at~$(x, n)$
 \begin{align*}
  \Delta(A)  \mathcal{F}^\bullet(b) &=  \frac{(-1)^{x+n}}{(2 \pi \i)^2}  \int_{\{\infty\}} dz_2 \frac{(2 A-n)!}{(z_2-x)_{2 A-n+1}} (-2A -z_2)_A (-z_2)_A \\
 &  \bigg( \i e^{\i \pi/6} \int_{\{A-1\}} \frac{dz_1}{z_2-z_1} \frac{(z_1-A+2)_{2 A-2}}{(2 A-2)!} \frac{1}{(-2A-z_1)_A (-z_1)_A}  \\
 & + \int_{\{-1, \dots, A-1\}} \frac{dz_1}{z_2-z_1}\frac{(z_1+2)_{0}}{0!} \frac{1}{(-2 A-z_1)_A (-z_1)_A}  \\
 &- \i e^{-\i \pi/6} \int_{\{-1, \dots, A-1\}} \frac{dz_1}{z_2-z_1} \frac{(z_1+2)_{2 A-2}}{(2 A-2)!} \frac{1}{(-2 A-z_1)_A (-z_1)_A}  \bigg)   \\
 &   + (\text{three indicator terms}).
 \end{align*}
For a black vertex~$b$ not on the boundary, we will get no indicator term except for from the third term in \eqref{eqn:Fform2}.

Now we compute the three single integrals over~$z_1$. The first simplifies to
\begin{align*}
I_1(z_2) &= \int_{\{A-1\}} \frac{(z_1-A+2)_{2 A-2}}{(2 A-2)!} \frac{1}{(-2 A-z_1)_A (-z_1)_A} \; \frac{dz_1 }{z_2-z_1}  \\
&= -(2 \pi \i) \frac{(1)_{2 A-2}}{(2 A-2)!} \frac{1}{(-2 A-A+1)_A (-(A-1))_{A-1}}\frac{1}{z_2-(A-1)} \\
 &= -(2 \pi \i)  \frac{1}{(-2 A-A+1)_A (-(A-1)!)}\frac{1}{z_2-(A-1)} \\
  &= (2 \pi \i)  \frac{1}{(2 A)_A (A-1)!}\frac{1}{z_2-(A-1)}.
 \end{align*}
 
The second integral can be rewritten as
\begin{align*}
I_2(z_2) &= \int_{\{0,\dots, A-1\}} \frac{1}{(-2 A-z_1)_A (-z_1)_A} \; \frac{dz_1 }{z_2-z_1}  
 \end{align*}
since the integrand is regular at $z_1 = -1$. This already matches the second term in \eqref{eqn:Fexact}.

The third term is
\begin{align*}
I_3(z_2) &= \int_{\{-1, \dots, A-1\}} \frac{(z_1+2)_{2 A-2}}{(2 A-2)!} \frac{1}{(-2 A-z_1)_A (-z_1)_A}   \; \frac{dz_1 }{z_2-z_1}  \\
&=  \int_{\{-1, \dots, A-1\}} \frac{(z_1+2)_{A-1}}{(2 A-2)!} \frac{1}{(2 A+z_1) (-z_1)_A} \; \frac{dz_1 }{z_2-z_1} \\
&=  -\int_{\{-2 A\}} \frac{(z_1+2)_{A-1}}{(2 A-2)!} \frac{1}{(2 A+z_1) (-z_1)_A} \; \frac{dz_1 }{z_2-z_1} + 2 \pi \i R(z_2) \\
&= - (2 \pi \i)  \frac{(-2 A+2)_{A-1}}{(2 A-2)!} \frac{1}{(2A)_A} \; \frac{1}{z_2+2 A} + 2 \pi \i R(z_2) \\
&=  (2 \pi \i)  \frac{(2 A-2)!}{(2A-2)! (A-1)!} \frac{1}{(2A)_A} \; \frac{1}{z_2+2 A}+ 2 \pi \i R(z_2) \\
&= (2 \pi \i)  \frac{1}{(A-1)!} \frac{1}{(2 A)_A} \; \frac{1 }{z_2+2 A} + 2 \pi \i R(z_2) .
\end{align*}

The term $R(z_2$) arises when we drag the~$z_1$ contour through~$\infty$ to surround only the residue at~$-2 A$, and it is the contribution from the pole~$\frac{1}{z_2-z_1}$. The integral of $R(z_2)$ over $z_2$ in the expression for $\mathcal{F}^\bullet(b)$ exactly cancels with the corresponding indicator term by \eqref{eqn:Kast_diag2}. Thus, summing up the integrals of the 1st, 2nd, and 3rd terms gives the result.
\end{proof}

Next, we integrate the closed $1$-form on the dual graph defined by $\mathcal{F}^\bullet$ and $\mathcal{F}^\circ$ (c.f.~\eqref{eq:TO-def-via-F}) to obtain exact formulas for $\mathcal{T} = \mathcal{T}_A$ and $\mathcal{O} = \mathcal{O}_A$. We choose the constant of integration in the definition of $\mathcal{T}$ and $\mathcal{O}$ so that $\mathcal{T}(v_1) = -e^{\i \pi /3}$ and $\mathcal{O}(v_1) = -\frac{1}{2}$. This way, with the renormalized black gauge function \eqref{eqn:Fform2} and with \eqref{eqn:Gform2}, the boundary polygon of the resulting embedding is a regular hexagon with side length $1$ and vertices $\mathcal{T}(v_j) = -e^{\i j \pi/3 }$, $j = 1,\dots, 6$. We also have $\mathcal{O}(v_j) = (-1)^j \frac{1}{2}$ (image of the boundary polygon under the origami map $\mathcal{O}$ equals $[-\frac{1}{2}, \frac{1}{2}] \subset \mathbb{R}$).

\begin{theorem}[Exact formulas for $\mathcal{T}$, $\mathcal{O}$]\label{thm:exact_hex}
    Let~$\mathcal{T} = \mathcal{T}_A$ be the perfect t-embedding of the $A \times A \times A$ reduced hexagon~$H_A'$ defined by $\mathcal{F}^\bullet$ and $\mathcal{F}^\circ$ as described above, and let $\mathcal{O} = \mathcal{O}_A$ be its origami map. The position of the face $(x, n)$ under the t-embedding~$\mathcal{T}$ of the reduced hexagon~$H_{A}'$ is given by the following exact formula:

    \begin{multline}\label{eqn:T_exact_hex}
        \mathcal{T}(x, n) =  \mathcal{T}(v_1) + \mathcal{F}^\bullet(b_1) \mathcal{F}^\circ(w_1) 
        \\
        +
        \frac{(2 A-1)!}{\Delta(A)}     \frac{1}{(2 \pi \i)^2} \int_{\{0,1,\dots,A-1\}} \int_{\{\infty\}}  \frac{dz_2 dz_1}{z_1 - z_2} \;\left(
            \frac{(z_1-x+1)_{2 A-n}}{(z_2-x+1)_{2 A-n}}  
        - \frac{(z_1+1)_{2 A-1}}{(z_2+1)_{2 A-1}}  
        \right) 
           \\
        \times    \frac{ (-2A-z_2)_{A} (-z_2)_{A}}{(-2A-z_1)_{A} (-z_1)_{A}} f_A(z_2) g_A(z_1).
    \end{multline}
    The origami map $\mathcal{O}(x, n)$ evaluated at the corresponding face is given by
    \begin{multline}\label{eqn:O_exact_hex}
        \mathcal{O}(x, n) =  \mathcal{O}(v_1^*) + \mathcal{F}^\bullet(b_1) \overline{\mathcal{F}^\circ(w_1) }
        \\
        +
        \frac{(2 A-1)!}{\Delta(A)}     \frac{1}{(2 \pi \i)^2} \int_{\{0,1,\dots,A-1\}} \int_{\{\infty\}}  \frac{dz_2 dz_1}{z_1 - z_2} \;\left(
            \frac{(z_1-x+1)_{2 A-n}}{(z_2-x+1)_{2 A-n}}  
        - \frac{(z_1+1)_{2 A- 1}}{(z_2+1)_{2 A-1}}  
        \right) 
           \\
        \times    \frac{ (-2 A-z_2)_{A} (-z_2)_{A}}{(-2 A-z_1)_{A} (-z_1)_{A}} f_A(z_2) \overline{g_A}(z_1) .
    \end{multline}
    In the formulas above, $f_A$ and $g_A$ are defined by \eqref{eqn:gdef} and \eqref{eqn:fdef}, and $\overline{g_A}(z_1) \coloneqq \overline{g_A(\overline{z_1})}$.
    
    \end{theorem}
    
    \begin{proof}
    We start with \eqref{eqn:Gexact} and \eqref{eqn:Fexact}, and use these to compute $\mathcal{T}(x, n)$ by summing $ \mathcal{F}^\bullet(b) K_r(b, w) \mathcal{F}^\circ(w)$ along a path in the dual graph from $v_1^*$ to $(x, n)$, where $K_r$ is the Kasteleyn matrix on the reduced graph $H_A'$. We recall these formulas here, using \eqref{eqn:gdef} and \eqref{eqn:fdef}:
    \begin{align*}
        \mathcal{F}^\circ(w(y, m)) &= (-1)^{y+m} \frac{1 }{2 \pi \i} \frac{(2A-1)!}{(2A-m-1)!} \int_{\{y,y+1,\dots,A-1\}} dz_1 \frac{(z_1-y+1)_{2A-m-1}}{(-2A-z_1)_{A} (-z_1)_{A}} g_A(z) \\
        \mathcal{F}^\bullet(b(x,n)) &= \frac{1}{\Delta(A)} (-1)^{x+n} \frac{(2A-n)!}{2 \pi \i} \int_{\{\infty\}} dz_2 \; \frac{ (-2A-z_2)_{A} (-z_2)_{A}}{(z_2-x)_{2A-n+1}}   f_A(z_2) .
        \end{align*}
    We sum along a path of the form $v_1^* \rightarrow (0, n) \rightarrow (x, n)$. So we must compute 
    \begin{align}
        \mathcal{T}(x, n) &=  \mathcal{T}(v_1^*) + \mathcal{F}^\bullet(b_1) \mathcal{F}^\circ(w_1)  \label{eqn:T_exact_hex1} 
        \\
        &- \sum_{m=1}^{n-1} \mathcal{F}^\bullet(b(-1, m+1)) \mathcal{F}^\circ(w(0, m)) \label{eqn:T_exact_hex2} \\
        &\pm \sum_{y}  \mathcal{F}^\bullet(b(y, n)) \mathcal{F}^\circ(w(y, n))  \label{eqn:T_exact_hex3} 
    \end{align}
    where the second summation over $y$ is either from $0$ to $x-1$ or from~$x$ to~$-1$ depending on whether $x > 0$ or $x < 0$, and in the latter case we pick up a minus sign in front as well.
    
    Now we compute the first summation
    \begin{align*}
        & \sum_{m=1}^{n-1}  \frac{1}{\Delta(A)}  \frac{(2A-m-1)!}{2 \pi \i} \int_{\{\infty\}} dz_2 \; \frac{ (-2 A-z_2)_{A} (-z_2)_{A}}{(z_2+1)_{2A-m}}  f_A(z_2) \\
        &\times \frac{1 }{2 \pi \i} \frac{(2A-1)!}{(2A-m-1)!} \int_{\{0,1,\dots,A-1\}} dz_1 \frac{(z_1+1)_{2A-m-1}}{(-2 A-z_1)_{A} (-z_1)_{A}} g_A(z_1)\\
        &= \frac{(2A-1)!}{\Delta(A)}     \frac{1}{(2 \pi \i)^2} \int_{\{0,1,\dots,A-1\}} \int_{\{\infty\}}  dz_2 dz_1 \; \sum_{m=1}^{n-1} \frac{ (z_1+1)_{2A-m-1}}{(z_2+1)_{2A-m}}   \\
        &\times    \frac{ (-2A-z_2)_{A} (-z_2)_{A}}{(-2A-z_1)_{A} (-z_1)_{A}} f_A(z_2) g_A(z_1) \\
        &= \frac{(2A-1)!}{\Delta(A)}     \frac{1}{(2 \pi \i)^2} \int_{\{0,1,\dots,A-1\}} \int_{\{\infty\}}  \frac{dz_2 dz_1}{z_1 - z_2} \;\left(\frac{ (z_1+1)_{2A-1}}{(z_2+1)_{2A-1}} -
        \frac{ (z_1+1)_{2A-n}}{(z_2+1)_{2A-n}}   \right) 
           \\
        &\times    \frac{ (-2A-z_2)_{A} (-z_2)_{A}}{(-2A-z_1)_{A} (-z_1)_{A}} f_A(z_2) g_A(z_1).
    \end{align*}

    To verify the second equality above one checks that~
    $$(z_1 - z_2) \sum_{m=1}^{n-1}\frac{ (z_1+1)_{2A-m-1}}{(z_2+1)_{2A-m}} = \sum_{m=1}^{n-1}\frac{ (z_1+1)_{2A-m-1}}{(z_2+1)_{2A-m}} ((z_1 + 2 A - m) - (z_2 + 2 A- m))$$
    is a telescoping sum. Then, we compute the second summation, first assuming $x > 0$:
    
    \begin{align*}
        & \sum_{y=0}^{x-1}  \frac{1}{\Delta(A)}  \frac{(2A-n)!}{2 \pi \i} \int_{\{\infty\}} dz_2 \; \frac{ (-2A-z_2)_{A} (-z_2)_{A}}{(z_2-y)_{2A-n+1}}  f_A(z_2) \\
        &\times \frac{1 }{2 \pi \i} \frac{(2A-1)!}{(2A-n-1)!} \int_{\{y,y+1,\dots,A-1\}} dz_1 \frac{(z_1-y+1)_{2A-n-1}}{(-2A-z_1)_{A} (-z_1)_{A}} g_A(z_1)\\
        &= \frac{(2A-1)!}{\Delta(A)}     \frac{(2A-n)}{(2 \pi \i)^2} \int_{\{0,1,\dots,A-1\}} \int_{\{\infty\}}  dz_2 dz_1 \; \sum_{y=0}^{x-1} \frac{ (z_1-y+1)_{2A-n-1}}{(z_2-y)_{2A-n+1}}   \\
        &\times    \frac{ (-2A-z_2)_{A} (-z_2)_{A}}{(-2A-z_1)_{A} (-z_1)_{A}} f_A(z_2) g_A(z_1) \\
        &= - \frac{(2A-1)!}{\Delta(A)}     \frac{1}{(2 \pi \i)^2} \int_{\{0,1,\dots,A-1\}} \int_{\{\infty\}}  \frac{dz_2 dz_1}{z_1 - z_2} \;\left(
            \frac{(z_1+1)_{2A-n}}{(z_2+1)_{2A-n}}    
        -
        \frac{(z_1-x+1)_{2A-n}}{(z_2-x+1)_{2A-n}} \right) 
           \\
        &\times    \frac{ (-2A-z_2)_{A} (-z_2)_{A}}{(-2A-z_1)_{A} (-z_1)_{A}} f_A(z_2) g_A(z_1).
    \end{align*}
    We were able to change the integration contour to one containing $\{0,1,\dots, A-1\}$ in the first equality above because there is in fact no singularity in $z_1$ at any point in $\{0,1,\dots, y-1\}$. The second equality is due to the fact that for $a, b \in \mathbb{C}$ and~$m \in \mathbb{Z}_{\geq 0}$
\begin{align*}
(b-a)\frac{(a + 1)_{m-1}}{(b)_{m+1}} &= \frac{1}{m} (b (a+m) - a (b+m) )\frac{(a + 1)_{m-1}}{(b )_{m+1}} \\
&=\frac{1}{m} \left(\frac{(a + 1)_{m}}{(b + 1)_{m}} - \frac{(a)_{m}}{(b)_{m}} \right).
\end{align*}
Indeed, one may apply the above with~$a = z_1 -y$,~$b = z_2 - y$, and~$m = 2 A - n$ to get the second equality.

    If $x < 0$, the second summation \eqref{eqn:T_exact_hex3} is instead $-\sum_{y=x}^{-1} $, and the result is given by the same exact expression. 
    
    Finally, we add up the contributions \eqref{eqn:T_exact_hex1}-\eqref{eqn:T_exact_hex3}, taking into account the signs in front of the summations, to get the result.

    The computation for $\mathcal{O}(x, n)$ is exactly the same, once one realizes that 
$$
    \overline{\mathcal{F}^\circ(w(y, m))} = (-1)^{y+m} \frac{1 }{2 \pi \i} \frac{(2A-1)!}{(2A-m-1)!} \int_{\{y,y+1,\dots,A-1\}} dz_1 \frac{(z_1-y+1)_{2A-m-1}}{(-2A-z_1)_{A} (-z_1)_{A}} \overline{g_A}(z_1).
    $$
    \end{proof}


\section{Scaling limit}
\label{sec:scaling_limit}

In this section we study the scaling limits $\lim_{A \rightarrow \infty} \mathcal{T}(\chi A, \eta A)$ and $\lim_{A \rightarrow \infty} \mathcal{O}(\chi A, \eta A)$, where $(\chi, \eta) = (\frac{x}{A}, \frac{n}{A})$ are rescaled continuous coordinates. We will first describe the limiting objects in Section~\ref{subsec:limiting}, and then state the theorems about convergence in Section \ref{subsec:limits}. We postpone the proof of convergence until Section~\ref{sec:steep_descent}, where we will perform the required saddle point analysis. At the end of this section, we will also explain how to match the conformal structure on the rough region~$\mathcal{D}$ (see Definition~\ref{def:liquid_frozen}) obtained via the limiting objects~$\lim_{A \rightarrow \infty} \mathcal{T}(\chi A, \eta A)$ and~$\lim_{A \rightarrow \infty} \mathcal{O}(\chi A, \eta A)$ with the Kenyon-Okounkov conformal structure which is known to describe the limiting GFF~\cite{Pet15}. 

\subsection{The continuous analogues of the functions $\mathcal{T}_A$ and $\mathcal{O}_A$}
\label{subsec:limiting}

For the purposes of describing the scaling limits of~$ \mathcal{T}$ and~$\mathcal{O}$, we define a pair of functions, given for~$\zeta \in \mathbb{H}$, where~$\mathbb{H}$ is the upper half plane, by
\begin{align} \label{eqn:tdef}
z(\zeta) &= e^{- \i \frac{2 \pi}{3}} -\frac{1}{2 \pi \i} \int_{\gamma_\zeta} f(z) g(z) \; dz \\
\vartheta(\zeta) &= -\frac{1}{2} -\frac{1}{2 \pi \i} \int_{\gamma_\zeta} f(z) \overline{g}(z) \; dz \label{eqn:odef}
\end{align}
where~$\overline{g}(z) \coloneqq \overline{g(\overline{z})}$, and $f$ and $g$ are the meromorphic functions given by:
\begin{align}\label{eqn:g_def2}
    g(z) &= -1 +  \frac{1}{z} e^{\i \pi/3}  -\frac{1}{z+1}  e^{-\i \pi/3} \\
    f(z) &= \frac{2}{3} \left( e^{\i 2 \pi/3} \frac{1}{z - 1} - \frac{1}{2}\frac{1}{z - (-\frac{1}{2})}
    + e^{-\i 2 \pi/3} \frac{1}{z + 2} \right). \label{eqn:f_def2} 
\end{align}
The functions $f$ and $g$ appear naturally from the asymptotic analysis of $\mathcal{T}$ and $\Or$ in Section~\ref{subsec:main_saddle} (see Equations \eqref{eqn:Rg_def} and \eqref{eqn:Rf12_def} and the discussion at the beginning of Section~\ref{subsec:main_saddle}), and are certain limits of the functions~$f_A$ and~$g_A$. The contour $\gamma_\zeta$ is given by a contour from $\overline{\zeta}$ to $\zeta$ which crosses $\R$ in the interval $(-\infty, -2)$. We will see that $z(\zeta(\chi, \eta))$ and $\vartheta(\zeta(\chi, \eta))$ are the scaling limits of the t-embeddings and origami maps, respectively, where~$\zeta(\chi, \eta)$ is the \emph{critical point} map (defined by~\eqref{eqn:crit_explicit}) from the rescaled hexagonal domain~$\lim_{A \rightarrow \infty} \frac{1}{A}H_A$ to the closure of the upper half plane. 

First we prove some preliminary facts about $z(\zeta)$ and $\vartheta(\zeta)$.

\begin{figure}
    \centering
\includegraphics*[scale=.3]{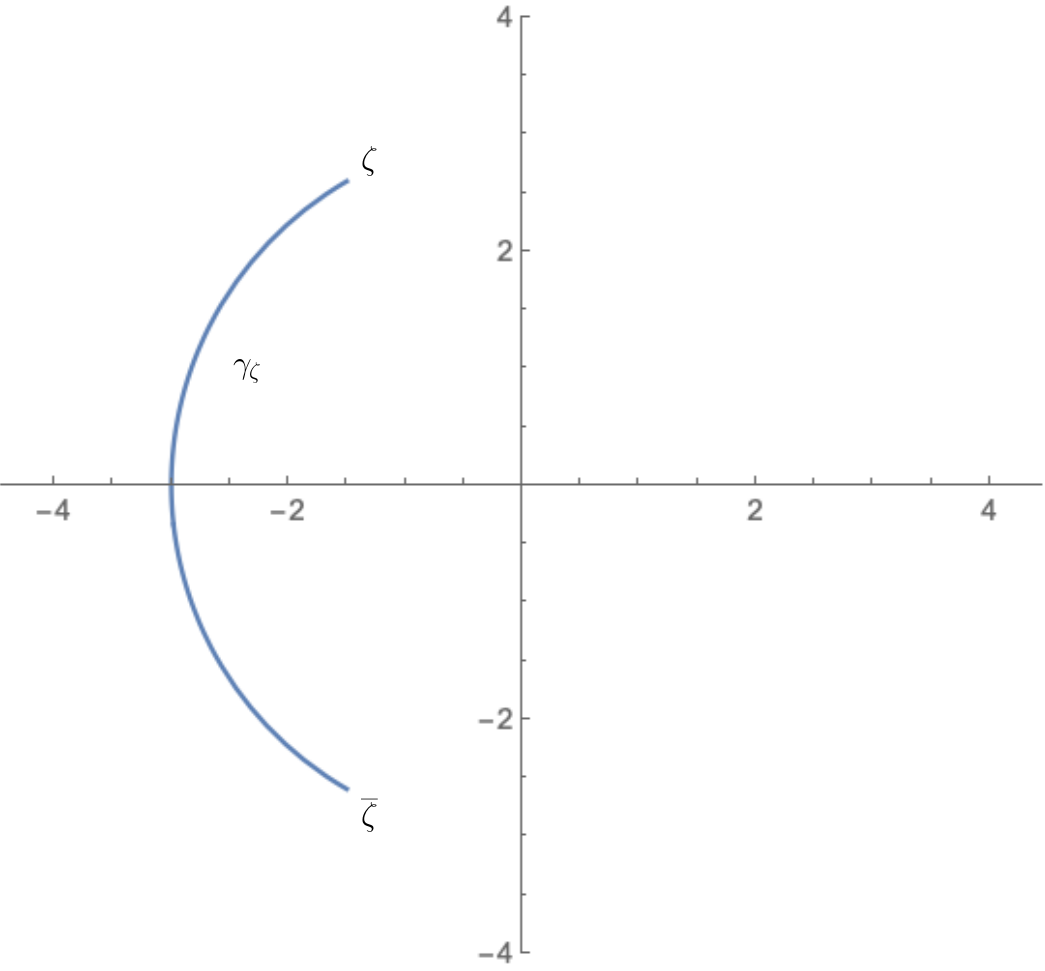}
\caption{The contour $\gamma_\zeta$ in \eqref{eqn:tdef} and \eqref{eqn:odef}.}
\label{fig:gamma_zeta}
\end{figure}

\begin{lemma}
    The function~$\vartheta(\zeta)$ defined by~\eqref{eqn:odef} is real valued.
\end{lemma}
\begin{proof}
Note that 
\begin{align*}
    f(z) \bar{g}(z) &= \frac{1}{z-1} - \frac{1}{z} - \frac{1}{z+1} + \frac{1}{z+2} + \frac{1}{z + \frac{1}{2}}.
\end{align*}
This function is a rational function with \emph{real} coefficients. Thus, for any choice of~$\zeta_0 \in (-\infty, -2)$, we have $\frac{1}{2 \pi \i} \int_{\gamma_\zeta} f(z) \overline{g}(z) dz = \frac{1}{\pi} \Im\left[ \int_{\zeta_0}^\zeta f(z) \overline{g}(z) dz \right] \in \mathbb{R}$. One way to see this is to choose the contour~$\gamma_{\zeta}$ from $\overline{\zeta}$ to $\zeta$ to be symmetric about the real axis.
\end{proof}

The next lemma shows that~$z : \mathbb{H} \rightarrow \mathbb{C}$ is nondegenerate, and that the surface~$\{(z(\zeta), \vartheta(\zeta)) : \zeta \in \mathbb{H}\} \subset \mathbb{R}^{2, 1}$ is a smooth, space-like surface. Here,~$\mathbb{R}^{2, 1} = \mathbb{C} \times \mathbb{R}$ is equipped with the metric~$|d z|^2 - d \vartheta^2$. Equation~\eqref{eqn:min_eq} will be used to show that the surface has zero mean curvature in~$\mathbb{R}^{2, 1}$, c.f. Lemma~\ref{lem:diff} below.

\begin{lemma}\label{lem:t_o_limit_calcs}
    Let~$z$ and~$\vartheta$ be given by~\eqref{eqn:tdef} and~\eqref{eqn:odef}, respectively. Then the following holds:
    \begin{enumerate}

        \item  For any $\zeta \in \mathbb{H}$
        \begin{equation}\label{eqn:min_eq}
            \partial_\zeta z(\zeta) \partial_\zeta \overline{z(\zeta)}-  \partial_\zeta \vartheta(\zeta) \partial_\zeta \overline{\vartheta(\zeta)} = 0 .
        \end{equation}
        \item For each  $\zeta \in \mathbb{H}$, and for any unit vector $v \in \mathbb{R}^2$,
        \begin{equation}\label{eqn:space_like}
        |\partial_v \vartheta(\zeta)|^2 < |\partial_v z(\zeta)|^2
        \end{equation}
        where $\partial_v$ denotes the directional derivative.
        \item In particular $z$ has no critical points and the surface $S_{\Hex} \coloneqq \{(z(\zeta), \vartheta(\zeta)) : \zeta \in \mathbb{H}\} \subset \mathbb{R}^{2, 1}$ is smooth and space-like.
    \end{enumerate}

    \end{lemma}
    \begin{proof}
    The proof of~\eqref{eqn:min_eq} is a computation that works as long as $f$ and $g$ are meromorphic functions on~$\mathbb{C}$ with all poles in~$\mathbb{R}$ and~$\gamma_{\zeta}$ never intersects the poles of~$f$ or~$g$. Recall $\overline{f}(z) \coloneqq \overline{f(\overline{z})}$, and similarly for $\overline{g}$. One can check that for~$z \in \mathbb{H}$

    \begin{align*}
     (2 \pi \i)^2   \partial_\zeta z(\zeta) \partial_\zeta \overline{z(\zeta)} &= f(\zeta)g(\zeta) \overline{f}(\zeta) \overline{g}(\zeta)\\
        &=  f(\zeta)\overline{g}(\zeta) \overline{f}(\zeta) g(\zeta) \\
        &= (2 \pi \i)^2   \partial_\zeta \vartheta(\zeta) \partial_\zeta \overline{\vartheta(\zeta)}
    \end{align*}
    which shows that \eqref{eqn:min_eq} holds for~$\zeta \in \mathbb{H}$, and thus proves (1).

    To show \eqref{eqn:space_like} we note that by \eqref{eqn:min_eq} it suffices to show \eqref{eqn:space_like} for $\partial_v = \partial_x$, using the coordinates~$\zeta = x + \i y$. Indeed, \eqref{eqn:min_eq} implies that the pullback of the Minkowski metric tensor under the map $\zeta \mapsto (z(\zeta), \vartheta(\zeta))$ is a multiple of the standard metric $|d \zeta|^2 $, i.e.~$|\partial_v z|^2 - |\partial_v \vartheta|^2$ is the same for all directions~$v$. First we observe that
    \begin{align}
        |\partial_x z|^2 - |\partial_x\vartheta|^2 
       &= |\partial_\zeta z + \partial_{\overline{\zeta}}z|^2-
        |\partial_{\zeta}\vartheta + \partial_{\overline{\zeta}}\vartheta|^2 \\
        &=  |\partial_{\zeta}z|^2  + \partial_{\overline{\zeta}}z \overline{\partial_{\zeta}z} + \partial_{\zeta}z \partial_{\zeta}\overline{z} + |\partial_{\overline{\zeta}}z|^2 -
        |\partial_{\zeta}\vartheta|^2  - \partial_{\overline{\zeta}}\vartheta \overline{\partial_{\zeta}\vartheta} - \partial_{\zeta}\vartheta \partial_{\zeta}\overline{\vartheta} - |\partial_{\overline{\zeta}}\vartheta|^2 \notag \\
        &= |\partial_{\zeta}z|^2 + |\partial_{\overline{\zeta}}z|^2 -
        |\partial_{\zeta}\vartheta|^2 - |\partial_{\overline{\zeta}}\vartheta|^2  \label{eqn:dx_expression}
    \end{align}
    where to get the last equation we have used \eqref{eqn:min_eq}. Next, we compute that 
    \begin{align*}
4 \pi^2 |\partial_{\zeta} z(\zeta)|^2 &= 
f(\zeta) g(\zeta) \overline{f(\zeta)} \overline{g(\zeta)} \\
4 \pi^2 |\partial_{\overline{\zeta}} z(\zeta)|^2 &= 
f(\overline{\zeta}) g(\overline{\zeta}) \overline{f(\overline{\zeta})} \overline{g(\overline{\zeta}) } \\
4 \pi^2 |\partial_{\zeta} \vartheta(\zeta)|^2 &= 
f(\zeta) \overline{g}(\zeta) \overline{f(\zeta)} \overline{\overline{g}(\zeta)} \\
4 \pi^2 |\partial_{\overline{\zeta}} \vartheta(\zeta)|^2 &= 
f(\overline{\zeta}) \overline{g}(\overline{\zeta}) \overline{f(\overline{\zeta})} \overline{\overline{g}(\overline{\zeta}) }.
    \end{align*}
From this we see that \eqref{eqn:dx_expression} is given by 
\begin{equation}\label{eqn:dx_exp_final}
\left(f(\zeta) \overline{f(\zeta)} - f(\overline{\zeta}) \overline{f(\overline{\zeta})} \right) \left( g(\zeta) \overline{g(\zeta)} -  g(\overline{\zeta}) \overline{g(\overline{\zeta})} \right).
\end{equation}
Then, we observe, by simplifying the formulas \eqref{eqn:g_def2} and \eqref{eqn:f_def2}, that 
\begin{align*}
    f(z)  &= - \frac{(z - e^{ \i 2 \pi/3})^2}{(z-1) (z+2) (z+\frac{1}{2})} \\
    g(z) &= -\frac{(z - e^{ \i 2 \pi/3})^2}{z (z+1)} .
\end{align*}
Since $e^{i 2\pi/3}$ is in the upper half plane, we now can see that \eqref{eqn:dx_exp_final} is strictly positive since~$|f(\overline{z})| > |f(z)|$ and~$|g(\overline{z})|> |g(z)|$. 

The discussion above concludes the proof of \eqref{eqn:space_like}, and this in turn implies the third claim of the lemma.
\end{proof}

\begin{remark}
In the above proof we only need the fact that~$f$ and~$g$ are meromorphic, and have all of their poles in~$\mathbb{R}$ and their zeros in the upper half plane.
\end{remark}

We now show that this surface~$S_{\Hex}$ has zero mean curvature and that the map~$\zeta \mapsto (z(\zeta), \vartheta(\zeta))$ is a conformal parameterization of this surface. We will also see that~$\overline{S_{\Hex}}$ has boundary equal to the contour~$C_{\Hex}$ in~$\mathbb{R}^{2,1}$ given by the union of line segments from~$(\mathcal{T}(v_j),\mathcal{O}(v_j))$ to~$(\mathcal{T}(v_{j+1}), \mathcal{O}(v_{j+1}))$,~$j=1,\dots, 6$ (recall~$\mathcal{T}(v_j) = -e^{\i j \pi/3}$ and~$\mathcal{O}(v_j) = (-1)^j\frac{1}{2}$, as described just before Theorem~\ref{thm:exact_hex}; also, here and below~$\mathcal{T}(v_{6+1})$ is taken to mean~$\mathcal{T}(v_1)$). We denote by~$\Hex$ the regular hexagon bounded by the~$6$ vertices~$\mathcal{T}(v_i)$.

\begin{lemma}\label{lem:diff}
     The following are true:
    \begin{enumerate}
        \item The map $\zeta \mapsto z(\zeta)$ is an orientation reversing diffeomorphism from $\mathbb{H} \rightarrow \Hex$. \label{item:diffeo}
    \item The map $\zeta \mapsto (z(\zeta), \vartheta(\zeta))$ is the conformal parameterization of the surface. \label{item:conformal}
    \item Each coordinate of the map $\zeta \mapsto (z(\zeta), \vartheta(\zeta))$ is harmonic. \label{item:harmonic}
    \end{enumerate}
    In particular, the surface $S_{\Hex}$ is a maximal surface in the Minkowski space $\mathbb{R}^{2, 1}$, with boundary contour~$C_{\Hex}$.
    
\end{lemma}
\begin{proof}

First we prove (3). The harmonicity of the functions $z$ and $\vartheta$ can be seen by writing them as a difference of holomorphic and anti-holomorphic functions in $\zeta$, e.g.
    \begin{align*}
    z(\zeta) =e^{- \i \frac{2 \pi}{3}} + \frac{1}{2 \pi \i} \left( \int_{-3}^{\overline{\zeta}} f(z) g(z) \; dz  - \int_{-3}^\zeta f(z) g(z) \; dz \right).
    \end{align*}
This proves \eqref{item:harmonic} above.

Next, we proceed to show (1) and (2), i.e. that $z : \mathbb{H} \rightarrow \Hex$ is a diffeomorphism, and that therefore $\zeta \mapsto (z(\zeta), \vartheta(\zeta))$ is a diffeomorphism.

The boundary behavior (described in the paragraph below), together with the harmonicity, implies that the image of the map $z$ as $\zeta$ varies over $\mathbb{H}$ is equal to the hexagon $\Hex$. Similarly, we will see that the surface~$\overline{S_{\Hex}}$ has boundary~$C_{\Hex}$.

Define the intervals 
\begin{align*}
    \mathcal{I}_1 &= (-\infty, -2) \\
    \mathcal{I}_2 &= (-2, -1) \\
    \mathcal{I}_3 &= \left( -1,-\frac{1}{2}\right) \\
    \mathcal{I}_4 &= \left(-\frac{1}{2}, 0\right) \\
    \mathcal{I}_5 &= (0, 1)\\
    \mathcal{I}_6 &= (1, \infty) .
\end{align*}
(These correspond to frozen regions via the critical point map~$\zeta(\chi, \eta)$~\eqref{eqn:crit_explicit}.) Then one can check, using the residue theorem, that for each $j$, $z$ maps $\mathcal{I}_j$ to $\mathcal{T}(v_{6-j+2}) = e^{-\i (j+1) \frac{\pi}{3}} $, and $\vartheta$ maps $\mathcal{I}_j$ to $(-1)^j \frac{1}{2}$. Furthermore, small semicircles around the points $-2,-1,-\frac{1}{2}, 0, 1, \infty $ map (approximately) to the edges of the boundary polygon $\partial \Hex$ under $z$, and to segments from $\pm \frac{1}{2}$ to $\mp \frac{1}{2}$ under $\vartheta$. Thus, as $\zeta$ varies counterclockwise around the boundary of the upper half plane, $z$ moves around the boundary polygon in the \emph{clockwise} direction, and $\vartheta$ correspondingly oscillates between $\pm \frac{1}{2}$. In particular $z$ winds around the boundary of $\Hex$ exactly once.

Since $z$ is nondegenerate (recall \eqref{eqn:space_like}), we may employ the argument principle for harmonic functions~\cite{DHL96} to conclude that $z : \mathbb{H} \rightarrow \Hex$ is a diffeomorphism, showing \eqref{item:diffeo}. 

This implies that $\zeta \mapsto (z(\zeta), \vartheta(\zeta))$ is a diffeomorphism, which proves (2). Indeed, equation~\eqref{eqn:min_eq} implies that $(z(\zeta), \vartheta(\zeta))$ is a conformal map to the surface $S_{\Hex}$, where $S_{\Hex}$ is equipped with the metric it inherits from $\mathbb{R}^{2, 1}$. (See the second paragraph of the proof of Lemma \ref{lem:t_o_limit_calcs} for details.)

Now we proceed to show the final claim, namely that~$S_{\Hex}$ is a maximal surface. We have shown that the map $(z, \vartheta)$ is a conformal parameterization of $S_{\Hex}$. Furthermore, we have showed that $z(\zeta)$ and $\vartheta(\zeta)$ are harmonic functions. From this, it follows that $S_{\Hex}$ has zero mean-curvature in $\mathbb{R}^{2, 1}$ (c.f. the discussion after Assumption 1.1 in~\cite{CLR2}).
\end{proof}

\subsection{Convergence}
\label{subsec:limits}

In the rest of this section, we state our main theorems about convergence of~$\mathcal{T}$ and~$\mathcal{O}$ to~$z$ and~$\vartheta$. An important object will be the \emph{action function}. This is essentially the same as the action function defined in~\cite[Definition 7.1]{Pet14} specialized to the setting of the regular hexagon with sidelength~$A$; it only differs because of our choice to use~$A$ as the large parameter in the asymptotic analysis of contour integrals rather than the parameter~$N = 2 A$ which is used there.

Throughout this section~$(\chi, \eta)$ denote rescaled limiting coordinates in~$\mathfrak{H} \coloneqq \lim_{A \rightarrow \infty} \frac{1}{A} H_A$, i.e.
\begin{equation}
(\chi, \eta) \in \mathfrak{H} \coloneqq  \{0 \leq \eta \leq 1    \text{ and } -\eta \leq \chi \leq 1 \} \cup \{1 \leq \eta \leq 2    \text{ and } -1 \leq \chi \leq 2-\eta \} . \label{eqn:limH_A}
\end{equation}

\begin{definition}[hexagon action function]\label{def:hex_action}
The \emph{action function}~$S(z; \chi, \eta)$ is defined, for~$z \in \mathbb{C}$ and~$(\chi, \eta) \in \mathfrak{H}$, by 
\begin{multline}\label{eqn:S_def_2}
S(z; \chi, \eta) \coloneqq -\eta - (-2 - z) \log(-2 - z) - (1 + z) \log(-1 - z) 
 \\
 - (-1 + z) \log(1 - z)
 + 
 z \log(-z)  \\
 + (z - \chi) \log(z - \chi) + (-2 +\eta - z + \chi) \log(2 - \eta + z - \chi).
\end{multline}
The logarithms all have cuts along the negative real axis. 
\end{definition}

Now we recall the definition of the \emph{critical point} of the action function. This is the same critical point used in the saddle point analysis of the kernel~\eqref{eqn:Kast_exact1}  in~\cite[Section 4.2]{Pet15} to compute GFF fluctuations. This will also show up naturally in the asymptotic analysis of the integral formulas for the t-embeddings~\eqref{eqn:T_exact_hex} and their origami maps~\eqref{eqn:O_exact_hex}. First, we give some discussion in order to give a concrete and precise definition of the critical point. As we will see, our discussion also leads us naturally to the definition of the \emph{liquid region} (also known as the \emph{rough region}), and the \emph{frozen} regions (c.f. \cite[Definition 7.6]{Pet14} and the proposition preceding it).

From differentiating the explicit formula \eqref{eqn:S_def_2} it is clear that for fixed $(\chi, \eta)$, the analytic function~$z \mapsto S(z;\chi, \eta)$ either has two complex conjugate critical points or two real critical points. In either case, the two critical points are solutions of the quadratic equation
\begin{align}\label{eqn:crit_eqn}
    z(z+2)(z - \chi) = (z-1)(z+1)(z-\chi-\eta+ 2)
\end{align}
which can be obtained from exponentiating~$\partial_z S(z; \chi, \eta) = 0$.

The critical points can be written explicitly as

\begin{align}\label{eqn:crit_explicit} 
    \frac{-1 + 2 \chi \pm \sqrt{1 - 8 \eta+ 4 \eta^2 - 4 \chi + 4 \eta \chi + 4 \chi^2}}{2 \eta} .
\end{align}

Now we define the region of the rescaled hexagon known as the liquid region.

\begin{definition}[Liquid and frozen regions]\label{def:liquid_frozen}
Denote by $\mathcal{D} \subset \mathfrak{H}$ the set of points 
$(\chi, \eta)$ where $S(z; \chi, \eta)$ has a complex conjugate pair of critical points. This is known as the \emph{liquid region} or the \emph{rough region}.

The connected components of the subset $\mathfrak{H} \setminus \overline{D}$ of the rescaled hexagon where $S(z; \chi, \eta)$ has two distinct real critical points are known as the \emph{frozen regions}.
\end{definition}

To describe the critical points in the frozen regions, we have the following lemma 

\begin{lemma}\label{lem:frozen_crit}
As a function of $z$, when the function $S(z; \chi, \eta)$ has two real critical points, they are in the same connected component of $\mathbb{R} \setminus \{-2, -1, -\frac{1}{2}, 0, 1\}$.
\end{lemma}
\begin{proof}
This can be seen by analyzing the possibilities for intersections of the graphs of $Q_1(z) = z (z+2) (z- \chi)$ and $Q_2(z) = (z-1)(z+1) (z-\chi - \eta +2)$ as $z$ varies along the real axis. A similar analysis appears in \cite[Proposition 7.6]{Pet14}, so we will be very brief.

We show that for each interval $\mathcal{I}_j$, if the first intersection of the curve $(z, Q_1(z))$ with $(z, Q_2(z))$ is in $\mathcal{I}_j$ (as $z$ moves from $-\infty$ along $\mathbb{R}$), then so is the other intersection. The situation for the first three intervals $\mathcal{I}_1 = (-\infty, -2)$, $\mathcal{I}_2 = (-2, -1)$ and $\mathcal{I}_3 = (-1,-\frac{1}{2})$ is symmetric to the situation for the other three intervals, so we only deal with the first three.

We can see that for $z << 0$, $Q_1(z) > Q_2(z)$, and also $Q_2(-2) < Q_1(-2) = 0$ (to see that~$Q_2(-2) < 0$, recall the domain of definition~\eqref{eqn:limH_A} of~$(\chi, \eta)$), so if $Q_1$ and $Q_2$ cross in $\mathcal{I}_1$, they must cross again before $z$ reaches $-2$. This gives the result for $\mathcal{I}_1$. A similar argument and the observation that $Q_2(-1) = 0$ yields the result for $\mathcal{I}_2$. If they first cross in $\mathcal{I}_3$, then the fact that $Q_2(-\frac{1}{2}) < Q_1(-\frac{1}{2})$ for any $(\chi, \eta)$, which we observe by explicit computation, yields that they must cross again before $z$ reaches $-\frac{1}{2}$.
\end{proof}

Now we have the following definition.
\begin{definition}[Critical point of the action function]\label{def:critical_point}
    For~$(\chi, \eta) \in \mathfrak{H}$, define the map~$\zeta(\chi, \eta)$ (which we also refer to as the \emph{critical point} of the action) by
\begin{equation}\label{eqn:critical_point_def}
\zeta(\chi, \eta)
= \frac{-1 + 2 \chi + \sqrt{1 - 8 \eta+ 4 \eta^2 - 4 \chi + 4 \eta \chi + 4 \chi^2}}{2 \eta}
\end{equation}
(c.f.~\eqref{eqn:crit_explicit}). The square root is chosen to be along the positive imaginary axis if its argument is negative, and chosen to be positive if its argument is positive.
\end{definition}
In other words, if
$(\chi, \eta)$ is in the \emph{liquid region} $\mathcal{D}$, then $\zeta(\chi, \eta)$ is defined as the unique solution of $\partial_z S(z; \chi, \eta) = 0$  
 which lies in the upper half plane (c.f. \cite[Section 7]{Pet14} and \cite[Section 4]{Pet15}). If $(\chi, \eta)$ is in the \emph{frozen region}, then both solutions lie in the same connected component of $\mathbb{R} \setminus \{-2, -1, -\frac{1}{2}, 0, 1\}$, and for concreteness, we define $\zeta(\chi, \eta)$ as the larger of these two critical points. If~$1 - 8 \eta+ 4 \eta^2 - 4 \chi + 4 \eta \chi + 4 \chi^2 = 0$ and the two critical points merge, then we are on the \emph{arctic curve} (c.f.~\cite[Section 7.7]{Pet14}); we do not consider this situation in this paper, although we expect our convergence result to hold on the arctic curve as well (with a different error bound).

\begin{remark}
For the purpose of Theorem \ref{thm:TLIM1} below, if $(\chi, \eta)$ is in a frozen region, then $\zeta(\chi, \eta)$ can in fact be taken as either critical point of the action. In fact, from the limiting expression it is clear that in that case $\zeta$ can be replaced by any real number in the same interval of $\mathbb{R} \setminus \{-2, -1, -\frac{1}{2}, 0, 1\}$ as $\zeta(\chi, \eta)$.
\end{remark}

Recall that we defined constants of integration for the t-embeddings and origami maps so that $\mathcal{T}(v_1) = -e^{\i \pi / 3}$ and $\Or(v_1 ) = -\frac{1}{2}$. We also recall the functions $z(\zeta)$ and $\vartheta(\zeta)$ defined in \eqref{eqn:tdef} and \eqref{eqn:odef}. We have the following limiting behavior for the perfect t-embedding $\mathcal{T} = \mathcal{T}_A$ and corresponding origami map $\mathcal{O} = \mathcal{O}_A$ as the size of the hexagon $A \rightarrow \infty$.

\begin{theorem}\label{thm:TLIM1}
    Fix a compact subset $K \subset \mathcal{D} \cup \left(\mathfrak{H} \setminus \overline{\mathcal{D}} \right)$ (recall that~$\mathcal{D}$ denotes the liquid region), i.e. bounded away from the arctic curve.
    Then for $(\chi, \eta) \in K$ we have the asymptotic behavior
\begin{align}\label{eqn:TLIM1}
\mathcal{T}(\lfloor \chi A \rfloor, \lfloor \eta A \rfloor) 
= z(\zeta(\chi, \eta)) + o(1)
\end{align}
and
\begin{align}\label{eqn:OLIM1}
    \mathcal{O}(\lfloor \chi A \rfloor, \lfloor \eta A \rfloor) 
    = \vartheta(\zeta(\chi, \eta)) + o(1) .
\end{align}
The $o(1)$ error terms have norm $\leq C \frac{1}{\sqrt{A}}$ for all $(\chi, \eta) \in K$, for some $C = C_K > 0$.
\end{theorem}
\begin{remark}
 As we will see in the proof in Section~\ref{subsec:main_saddle}, if~$K$ is completely contained in the frozen regions then the~$o(1)$ error can be improved to a uniform~$O(e^{-c A})$ error, for some~$c = c_K > 0$. Furthermore, the theorem implies that each of the six frozen regions are asymptotically collapsed to a point under~$\mathcal{T}$.
\end{remark}

\begin{remark}\label{rmk:OrPresDiff}
We see that the map~$\mathcal{D} \rightarrow \Hex$ given by
\begin{align*}
    (\chi, \eta) & \mapsto \lim_{A \rightarrow \infty}\mathcal{T}(\lfloor \chi A \rfloor, \lfloor \eta A \rfloor)
    = z(\zeta(\chi, \eta))
    \end{align*}
    is an orientation preserving diffeomorphism from the liquid region to the interior of the regular hexagon. Indeed,

    \begin{enumerate}
        \item By Lemma \ref{lem:diff}, the map $\zeta \mapsto z(\zeta)$ is an orientation reversing diffeomorphism from the upper half plane to the interior of the regular hexagon $\Hex$.
        \item By \cite[Proposition 4.3]{Pet15}, the map $(\chi, \eta) \mapsto \zeta(\chi, \eta)$ is a diffeomorphism from the liquid region $\mathcal{D}$ to the upper half plane. It is easy to check that this map is also orientation reversing, e.g. by explicitly checking the boundary behavior.
    \end{enumerate}

\end{remark}

Now we can prove that the graph $\{(z, \mathcal{O}(z)) : z \in \Hex\}$ (recall $\mathcal{O}$ is defined as a function of the complex coordinate $z \in \Hex$ in the domain covered by $\mathcal{T}$) converges to a maximal surface in $\mathbb{R}^{2, 1}$. Recall that as $\zeta$ moves along $\mathbb{R} = \partial \mathbb{H}$, the map $(z(\zeta), \vartheta(\zeta))$ traces out the piecewise linear curve $C_{\Hex} \subset \mathbb{R}^{2, 1}$ with vertices $(\mathcal{T}(v_j), (-1)^j \frac{1}{2}) =  (-e^{\i  j \pi /3}, (-1)^j \frac{1}{2})$.

\begin{corollary}\label{cor:orig}
    The origami maps converge 
    \begin{equation}\label{eqn:or_conv}
    \Or_A(z') \rightarrow \vartheta(z^{-1}(z'))
    \end{equation}
    uniformly on compact subsets of $\Hex$ as $A \rightarrow \infty$. 
\end{corollary}

\begin{remark}
    Recall that the graph of~$\vartheta(z^{-1}(z))$ is~$S_{\Hex}$, the maximal surface in the Minkowski space $\mathbb{R}^{2, 1}$ with boundary contour given by $C_{\Hex}$ (see Lemma~\ref{lem:diff}). Therefore, the corollary above means that the discrete surfaces $(\mathcal{T}, \mathcal{O})$ converge to a maximal surface in the Minkowski space $\mathbb{R}^{2, 1}$.
\end{remark}

\begin{proof}
This follows from Theorem \ref{eqn:TLIM1}, together with Lemma \ref{lem:diff}. For a detailed argument proving convergence of $\mathcal{O}$ on the domain of the embedding given the convergence of $\mathcal{T}$ to a diffeomorphism and convergence of $(\mathcal{T}, \mathcal{O})$ as functions on the liquid region, see \cite[Corollary 5.5]{BNR23}.
\end{proof}

We conclude this section with a consistency check. By~\cite[Theorem 1.4]{CLR2}, Theorem~\ref{thm:TLIM1} and the results of Section~\ref{subsec:ef_Lip} imply that the (gradients of) height fluctuations converge to (gradients of) a Gaussian free field. This convergence is already known; it is a special case of a result in~\cite{Pet15}. In~\cite{Pet15}, the underlying conformal structure on the liquid region which defines the Gaussian free field is defined by the critical point map~$\zeta : \mathcal{D} \rightarrow \mathbb{H}$ defined by~\eqref{eqn:critical_point_def}. On the other hand, by our results the limit of~$(\mathcal{T}(x, n), \mathcal{O}(x, n))$ (for~$(x, n) = (\lfloor \chi A \rfloor,\lfloor \eta A \rfloor) $) defines a conformal structure on~$\mathcal{D}$ (see Remark~\ref{rmk:OrPresDiff}), and by~\cite{CLR2} this conformal structure also describes the limiting GFF.

Since~$(z(\zeta), \vartheta(\zeta))$ is a conformal isomorphism from the upper half plane to~$S_{\Hex}$, and since the limit of~$(\mathcal{T}(x, n), \mathcal{O}(x, n))$ is the composition of~$(z(\zeta), \vartheta(\zeta))$ with~$\zeta(\chi, \eta)$, the conformal structure defined on~$\mathcal{D}$ from the t-embeddings and origami maps agrees with the one defined by~$\zeta(\chi, \eta)$. In other words, the description of the limiting height fluctuations obtained via t-embeddings using the result of~\cite{CLR2} is consistent with the result of~\cite{Pet15}.

\section{Structural Rigidity}
\label{subsec:ef_Lip}
Now we provide a proof of the rigidity assumptions of Proposition~\ref{prop:rigidityCLRthm}. More precisely, we need to check that in compact subsets~$K$ of the liquid region~$\mathcal{D}$, lengths of edges contained in~$K$ are uniformly of length~$\frac{1}{A}$ and angles~$\theta$ at a vertex inside~$K$ are uniformly bounded away from~$0$ and~$\pi$. To prove the rigidity, we use asymptotics of the gauge functions, given in Lemma~\ref{lem:expansions} below. The saddle point analysis needed to obtain  these asymptotics is given in Section~\ref{sec:steep_descent}.

Recall the lattice coordinates~$(x, n)$ defined in Section~\ref{subsubsec:reduced_hex}, and that we denote rescaled coordinates by~$(\chi, \eta) = (\frac{x}{A} , \frac{n}{A} )$. It is possible to see from a saddle point analysis of \eqref{eqn:Fexact} (using Lemma \ref{lem:tilde_f}) and of \eqref{eqn:Gexact} that we have the following lemma.  

\begin{lemma}\label{lem:expansions}
    Let $(\chi, \eta) = (\frac{x}{A}, \frac{n}{A})$ and~$\zeta = \zeta(\chi, \eta)$. Then there is a rational function $h$ with zeros and poles only on the real axis such that
\begin{align*}
\mathcal{F}^\bullet(x, n) &\approx \frac{1}{\sqrt{A}} \left( e^{A S(\zeta; \chi, \eta)} f(\zeta) \sqrt{\frac{\partial \zeta}{\partial \chi} h(\zeta)} + e^{A \overline{S(\zeta; \chi, \eta)}} f(\overline{\zeta}) \sqrt{\overline{\frac{\partial \zeta}{\partial \chi}} h(\overline{\zeta})} \right)
\end{align*}
and
\begin{align*}
    \mathcal{F}^\circ(x, n) &\approx \frac{1}{\sqrt{A}} \left( e^{-A S(\zeta; \chi, \eta)} g(\zeta) \sqrt{\frac{\partial \zeta}{\partial \chi} \frac{1}{h(\zeta)}} + e^{-A \overline{S(\zeta; \chi, \eta)}} g(\overline{\zeta}) \sqrt{\overline{\frac{\partial \zeta}{\partial \chi}} \frac{1}{h(\overline{\zeta})}} \right)
\end{align*}
where $\approx$ denotes equality up to multiplicative $1 + o(1)$ error uniformly in compact subsets of the liquid region. The functions $f$ and $g$ are given by \eqref{eqn:f_def2} and \eqref{eqn:g_def2}.
\end{lemma}

This can be proved for any $(\chi, \eta)$ with an argument similar to that in the proof of Theorem \ref{thm:TLIM1} (which is given in Section~\ref{sec:steep_descent}), but in fact it suffices to only check this on ``half'' of the hexagon $\eta + 2 \chi \leq 1$ by the reflectional symmetry of the perfect t-embeddings~$\mathcal{T}_A$. For $\eta + 2 \chi \leq 1$, it is straightforward to prove the asymptotics in the lemma above, as it is not necessary to repeat the argument to bound the extra integral in the exact formula for $\mathcal{F}^{\bullet}$ which comes from the error term in Lemma~\ref{lem:tilde_f} (c.f. \eqref{eqn:fA_approx} as well).

\begin{remark}
Note that if one could provide some control on $\mathcal{F}^\circ$ and $\mathcal{F}^\bullet$ in the frozen region (here their product should be exponentially decaying) and near the arctic curve, it would be possible to deduce Theorem \ref{thm:TLIM1} from Lemma \ref{lem:expansions} by summing up the increments $d \mathcal{T}$ along a path from $v_1^*$ to a point $(x, n)$ in the bulk. 
\end{remark}

\begin{proposition}\label{prop:edges}
Let $K$ be a compact subset of the liquid region~$\mathcal{D}$. For dual edges~$e^*$ of the hexagon~$ H_A$ such that~$\frac{1}{A} e^* \in K$, the edge lengths~$|d\mathcal{T}(e^*)|$ of dual edges~$e^*$ adjacent to $(x, n)$ satisfy 
$$\frac{1}{C A} < |d\mathcal{T}(e^*)| < \frac{C}{A}$$
for some $C = C_K > 0$. 
\end{proposition}

\begin{proof}
To see that $d \mathcal{T}(e^*)$ satisfies our desired bound we do the following. We will show that in fact both $\mathcal{F}^\bullet(x, n)$ and $\mathcal{F}^\circ(x, n)$ are uniformly of the order $\frac{1}{\sqrt{A}}$, up to the factor $ e^{\pm A \Re S} $, where~$S = S(\zeta; \chi, \eta)$, for $(\chi, \eta) = (\frac{x}{A}, \frac{n}{A}) \in K$. 

We argue for $\mathcal{F}^\circ$ first, and the argument for $\mathcal{F}^\bullet$ is similar so we will omit details. We must show that the expression below, which we obtain from Lemma~\ref{lem:expansions}
\begin{align*}
e^{-A \Re{S}} \sqrt{A}\mathcal{F}^\circ(x, n) &\approx e^{- \i A \Im[S]}  \left(  g(\zeta) \sqrt{\frac{\partial \zeta}{\partial \chi} \frac{1}{h(\zeta)}} + e^{\i 2 A  \Im S} g(\overline{\zeta}) \sqrt{\overline{\frac{\partial \zeta}{\partial \chi}} \frac{1}{h(\overline{\zeta})}} \right) 
\end{align*}
is uniformly bounded from above and below. Ignoring the prefactor on the right hand side (which is a unit length complex number), we have 
\begin{align*}
\frac{g(\overline{\zeta}) \sqrt{\frac{\partial \zeta}{\partial \chi}}}{\sqrt{h(\zeta)}}\left(  \frac{g(\zeta)}{g(\overline{\zeta})}  + e^{\i 2 A  \Im[S]}  \sqrt{\frac{\overline{\frac{\partial \zeta}{\partial \chi}}}{\frac{\partial \zeta}{\partial \chi}} \frac{h(\zeta)}{h(\overline{\zeta})}} \right) .
\end{align*}
The prefactor~$\frac{g(\overline{\zeta}) \sqrt{\frac{\partial \zeta}{\partial \chi}}}{\sqrt{h(\zeta)}}$ is uniformly bounded away from $0$ and $\infty$ for $(\chi, \eta) \in K \subset \mathcal{D}$, because (a)~$\zeta : \mathcal{D} \rightarrow \mathbb{H}$ is a diffeomorphism, (b) the zeros of~$g$ are in the upper half plane, and (c) the poles of~$g$ and poles and zeros of~$h$ are on the real line. Furthermore, since $g$ has poles on the real line and zeros only in the upper half plane, we have 
$$\left|\frac{g(\zeta)}{g(\overline{\zeta})}\right| < c < 1$$
 for some $0 < c < 1$ for all $(\chi, \eta) \in K$. Since the second term in parentheses is a unit length complex number, this completes the proof for $\mathcal{F}^\circ$.

For $\mathcal{F}^\bullet$ we may argue in exactly the same way, since $f$ is also a rational functions whose poles are all on the real line and zeros are all in the upper half plane.

Finally, to complete the proof it suffices to note that to compute~$d \mathcal{T}(e^*)$ we will (depending on the type of dual edge~$e^*$) compute products of the form~$\cF^{\bullet}(x, n)\cF^{\circ}(x', n' )$, with~$x = x'$ and~$n = n'$, or with~$x' = x $ and~$n' = n - 1$, or with~$x' = x +1$ and~$n' = n - 1$. The change~$(x, n) \rightarrow (x', n')$ in~$\mathcal{F}^\circ(x, n)$ changes the prefactor of~$e^{A \Re S}$ by at worst a uniformly upper and lower bounded factor, as~$(\chi, \eta) = (\frac{x}{A}, \frac{n}{A})$ ranges over~$K$. Therefore, up to uniformly bounded prefactors, the prefactor~$e^{A \Re S}$ cancels with the prefactor~$e^{-A \Re S}$ in products of the form~$\cF^{\bullet}(x, n)\cF^{\circ}(x', n' )$.
\end{proof}

\begin{proposition}\label{prop:angles}
Let $K$ be a compact subset of $\mathcal{D}$. There exists~$\epsilon = \epsilon_K > 0$, such that at vertices~$(x, n)$ with~$(\frac{x}{A}, \frac{n}{A}) \in K$, angles of $\mathcal{T}$ are in the interval $(\epsilon, \pi- \epsilon)$.
\end{proposition}
\begin{proof}
 We have the symmetry of~$\mathcal{T}_A$ under a rotation of~$H_A'$ by $\pi/3$ and swapping black and white vertices (see Figures~\ref{fig:hex_reduction} and~\ref{fig:hex_temb}), as indicated in the proof of Proposition~\ref{prop:T_A}; thus, we only have to check the angle condition on the angles adjacent to white vertices (of the primal graph).

Note that an angle of a white face can be written as 
$$\arg\left( \frac{d\T(w b^*)}{d \T((wb')^*)} \right) 
= \arg \left( \frac{\mathcal{F}^\circ(w) \mathcal{F}^\bullet(b)}{\mathcal{F}^\circ(w) \mathcal{F}^\bullet(b')} \right) = 
\arg \left( \frac{\mathcal{F}^\bullet(b)}{\mathcal{F}^\bullet(b')} \right), $$
so we want to study the ratios of the form 
\begin{align}
 & \frac{\mathcal{F}^\bullet(x, n)}{\mathcal{F}^\bullet(x-1, n+1)} \label{eqn:rat1}\\
 & \frac{\mathcal{F}^\bullet(x, n)}{\mathcal{F}^\bullet(x, n+1)} \label{eqn:rat2}\\
 & \frac{\mathcal{F}^\bullet(x+1, n)}{\mathcal{F}^\bullet(x, n)}. \label{eqn:rat3}
\end{align}

It suffices to show that these ratios have arguments uniformly bounded away from $0$ and $\pm \pi$. 

Repeating the manipulations in the proof of the Proposition~\ref{prop:edges} (with $\mathcal{F}^\bullet$ instead of $\mathcal{F}^\circ$) we see that 
\begin{align}
 \frac{\mathcal{F}^\bullet(x, n)}{\mathcal{F}^\bullet(x-1, n+1)}  = \frac{ \left(  \frac{f(\zeta)}{f(\overline{\zeta})}  + e^{-\i 2 A  \Im[S]}  \sqrt{\frac{\overline{\frac{\partial \zeta}{\partial \chi}}}{\frac{\partial \zeta}{\partial \chi}} \frac{h(\overline{\zeta})}{h(\zeta)}} \right)}{ (\overline{\zeta}-\chi) \left(  \frac{f(\zeta) (\zeta-\chi)}{f(\overline{\zeta}) (\overline{\zeta}-\chi)}  + e^{-\i 2 A  \Im[S]}  \sqrt{\frac{\overline{\frac{\partial \zeta}{\partial \chi}}}{\frac{\partial \zeta}{\partial \chi}} \frac{h(\overline{\zeta})}{h(\zeta)}} \right)}. \label{eqn:expr}
\end{align}
The extra factors of $\zeta - \chi$ in the denominator comes from Taylor expanding~$S(\zeta; \chi - \frac{1}{A}, \eta + \frac{1}{A})$ in the exponent of~$e^{A S(\zeta; \chi - \frac{1}{A}, \eta + \frac{1}{A})}$ and observing that $e^{-S_{\chi} + S_{\eta}} = \zeta - \chi$. Now we again note that 
$$\gamma \coloneqq - e^{-\i 2 A  \Im[S]}  \sqrt{\frac{\overline{\frac{\partial \zeta}{\partial \chi}}}{\frac{\partial \zeta}{\partial \chi}} \frac{h(\overline{\zeta})}{h(\zeta)}} $$ 
is a unit length complex number. Furthermore, we know that 
$$\left|\frac{f(\zeta)}{f(\overline{\zeta})}\right| < c < 1$$ 
for all $(\chi, \eta) \in K$ for some $0 < c < 1$. Finally, note that $\arg(\frac{\zeta - \chi}{\overline{\zeta} - \chi}) = -2 \arg(\overline{\zeta} - \chi)$ (possibly up to $\pm 2 \pi$). Therefore, up to a real constant (which is uniformly bounded away from $0$ and $\infty$), \eqref{eqn:expr} has the form 
\begin{align}\label{eqn:expr2}
 \frac{\delta - \gamma}{ e^{- \i \theta} (\delta e^{\i 2 \theta} - \gamma)} 
\end{align}
for some $|\gamma| = 1$ (whose argument is unknown and will in fact oscillate as $A$ varies), $\theta \in (0, \pi)$ uniformly bounded away from $0$ and $\pi$, and for $\delta$ satisfying the uniform bound $0 < |\delta| < c < 1$. Note that if we exactly had $|\delta| = 1$, then this quantity would be exactly equal to $1$ by the fact that the angle subtended by an arc (from the center of the unit disk) is twice the angle subtended from a point on the circle. It is an exercise in elementary geometry to show that $0 < |\delta| < c < 1$ indeed guaruntees that the argument of \eqref{eqn:expr2} is uniformly bounded away from $0$ and $\pm \pi$.

For \eqref{eqn:rat2} and \eqref{eqn:rat3} there is a very similar reduction to studying the argument of an expression of the form \eqref{eqn:expr2}. In these cases, the factor $(\zeta - \chi)$ will be replaced by a different rational function, but still a rational function with \emph{real} coefficients, so the exact same argument goes through at each step. This completes the proof. 
\end{proof}

Now from Propositions~\ref{prop:edges} and~\ref{prop:angles} we can deduce that the \emph{rigidity conditions} are satisfied by~$\mathcal{T}_A$. As discussed in Section~\ref{sec:bg}, see Proposition~\ref{prop:rigidityCLRthm}, these imply the assumptions of~\cite[Theorem 1.4]{CLR2} are satisfied, and thus that dimer height fluctuations converge to a Gaussian free field whose underlying metric is that of the maximal surface~$S_{\Hex}$. 



\section{Saddle Point Analysis}
\label{sec:steep_descent}
In this section we carry out the saddle point arguments necessary to prove Theorem~\ref{thm:TLIM1}. Namely, we identify the limiting behavior of the functions~$f_A$ and~$g_A$ defined in~\eqref{eqn:fdef} and~\eqref{eqn:gdef}, and then we use this to take a limit of the formulas in Theorem~\ref{thm:exact_hex}. First, in Section~\ref{subsec:ftilde} we must perform a preliminary asymptotic analysis of the behavior of the function~$f_A$: This is the content of Lemma~\ref{lem:tilde_f}. The function~$f_A$ approximates the holomorphic function~$f$ given in~\eqref{eqn:f_def2}, up to an error term which does not contribute to the asymptotics of~$\mathcal{T}_A$ or~$\mathcal{O}_A$; see \eqref{eqn:fA_approx}. In Section~\ref{subsec:main_saddle} we perform the asymptotic analysis necessary to prove the main theorem. Bounding the error from the extra term in~\eqref{eqn:fA_approx} coming from the error term in Lemma~\ref{lem:tilde_f} requires care in the proof.

\subsection{A preliminary lemma}
\label{subsec:ftilde}

For the purposes of the next lemma, define 

\begin{equation}\label{eqn:tilde_f}
\tilde{f}(w) \coloneqq \frac{1}{2 \pi \i} \int_{\{0, 1,\dots,A-1\}} 
     \frac{1}{(-2A-z)_A (-z)_A} \frac{d z}{w - z}.
\end{equation}
By definition, the contour contains the points $\{0,1,\dots, A-1\}$ and does not contain $w$. This is the second term of the function $f_A$ appearing in the integrand in the contour integral formula for $\mathcal{F}^\bullet$ in Proposition \ref{prop:contour_int}, and appearing in the exact formulas for~$\T$ and~$\mathcal{O}$ in Theorem~\ref{thm:exact_hex}. Note that $\tilde{f}(A w)$ is analytic in $w$ away from $[0, 1]$. 

We compute the asymptotic behavior of $\tilde{f}(A w)$ as $A \rightarrow \infty$ in order to ultimately compute the asymptotic behavior of $\mathcal{T}$ and $\mathcal{O}$.

\begin{lemma}\label{lem:tilde_f}
    We have the asymptotic behavior as~$A \rightarrow \infty$
    \begin{multline}\label{eqn:tilde_f_approx}
        2 (2 A)_A  A ! \tilde{f}(A w) =  -  \frac{1}{w - (-\frac{1}{2})}\left(1 + O(\frac{1}{A}) \right) \\
        +O(\sqrt{A}) \cdot \mathbf{1}\{\Re[w] > -1/2\}  e^{A(S(-\frac{1}{2}; 0, 2)- S(w; 0, 2) )}
    \end{multline}
   The implicit constants in the $O(\frac{1}{A})$ error can be taken uniform for $w$ in compact subsets of $\mathbb{C} \setminus \{-2, -1, -\frac{1}{2}, 0, 1\}$, and the $O(\sqrt{A})$ can be taken uniform for $w$ in compact subsets of $\mathbb{C} \setminus ([-2,-1] \cup [0, 1])$.
    \end{lemma}
    \begin{proof}
    We simply perform a saddle point analysis on the integral formula. We start by rewriting the formula as
    \begin{align} \label{eqn:f_tilde2}
    \tilde{f}(A w) =  \frac{1}{2 \pi \i} \int_{\{[0, 1]\}} \frac{1}{(-2A-A z)_A (-A z)_A} \; \frac{dz }{w-z}  
    \end{align}
     which is obtained by making the integration variable change~$z \rightarrow A z$. The new contour is a closed contour containing the interval $[0, 1]$ on the real axis and not containing $w$ or any point in $(-\infty, -1/2)$.

    We can see via Stirling's approximation for Gamma functions that the integrand is equal to 
    \begin{equation}\label{eqn:stirling}
     A^{-2 A} \frac{1}{w-z} \sqrt{\frac{-1 + z}{z}} \sqrt{\frac{1 + z}{2 + z}} \exp \left(-A S(z; 0, 2)  \right) \bigg(1 + O(1/A) \bigg).\end{equation}
For convenience, we recall here that 
    \begin{multline}\label{eqn:S_02}
        S(z; 0, 2) \coloneqq -2 +(2 + z) \log(-2 - z) \\- (1 + z) \log(-1 - z) -
     (-1+z) \log(1- z) + z \log(-z).
    \end{multline}
     
     We have chosen branches of the logarithms in the approximation \eqref{eqn:stirling} in such a way that the $z$ derivative of the expression \eqref{eqn:S_02} is analytic away from $[-2, -1] \cup [0,1] \subset \mathbb{R}$. In particular, $\partial_z S(z; 0, 2)$ is analytic in a neighborhood of $z = -\frac{1}{2}$. The critical point for~$S(z; 0, 2)$ occurs at~$z = -\frac{1}{2}$. One can check that the steepest descent integration contour is a vertical line through~$z=-\frac{1}{2}$. Noting that the integrand is decaying faster than $\frac{1}{|z|^2}$ at~$\infty$, we can indeed drag the contour to this position. Note that if $w$ is to the right of the vertical line $\Re[z] = -\frac{1}{2}$, then as we drag the contour to this position, we will pick up a residue from the factor $\frac{1}{w-z}$. This leads to the second term in \eqref{eqn:tilde_f_approx}. If $w$ is exactly on the line $\{\Re[z] = -\frac{1}{2}\}$, we can deform the steepest descent contour slightly to avoid picking up this residue, while preserving the properties of the contour described in the next paragraph.
     
     On this contour, $\Re[S(z; 0, 2) -S(-\frac{1}{2}; 0, 2)] < 0$ if $z \neq -\frac{1}{2}$, so the integral is dominated by the contribution from the critical point. In particular, for a small enough $\delta > 0$ (which is independent of $A$), the integral along the part of this contour which is outside of a size $\sqrt{A}^{-1 + \delta}$ neighborhood will have a negligible contribution. 
     

     Thus, we evaluate the leading order behavior of the remaining integral in the following standard way: we throw away the part outside of a size $\sqrt{A}^{-1 + \delta}$ neighborhood of the critical point, make the variable change $z = \frac{1}{\sqrt{A}} \tilde{z}$, taylor expand the integrand around the critical point, and compute the Gaussian integral appearing as part of the leading order coefficient. For this, we use the fact that~$-\partial_z^2 S(z; 0, 2) = \frac{8}{3}$. All in all we are left with 
     $$- \frac{1}{\sqrt{2 \pi \cdot A \frac{8}{3}}}  A^{-2 A} \frac{1}{w-(-\frac{1}{2})} \exp \left(- A S(-\frac{1}{2}; 0, 2) \right) \bigg(1 + O(1/A) \bigg)$$
     for the leading order contribution after deforming the contour in \eqref{eqn:f_tilde2} to $\Re[z] = -\frac{1}{2}$.
     
     Further simplifying gives
     \begin{align*}
    -\frac{1}{\sqrt{2 \pi} \sqrt{ A \frac{8}{3}}} (4/27)^A e^{2 A} A^{-2 A} \frac{1}{w-(-\frac{1}{2})} \bigg(1 + O(1/A) \bigg) .
     \end{align*}
     One can check that up to multiplicative~$O(1/A)$ error the prefactor is the same in magnitude as 
    $$\frac{1}{ 2 (2 A)_A A!} .$$
    \end{proof}

\subsection{Proof of Theorem~\ref{thm:TLIM1}}
\label{subsec:main_saddle}

Now we are prepared for the proof of Theorem~\ref{thm:TLIM1}. Before going into the proof, we give some notation. Recall \eqref{eqn:gdef}, which defines the holomorphic function $g_A$ that appears in the integrand of $\mathcal{F}^{\circ}$ and thus in the integrand of $\mathcal{T}$ (c.f. Theorem \ref{thm:exact_hex}). We note that 

    \begin{align*}
        g_A(A z) &= -1 +  \frac{A}{1+A z} e^{\i \pi/3}  +\frac{A}{-A-A z}  e^{-\i \pi/3}   \\
        &= -1 +  \frac{1}{z} e^{\i \pi/3}  -\frac{1}{z+1}  e^{-\i \pi/3} + O(\frac{1}{A}).
    \end{align*}
    So recalling~$g$ from \eqref{eqn:g_def2} defined by
    \begin{align*}
        g(z) \coloneqq -1 +  \frac{1}{z} e^{\i \pi/3}  -\frac{1}{z+1}  e^{-\i \pi/3}
    \end{align*}
    we get that 
    \begin{equation}\label{eqn:Rg_def}
        g_A(A z) = g(z) + R_g(z)
    \end{equation}
    where the error term~$R_g(z)$ is~$O(\frac{1}{A})$ uniformly on compact subsets $K\subset \mathbb{C}\setminus \{0,-1\}$.

    Next, recall \eqref{eqn:fdef}, which defines the other function $f_A$ in the integrand of the formula for $\mathcal{T}$, and recall Remark~\ref{b_ne_a} where~$\Delta(A)$ is given. In a similar way as for~$g_A$, except now also employing Lemma \ref{lem:tilde_f}, we are led to compute
    \begin{align}
       \frac{A (2 A-1)!}{\Delta(A)} f_A(A z) &= \frac{A (2A-1)!}{\Delta(A)} \bigg( \i e^{\i \pi/6}  \frac{1}{(2 A)_A (A-1)!}\frac{1}{A z-(A-1)} \notag\\
        &
        -\frac{1}{2 (2A)_A  A !} \frac{1}{z - (-\frac{1}{2})}\left(1 + O(\frac{1}{A}) \right) -  \i e^{-\pi \i/6} \frac{1}{(2 A)_A (A-1)!}  \; \frac{1 }{A z+2 A}  \bigg) \notag \\
        &+ O(\sqrt{A}) \cdot \mathbf{1}\{\Re[z] > -1/2\}  e^{A(S(-\frac{1}{2}; 0, 2)- S(z; 0, 2) )}   \notag \\
        &=  f(z) + O(\frac{1}{A}) + O(\sqrt{A}) \cdot \mathbf{1}\{\Re[z] > -1/2\} e^{A(S(-\frac{1}{2}; 0, 2)- S(z; 0, 2) )}  \label{eqn:fA_approx}
     \end{align}
     where 
    \begin{align*}
      f(z) \coloneqq \frac{2}{3} \left( \i e^{\i \pi/6} \frac{1}{z - 1} - \frac{1}{2}\frac{1}{z - (-\frac{1}{2})}
      -\i e^{-\i \pi/6} \frac{1}{z + 2} \right)
      \end{align*}
is as defined in~\eqref{eqn:f_def2}.
      
    More precisely, defining~$R_{f, 1} (z)$ as the~$O(\frac{1}{A})$ error term in~\eqref{eqn:fA_approx} and~$R_{f, 2} (z)$ as the second error term in~\eqref{eqn:fA_approx}, we have
    \begin{align}
        \frac{A (2 A-1)!}{\Delta(A)}   f_A(A z) = f(z) + R_{f, 1} (z) + R_{f, 2}(z) .\label{eqn:Rf12_def}
    \end{align}
The error terms~$R_{f, 1}$ and~$R_{f, 2}$ can be bounded as follows. First, $R_{f,1}(z) = O(\frac{1}{A})$ uniformly on compact subsets $K \subset \mathbb{C}\setminus \{-2, -\frac{1}{2}, 0, 1\}$. Second, we may write $R_{f,2}(z) = R_{f, 3}(z) e^{A(S(-\frac{1}{2}; 0, 2)- S(z; 0, 2) )}$, for some $R_{f,3} =  O(\sqrt{A})$ uniformly on compact subsets $K \subset \mathbb{C}\setminus ([-2,-1] \cup  [0, 1])$. We note that the function $R_f(z) = R_{f, 1} (z) + R_{f, 2}(z)$ is analytic on $\mathbb{C}\setminus \left( \{-2, -\frac{1}{2}\} \cup [0, 1] \right)$, whereas $R_{f, i}$, $i = 1,2$, are not; each of these two has a discontinuity at the line $\Re[z] = -\frac{1}{2}$.

With this setup, we are now ready to prove Theorem~\ref{thm:TLIM1}. Namely, we will show that we have the  limiting behavior~\eqref{eqn:TLIM1} and~\eqref{eqn:OLIM1} for the perfect t-embedding $\mathcal{T} = \mathcal{T}_A$ and corresponding origami map $\mathcal{O} = \mathcal{O}_A$ as the size of the hexagon $A \rightarrow \infty$. First, we prove the theorem in the case when our compact subset $K \subset \mathcal{D}$ is inside of the liquid region. Even just to prove convergence in the liquid region we still must perform a steepest descent argument in one of the frozen regions. Therefore, the convergence in the frozen regions (which is a part of Theorem~\ref{thm:TLIM1}) will actually follow from the proof below together with a symmetry argument, which is given at the end of this section.

\begin{proof}[Proof of Theorem~\ref{thm:TLIM1}: Liquid region case]

The proof goes through a steepest descent analysis of the formulas in Theorem~\ref{thm:exact_hex}. Parts of this are very similar to that of Sections 7.1-7.5 of \cite{Pet14} and Section 4 of \cite{Pet15}; however, the term in the second line of \eqref{eqn:tilde_f_approx} requires additional analysis. We start with the formula $\mathcal{T}(x, n) = \mathcal{T}(v_1^*) + \mathcal{F}^\bullet(b_1) \mathcal{F}^\circ(w_1) + \text{(double integral)}$ for $\mathcal{T}$ from \eqref{eqn:T_exact_hex}. We can clearly see that $|\mathcal{F}^\bullet(b_1) \mathcal{F}^\circ(w_1)| = \Delta(A)$ (c.f. \eqref{eqn:Fform2}, \eqref{eqn:Gform2}, and \eqref{eqn:Delta_eq}) is negligible, so we will proceed to analyze the double integral. We recall the double integral here, after the change of variable $z_1 \rightarrow A z_1$, and $z_2 \rightarrow A z_2$, and noting that we can choose the same $z_2$ contour after the variable change:
\begin{multline}\label{eqn:T_exact_restate}
\frac{A (2A-1)!}{\Delta(A)}     \frac{1}{(2 \pi \i)^2} \int_{\{[0,1]\}} \int_{\{\infty\}}  \frac{dz_2 dz_1}{z_1 - z_2} \; \left(
    \frac{(A z_1-x+1)_{2A-n}}{(A z_2-x+1)_{2A-n}}  
- \frac{(A z_1+1)_{2A-1}}{(A z_2+1)_{2A-1}}  
\right) 
   \\
\times    \frac{ (-2A-A z_2)_{A} (-A z_2)_{A}}{(-2A-A z_1)_{A} (-A z_1)_{A}} f_A(A z_2) g_A(A z_1).
\end{multline}
The $z_1$ contour (denoted by $\{[0,1]\}$ above) is a counterclockwise loop surrounding the interval $[0,1]$ but not containing any point from $[-2,-1]$, and the $z_2$ contour is large enough and contains the $z_1$ contour.

Now, we compute asymptotics of the integral 
\begin{multline}\label{eqn:main_int}
    \frac{A (2A-1)!}{\Delta(A)}   \frac{1}{(2 \pi \i)^2} \int_{\{[0,1]\}} \int_{\{\infty\}}  \frac{dz_2 dz_1}{z_1 - z_2} \;
     \frac{(A z_1-x+1)_{2A-n}}{(A z_2-x+1)_{2A-n}}  
    \\
 \times    \frac{ (-2A-A z_2)_{A} (-A z_2)_{A}}{(-2A-A z_1)_{A} (-A z_1)_{A}} f_A(A z_2) g_A(A z_1)
 \end{multline}
 as the other term is this same integral evaluated at $n = 1$.

 We start by noting that the integrand is equal to
 \begin{align*}
     e^{A (S(z_2; \frac{x}{A}, \frac{n}{A} ) - S(z_1; \frac{x}{A}, \frac{n}{A} )) + O(\frac{1}{A})} \frac{h(z_2)}{h(z_1)} \frac{1}{z_1-z_2} f_A(z_2) g_A(z_1)
 \end{align*} 
 where $S$ is the action function of Definition \ref{def:hex_action}, and where $h(z)$ is analytic away from~$\mathbb{R}$. The fact we will use, which can be checked from the explicit expression
 $$h(z) = \sqrt{\frac{z-\frac{x}{A}-\frac{n}{A} +2 + \frac{1}{A}}{z - \frac{x}{A} +\frac{1}{A}}} \sqrt{\frac{z-1}{z}}\sqrt{\frac{z+1}{z+2}}$$
 is that~$\left| \frac{h(z_2)}{h(z_1)} \right|$ is upper bounded uniformly in $A$ as long as~$z_1$ and~$z_2$ are both bounded away from~$\{\chi, \chi + \eta -2, -2,-1,0,1\}$.

\textbf{Case 1:} We first assume $(\frac{x}{A}, \frac{n}{A})  \in \mathcal{D}$, i.e. we are at a lattice site near a point in the liquid region.

 We continue to denote $(\chi, \eta) = ( \frac{x}{A}, \frac{n}{A})$. We use $\zeta = \zeta(\chi, \eta)$ to denote the critical point of $S(z; \chi, \eta )$, with fixed $(\chi, \eta)$. As described in \cite[Proposition 7.9]{Pet14}, we may deform contours to new contours $C_1, C_2$, which pass through the critical points $\zeta, \overline{\zeta}$ in such a way that $\Im[S(z; \chi, \eta)] = \Im[S(\zeta; \chi, \eta)]$ or $\Im[S(z; \chi, \eta)] = \Im[S(\overline{\zeta}; \chi, \eta)]$ along both contours, and
 \begin{align} \label{eqn:new_contour_relation}
 \Re[S(z_2; \chi, \eta)] < \Re[S(\zeta; \chi, \eta)] < \Re[S(z_1; \chi, \eta)] \quad z_1 \in C_1, z_2 \in C_2 ; z_1,z_2 \neq \zeta .
 \end{align}
This can be done in the same way as described in \cite[Section 7.5]{Pet14}, even with the factor $g_A(z_1) f_A(z_2)$ in the integrand, since in order to acheive this the $z_2$ contour does not have to be dragged through any point in a small neighborhood of $[\chi-2+\eta, \chi]$ (this interval approximately contains the set of poles coming from the factor~$\frac{1}{(A z_2-x+1)_{2 A -n}}$), and the $z_1$ contour does not have to be dragged through any point in $[-2, -1] \cup [0, 1]$ (these two intervals contain the poles of~$\frac{1}{(-2A-A z_1)_{A} (-A z_1)_{A}}$).

\begin{figure}
    \centering
\includegraphics*[scale=.15]{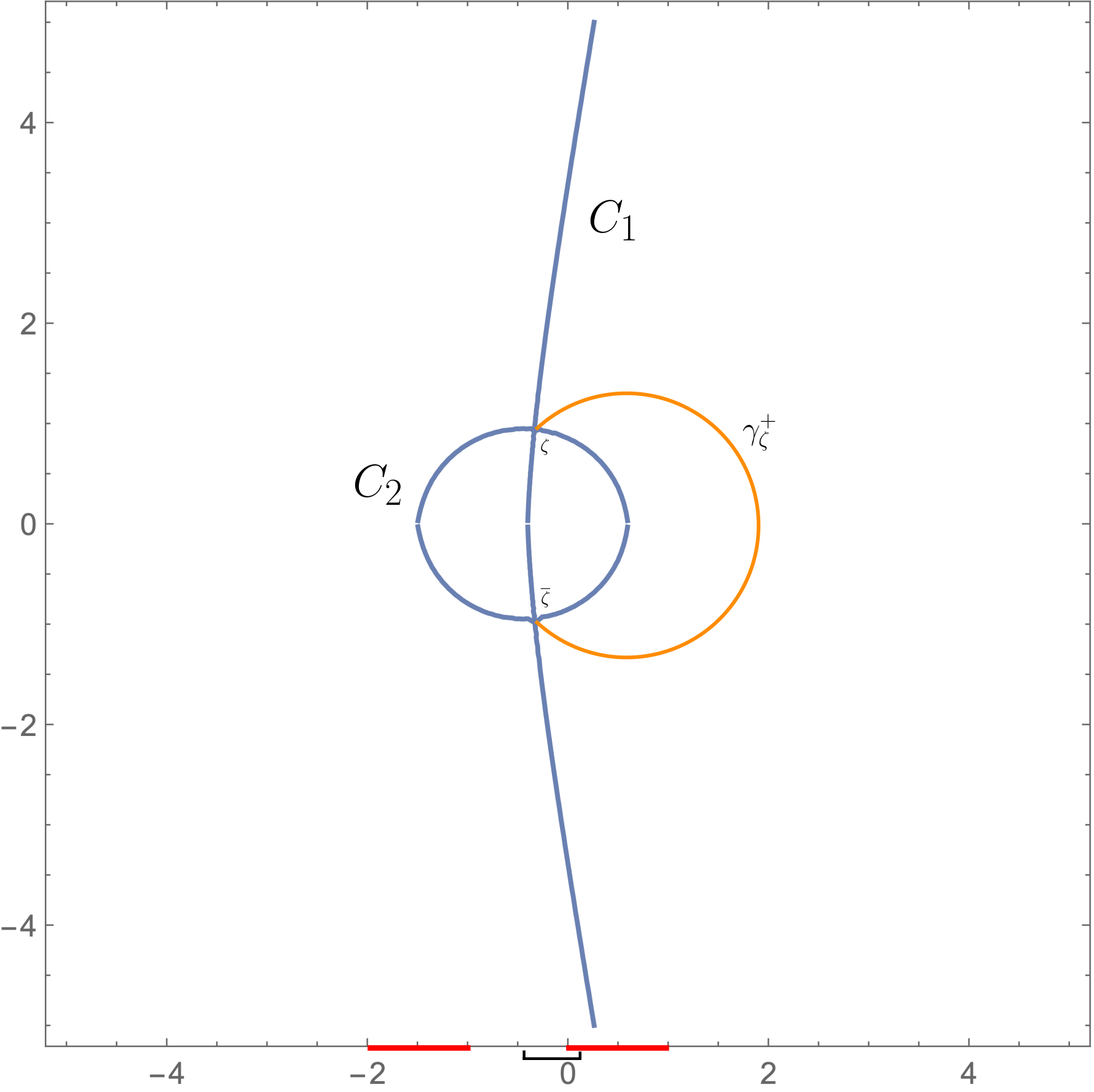}
\caption{The contours $C_1, C_2$ that we deform the $z_1, z_2$ contours to, respectively, in order to perform a steepest descent analysis. They intersect at the critical points $\zeta = \zeta(\chi, \eta)$ and $\overline{\zeta}$. The interval $[\chi + \eta - 2, \chi]$ is shown in black, and $[-2, -1] \cup [0,1]$ is shown in red. The contour $C_1$ crosses the real axis once in the interval $[\chi + \eta - 2, \chi] \setminus ([-2,-1] \cup [0, 1])$ (the part of the black interval not overlapping with red), and moves off to $\infty$. Moving along $C_1$ in the upper half plane, $\Re[S(z;\xi, \eta)]$ increases as we move away from $\zeta$. The curve $C_2$ intersects $\mathbb{R}$ twice, once in $[0,1]$ and once in $[-2,1]$, and $\Re[S(z;\xi, \eta)]$ decreases as we move away from $\zeta$ along this curve in the upper half plane.
The curve $\gamma_{\zeta}$ is shown in orange; this is the integration contour of the $z_1 = z_2$ residue we pick up from deforming the contours.}
\label{fig:C1C2}
\end{figure}

From the residue at $z_1 = z_2$ which we pick up while dragging the contours past each other, we obtain an integral over a curve $\gamma_{\zeta}^+$ from $\overline{\zeta}$ to $\zeta$, crossing the real line at a point $> 1$. See Figure \ref{fig:C1C2}. So we see that the integral \eqref{eqn:main_int} is equal to 
\begin{align*}
   &  \frac{A (2 A-1)!}{\Delta(A)}   \frac{1}{(2 \pi \i)^2} \int_{C_1} \int_{C_2}  dz_2 dz_1 \; e^{A (S(z_2; \frac{x}{A}, \frac{n}{A} ) - S(z_1; \frac{x}{A}, \frac{n}{A} )) + O(\frac{1}{A})} \frac{h(z_2)}{h(z_1)} \frac{1}{z_1-z_2} f_A(z_2) g_A(z_1) \\
       &-
       \frac{A (2 A-1)!}{\Delta(A)}  \frac{1}{2 \pi \i} \int_{\gamma_{\zeta}^+}   dz  f_A(z) g_A(z) .
    \end{align*}
Using \eqref{eqn:Rf12_def}, we decompose the display above into four terms as follows:

\begin{align}
    &    \frac{1}{(2 \pi \i)^2} \int_{C_1} \int_{C_2}  dz_2 dz_1 \; e^{A (S(z_2; \frac{x}{A}, \frac{n}{A} ) - S(z_1; \frac{x}{A}, \frac{n}{A} )) + O(\frac{1}{A})} \frac{h(z_2)}{h(z_1)} \frac{1}{z_1-z_2} (f(z_2) + R_{f, 1}(z_2)) g_A(z_1) \label{eqn:i1}\\
&+
\frac{1}{(2 \pi \i)^2} \int_{C_1} \int_{C_2}  dz_2 dz_1 \; e^{A (S(z_2; \frac{x}{A}, \frac{n}{A} ) - S(z_1; \frac{x}{A}, \frac{n}{A} )) + O(\frac{1}{A})} \frac{h(z_2)}{h(z_1)} \frac{1}{z_1-z_2} R_{f, 2}(z_2) g_A(z_1) \label{eqn:i2} \\
& 
-\frac{1}{2 \pi \i} \int_{\gamma_{\zeta}^+}   dz  R_{f,2}(z) g_A(z) \label{eqn:i3} \\
        &
      -  \frac{1}{2 \pi \i} \int_{\gamma_{\zeta}^+}   dz  (f(z) + R_{f, 1}(z)) g_A(z) . \label{eqn:i4}
\end{align}
We now proceed to bound the terms \eqref{eqn:i1}-\eqref{eqn:i3}, to show that their contributions are negligible.

We start with \eqref{eqn:i2}, which we may rewrite as
\begin{align}\label{eqn:i2_2}
    \eqref{eqn:i2} =  \frac{1}{(2 \pi \i)^2}\int_{C_1} \int_{C_2 \cap \{\Re[z] > -\frac{1}{2}\}}   \frac{dz_2 dz_1}{z_1 - z_2} \; e^{A (S(z_2; \frac{x}{A}, \frac{n}{A} ) + S(-\frac{1}{2}; 0, 2 ) - S(z_2; 0, 2 ) - S(z_1; \frac{x}{A}, \frac{n}{A}) ) + O(\frac{1}{A})} \cdot O(\sqrt{A})  .  
\end{align}
where $O(\sqrt{A})$ represents a function which is uniformly $O(\sqrt{A})$ in compact subsets of $\mathbb{C} \cup \{\infty\}$ bounded away from~$\{-2, -1, 0,1\} \subset \mathbb{R}$. Furthermore, the integrand is holomorphic in $z_2$ away from $[-2,-1] \cup [0,1] \subset \mathbb{R}$, and decays faster than $\frac{1}{|z_2|^2}$ as $z_2 \rightarrow \infty$, so that we may deform the contour through $\infty$. In order to describe how we deform the contour of integration, we need the following lemma.

\begin{lemma}\label{lem:xilemma}
    Let $(\chi, \eta) = (\frac{x}{A}, \frac{n}{A}) \in \mathfrak{H}$ be rescaled positions in the hexagon, and let $\zeta = \zeta(\chi, \eta)$ be the corresponding critical point as above, with $\Re[\zeta] \neq -\frac{1}{2}$. Then, the contour $C_2$ intersects the line $\Re[z] = -\frac{1}{2}$ exactly once in the upper half plane, say at $\xi \in \mathbb{H}$, and it also intersects it in the lower half plane at $\overline{\xi}$. The point $\xi$ is the only intersection point in $\mathbb{H}$ of $\Re[z] = -\frac{1}{2}$ with the level curve $\{\Im[S(z; \chi, \eta)] = \Im[S(\zeta; \chi, \eta)]\}$.

    If instead $\Re[\zeta] = -\frac{1}{2}$, then $C_1 = \{\Re[z] = -\frac{1}{2} \}$.
\end{lemma}
\begin{proof}
Due to the properties of the contour $C_2$ described in Figure \ref{fig:C1C2}  (which follow from \cite[Proposition 7.9]{Pet14}), as the curve $C_2$ traverses the upper half plane (counterclockwise), it moves between a point in $[0,1]$ and a point in $[-2,-1]$. Therefore, it must intersect the line $\Re[z] = -\frac{1}{2}$ in the upper half plane. 

Furthermore, we may check by direct calculation that 
$$\partial_y \Im[S(-\frac{1}{2} + \i y ;\chi, \eta)] = \log \left| \frac{-1/2 - \chi + \i y}{3/2 - \eta - \chi + \i y} \right|$$
which is nonzero for $y \in (0, \infty)$, so long as $\eta \neq 1 - 2 \chi$. But this condition is equivalent to the condition $\Re[\zeta] \neq -\frac{1}{2}$ (by inspecting the exact formula for $\zeta$), which we assumed. This implies that the intersection point with $C_2$ is in fact the only point of intersection (in the upper half plane) with the level curve $\{\Im[S(z; \chi, \eta)] = \Im[S(\zeta; \chi, \eta)]\}$. 

In the case that $\Re[\zeta] = -\frac{1}{2}$, we actually have $\eta = 1 - 2 \chi$, which means $\partial_y \Im[S(-\frac{1}{2} + \i y ;\chi, \eta)] \equiv 0$ for $y \neq 0$. This implies the claim.
\end{proof}

Now we may use the lemma to deform the contours in \eqref{eqn:i2_2}. We denote $F(z_2) \coloneqq S(z_2; \frac{x}{A}, \frac{n}{A} ) + S(-\frac{1}{2}; 0, 2 ) - S(z_2; 0, 2 )$. There exists a curve in $\mathbb{H}$ starting at $\xi$, the point along $C_2$ where $\Re[\xi] = -\frac{1}{2}$, such that $\Im[F(z_2)] = \Im[F(\xi)]$ is constant and $\Re[F(z_2)]$ is strictly decreasing as we move along the contour away from $\xi$. We claim that this curve moves off to $\infty$ as we move away from $\xi$, and does not intersect $C_2$ at any point other than $\xi$. We call $\tilde{C}_2$ the union of this curve with its complex conjugate in the lower half plane.

To verify the claim we argue as follows: Since $F$ has no critical points (this is a straightforward check), the level set $\{z_2 \in \mathbb{H} : \Im[F(z_2)] = \Im[F(\xi)]\}$ is a smooth curve. From analyzing the imaginary part of $F$ along the real axis, and by Lemma \ref{lem:xilemma}, we can see that the part of this level set that we want either moves through the interior of $C_2$ until it intersects the real axis in $[\chi + \eta -2, \chi]$, or it remains in the exterior of $C_2$ and goes off to $\infty$. In fact, one can check that the intersection with the real axis happens exactly at the same point as the intersection of $C_1$ with the real axis (see Figure \ref{fig:ImF}). By computing that $\Re[F(z)] \sim (-2 + \eta) \log|z|$ for large $z$, we can deduce that the contour we want moves off to $\infty$. This verifies the claim.

\begin{figure}
\centering
\includegraphics[scale=.18]{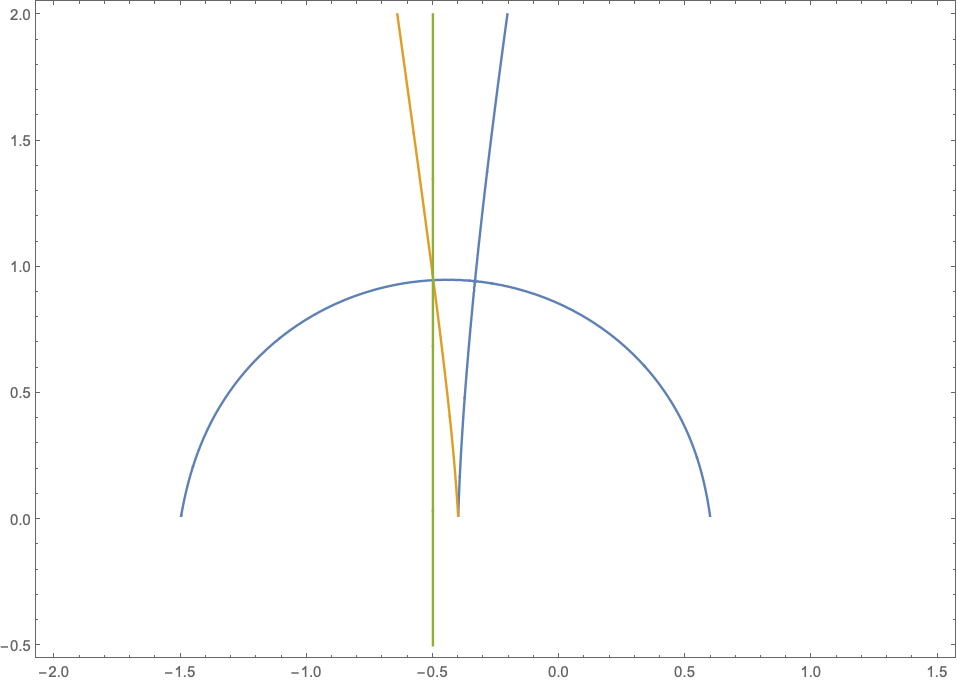}
\includegraphics[scale=.21]{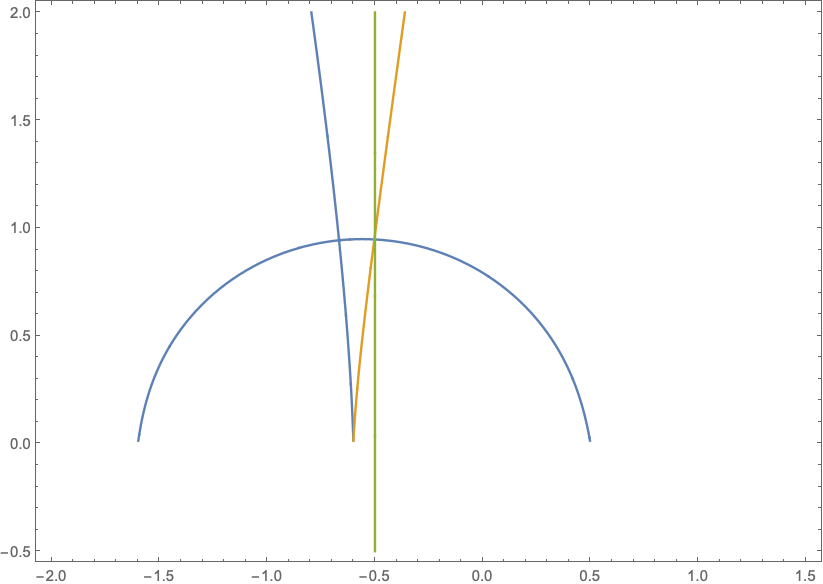}
\caption{We illustrate various contours in the upper half plane: We show $C_1$ and $C_2$ in blue, the level line $\Im[F(z)] = \Im[F(\xi)]$ in orange, and $\Re[z] = -\frac{1}{2}$ in green. On the left we have the case $\Re[\zeta] > -\frac{1}{2}$, and on the right we have the case $\Re[\zeta] < -\frac{1}{2}$.}
\label{fig:ImF}
\end{figure}

Thus, we can deform the contour $C_2 \cap \{\Re[z] > -\frac{1}{2}\}$, whose endpoints are $\xi$ and $\overline{\xi}$, to a contour $\tilde{C}_2$ moving from $\overline{\xi}$ to $\infty$ in the lower half plane, and then from $\infty$ to $\xi$ in the upper half plane along a level line of $\Im[F]$, such that along $\tilde{C}_2$, $\Re[F(z_2) - F(\xi)] < 0$.  We claim that as we deform the contour to this position, we either have $\Re[\zeta] \geq -\frac{1}{2}$ and we pick up a residue which exactly cancels \eqref{eqn:i3}, or $\Re[\zeta] < -\frac{1}{2}$ and we pick up no residue at all.

 First we observe, by unraveling definitions, that if any two of the three curves $\tilde{C}_2, C_1$ and $\Re[z]= -\frac{1}{2}$ intersect away from the real axis, then all three intersect. Together with Lemma \ref{lem:xilemma}, this implies that there is no intersection point of $\tilde{C}_2$ and $C_1$ in $\mathbb{H}$, unless if $\Re[\zeta] = -\frac{1}{2}$, in which case we have $\xi = \zeta$ and $C_1 = \tilde{C}_2 = \{\Re[z] = -\frac{1}{2}\}$. Thus, if $\Re[\zeta] \neq -\frac{1}{2}$, then $\tilde{C}_2$ and the part of $C_1$ starting at $\overline{\zeta}$ and moving through $\infty $ to $\zeta$ do not intersect. In fact, from the additional observation that the other part of $\Im[F(z_2)] = \Im[F(\xi)]$ intersects $\mathbb{R}$ at the same point as $C_1$, it follows that $\tilde{C}_2$ and this part of $C_1$ lie on opposite sides of the line $\Re[z] = -\frac{1}{2}$. See Figure \ref{fig:ImF} for an illustration. Thus, if $\Re[\zeta] > -\frac{1}{2}$, then as we drag the contour $C_2$ to $\tilde{C}_2$ in \eqref{eqn:i2}, we pick up a residue which cancels the integral \eqref{eqn:i3} exactly (indeed, in this case, we may deform $\gamma_{\zeta}^+$ in \eqref{eqn:i3} in such a way that it moves along $C_1$, from $\overline{\zeta}$ to $\zeta$ through $\infty$), and if $\Re[\zeta] < -\frac{1}{2}$ we may first deform $\gamma_\zeta^+$ to move along $\Re[z] = -\frac{1}{2}$, so that we pick up nothing. We can include $\Re[\zeta] = -\frac{1}{2}$ in the first case by slightly modifying the contour $\tilde{C}_2$, without affecting its key properties.

By examining \eqref{eqn:i2_2} with $\tilde{C}_2$ replacing $C_2$, we see that after this deformation of contours the integral \eqref{eqn:i2} is decaying exponentially. This is because the quantity in the exponential equals $A (F(z_2) - S(z_1; \chi, \eta))$, and along the new contours $F(z_2) - S(z_1; \chi, \eta) < -c < 0$ for some constant $c > 0$.

Now we move on to bound \eqref{eqn:i3}, in case it does not get cancelled out. In this case, we know from above that $\Re[\zeta] < -\frac{1}{2}$. Due to the indicator in the definition of~$R_{f,2}$, see ~\eqref{eqn:fA_approx}, the integral \eqref{eqn:i3} can be rewritten as 
\begin{align*}
    \frac{1}{2 \pi \i} \int_{\gamma_{\zeta}^+}   dz  R_{f,2}(z) g_A(z) =   \frac{1}{2 \pi \i}\int_{\gamma_{\zeta}^+ \cap \{\Re[z] \geq -\frac{1}{2}\}}   dz  R_{f,2}(z) g_A(z) . 
    \end{align*}
By analytic continuation (i.e. by removing the indicator prefactor), we may consider $R_{f,2}(z)$ to be holomorphic on $\mathbb{C} \setminus ( [-2, -1] \cup [0, 1] )$. If $\xi$ denotes the intersection of $\gamma_{\zeta}^+$ with the line $\Re[z] = -\frac{1}{2}$, we may deform $\gamma_{\zeta}^+ \cap \{\Re[z] \geq -\frac{1}{2}\}$ to a contour $\tilde{\gamma}_\zeta$ moving from $\xi$ to $\overline{\xi}$ through $\infty$ along $\Re[z] = -\frac{1}{2}$; we can do this since $R_f(z)$ decays faster than $\frac{1}{|z|^2}$ as $z \rightarrow \infty$. Thus, we have 
\begin{align*}
    \left|\frac{1}{2 \pi \i} \int_{\gamma_{\zeta}^+ \cap \{\Re[z] \geq -\frac{1}{2}\}}   dz  R_{f,2}(z) g_A(z) \right| &\leq   C \sqrt{A} \int_{\tilde{\gamma}_{\zeta} }     e^{A(\Re[S(-\frac{1}{2}; 0, 2)] - \Re[S(z; 0, 2)])} |g_A(z)| |dz| 
\end{align*}
     and since $\Re[S(-\frac{1}{2}; 0, 2)] - \Re[S(z; 0, 2)] < 0$ for $\Re[z] = -\frac{1}{2}, z\neq -\frac{1}{2}$, the integral \eqref{eqn:i3} is $O(e^{-A c})$ for some $c > 0$.

     The integral \eqref{eqn:i1} can be seen to be $O(\frac{1}{\sqrt{A}})$ by standard arguments, using the fact that $f(z_2) + R_{f, 1}(z_2)$, $g_A(z_1)$, and~$\frac{h(z_2)}{h(z_1)}$ are bounded along the contours, uniformly in~$A$.

Finally, the estimates in the discussions following \eqref{eqn:Rg_def} and \eqref{eqn:Rf12_def} lead to the final integral \eqref{eqn:i4} being given by
\begin{equation}\label{eqn:main_int_final}
-  \frac{1}{2 \pi \i} \int_{\gamma_{\zeta}^+}     f(z)  g(z) dz + O(\frac{1}{A}).
\end{equation}

To complete the analysis, it remains to observe that by continuity (in fact smoothness) of $S(z; \chi, \eta)$ and $\zeta(\chi, \eta)$ in the parameters $(\chi, \eta) \in \mathcal{D}$, all errors in the asymptotic analysis of \eqref{eqn:i1}-\eqref{eqn:i4} can be taken uniformly small for $(\frac{x}{A}, \frac{n}{A}) \in K$, where $K \subset \mathcal{D}$ is a compact subset of the liquid region.

This completes the proof of the limiting expression for the term \eqref{eqn:main_int} in the case that $(\chi, \eta) = (\frac{x}{A}, \frac{n}{A})$ are in the liquid region. Now to compute the limit of $\T$ we have to analyze the limit of the same term but with $(\chi, \eta)$ in the frozen region.

\begin{figure}
\centering
\includegraphics*[scale=.2]{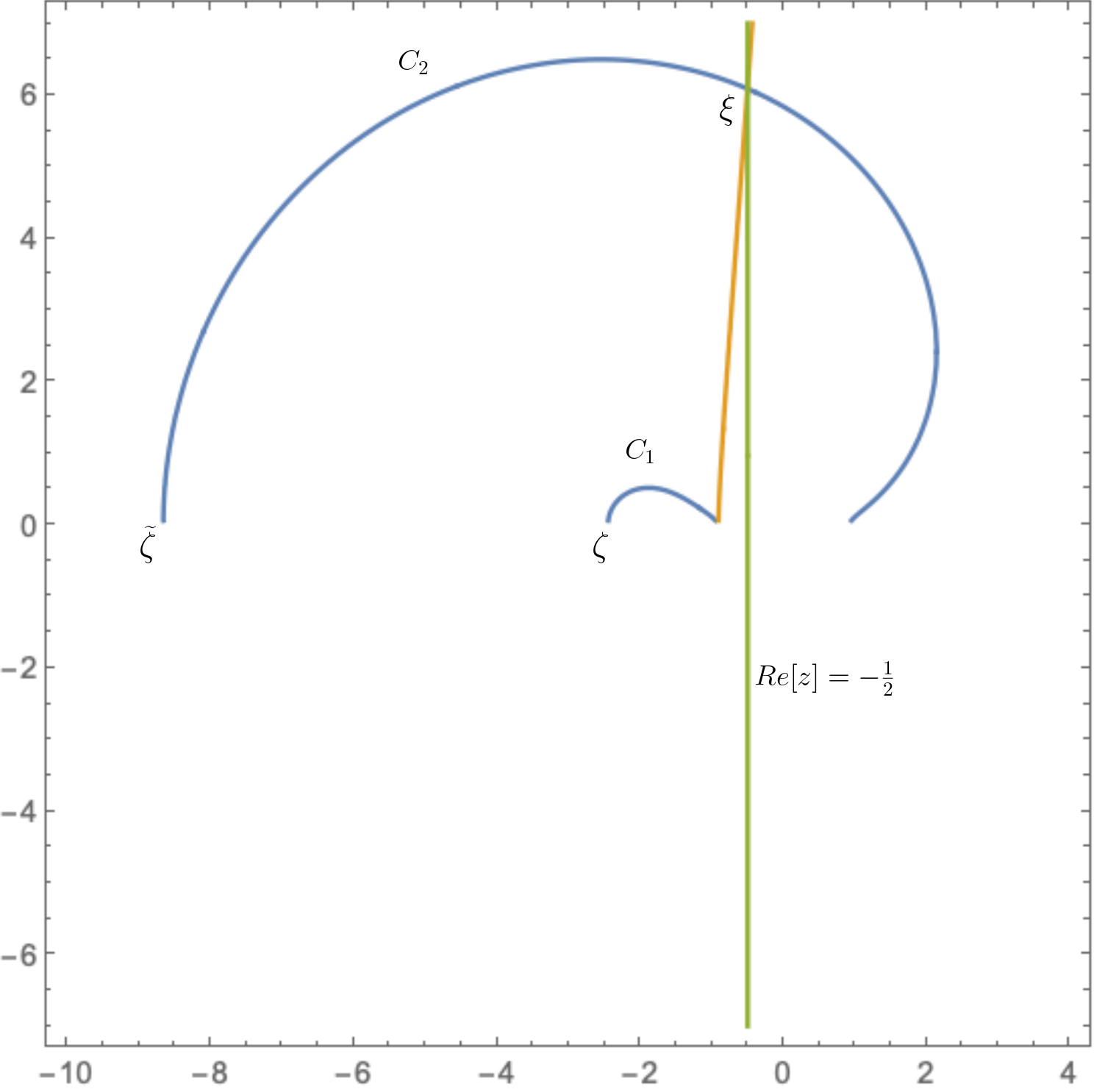}
\caption{Above we depict the parts of the steepest descent contours (in blue) for $(\chi, \eta)$ in one of the frozen regions. We only show the parts in the upper half plane. We also show the line $\Re[z]= -\frac{1}{2}$ (in green), and the curve $\Im[F(z)] = \Im[F(\xi)]$ (in orange), where $F$ and $\xi$ are defined as in the liquid region steepest descent argument. As before, we must deform $C_2$ in \eqref{eqn:i2_2} to the part of the orange curve outside of $C_2$ in order to bound \eqref{eqn:i2}.}
\label{fig:frozen_contours}
\end{figure}

\textbf{Case 2: } Next, we consider the case when $(\chi, \eta)$ is a point in the frozen region. We will only consider the case when $(\chi, \eta)$ is in the same frozen region as the corner $(0, \frac{1}{A})$; however, we begin with a discussion of a general scheme for performing the saddle point analysis in the frozen regions.

In the frozen region, by definition, $S(z; \chi, \eta)$ has two real critical points. From an explicit analysis of $S$, one may deduce that these two critical points lie in the same connected component of $\mathbb{R} \setminus \{-2, -1, -\frac{1}{2}, 0, 1\}$; see Lemma \ref{lem:frozen_crit}. In this case, we call the two real critial points $\zeta(\chi, \eta) > \tilde{\zeta}(\chi,\eta)$; these lie in one of six intervals in $\mathbb{R} \setminus \{-2, -1, -\frac{1}{2}, 0, 1\}$, corresponding to the six frozen regions. It is possible to show, by analyzing each of these six cases separately, that we can always deform the integration contours to two contours which pass through $\zeta$ and $\tilde{\zeta}$, respectively, and which \emph{do not intersect}. We will call the two contours $C_1$ and $C_2$; moving along $C_1$ away from the critical point, $\Re[S]$ will increase, and moving along $C_2$ away from the critical point, $\Re[S]$ will decrease. In the process of the deformation we will pick up a residue in such a way that the main contribution to \eqref{eqn:main_int} is given by \eqref{eqn:main_int_final}, with $\gamma_\zeta^+$ in this case being given by a circle crossing the real axis in the interval containing $\zeta$ as well as in the interval~$(1,\infty)$. 

We sketch the procedure for one of the six cases, corresponding to the frozen region containing the point $(\chi, \eta) = (0, \frac{1}{A})$, in the following paragraphs. In fact, to finish the proof we only have to analyze this frozen region, since the the value $(\chi, \eta) = (0, \frac{1}{A})$ itself corresponds to the second term in \eqref{eqn:T_exact_restate}. It can be checked that this frozen region is the set where~$\zeta(\chi, \eta) < -2$.

Suppose $\zeta < -2$. Throughout this discussion we define $S(z) \coloneqq S(z; \frac{x}{A}, \frac{n}{A})$. In this case the contours look as in Figure \ref{fig:frozen_contours}. Recall $\zeta > \tilde{\zeta}$. From analyzing the asymptotic behavior and intersections of the graphs of the two functions 
\begin{align*}
Q_1(z) &= (z- \chi)(z+2) z\\
Q_2(z) &= (z-\chi - \eta + 2)(z+1)(z-1)
\end{align*}
(the critical points are points where $Q_1(z) = Q_2(z)$), we can see that $S''(\zeta) < 0$, while $S''(\tilde{\zeta}) > 0$, which means $C_1$ passes through $\zeta$ and $C_2$ through $\tilde{\zeta}$. Furthermore, we can check (by analyzing $\Im[S]$ along the real axis) that the contour $C_1$ must either pass through $[\chi + \eta - 2, \chi] \setminus ([-2,-1] \cup [0,1])$ or go to $\infty$, and $C_2$ must intersect the real axis in $[0,1]$. Thus, as $C_1$ and $C_2$ must not cross in the upper half plane, both $C_1$ and $C_2$ are bounded, and $C_1$ intersects the real axis in $[\chi + \eta - 2, \chi] \setminus ([-2,-1] \cup [0,1])$. We can deform the contours to the desired position by first deforming the $z_1$ contour through $\infty$. We pick up a residue at $z_1 = z_2$ which equals the integral \eqref{eqn:main_int_final} with $\zeta = \zeta(\chi, \eta)$ in the interval $(-\infty, -2)$; in other words, the term which will be the leading order contribution is 

\begin{equation}\label{eqn:main_int_final_fr}
    -  \frac{1}{2 \pi \i} \int_{\gamma^-} f(z)  g(z)  dz  + O(\frac{1}{A}).
\end{equation}
where the contour is now a full counterclockwise loop $\gamma^-$, crossing $\mathbb{R}$ in $(1, \infty)$ and in $(-\infty, -2)$.

As in the liquid region case, we again have to bound the first three of the four terms \eqref{eqn:i1}-\eqref{eqn:i3}. With these choices of $C_1$ and $C_2$, it is clear that \eqref{eqn:i1} decays exponentially. We now have to bound the terms \eqref{eqn:i2} and \eqref{eqn:i3}. An argument very similar to the one given in the setting of the liquid region will indeed allow us to conclude that these terms are $o(1)$. In more detail, Lemma \ref{lem:xilemma} applies here as well, so to bound \eqref{eqn:i2} we can deform the contour $C_2 \cap \{Re[z]\geq-\frac{1}{2}\}$ to a contour moving from $\xi$ to $\infty$ in the upper half plane along which $\Re[F(z_2)] - Re[S(z_1)]$ is strictly negative (where here we consider the same $F$ from the proof in the liquid region setting), and then from $\infty$ to $\overline{\xi}$ in the lower half plane. See Figure \ref{fig:frozen_contours}. For \eqref{eqn:i3} we can deform $\gamma_{\zeta}\cap \{Re[z]\geq-\frac{1}{2}\}$ to a contour moving from $\xi$ to $\overline{\xi}$ through $\infty$ along $\Re[z] = -\frac{1}{2}$. With these new contours, we can now see that \eqref{eqn:i2} and \eqref{eqn:i3} are $o(1)$ (in fact, decaying exponentially) as $A \rightarrow \infty$. 

Now, subtracting the term \eqref{eqn:main_int_final_fr} we get from the frozen region from the term \eqref{eqn:main_int_final} that we got from the liquid region analysis, we obtain an integral along $\gamma^- - \gamma_\zeta^+ = -\gamma_\zeta$, which gives the result in the theorem statement. This completes the proof of the asymptotic \eqref{eqn:TLIM1}.

The proof for $\mathcal{O}$ is exactly the same, except we start with the formula \eqref{eqn:O_exact_hex} for the origami map instead.
\end{proof}

By using symmetries of the reduced hexagon~$H_{A}'$ and the corresponding embeddings~$\mathcal{T}_A$ under reflections, we can actually obtain the convergence of Theorem~\ref{thm:TLIM1} in the frozen regions as a corollary of the proof above.

\begin{proof}[Proof of Theorem~\ref{thm:TLIM1}: Frozen region case]
From \textbf{Case 2} in the liquid region proof above, see in particular~\eqref{eqn:main_int_final_fr}, we get that the embedding~$\mathcal{T}(\lfloor \chi A \rfloor , \lfloor \eta A \rfloor)$ converges exponentially fast to~$\mathcal{T}(v_1^*) = e^{-\i \frac{2 \pi}{3}} $ for~$(\chi, \eta)$ in the same frozen region as the point~$(0, \frac{1}{A})$.

From the exact formulas~\eqref{eqn:Fform_ren} and~\eqref{eqn:Gform} for~$\mathcal{F}^\bullet$ and~$\mathcal{F}^\circ$ in terms of the inverse Kasteleyn, we can see that the embedding (as a subset of~$\mathbb{C}$) is invariant under orthogonal reflections about the lines~$\i \mathbb{R}$,~$e^{ \i \frac{\pi}{3} } \mathbb{R}$, and~$ e^{- \i \frac{\pi}{3}} \mathbb{R}$ in~$\mathbb{C}$. C.f. Figure~\ref{fig:hex_temb}. Each of these symmetries corresponds to a symmetry of the reduced hexagon graph~$H_A'$. 

For example, the invariance under reflections about~$\i \mathbb{R}$ can be seen from the invariance of~$H_A'$, as drawn in Figure~\ref{fig:hex_reduction}, under reflection about the straight line connecting~$b_1$ and~$w_3$ (more precisely, this reflection induces an edge-weight-preserving and color-preserving graph isomorphism). Denote the latter reflection, viewed as a mapping of faces and vertices of~$H_A'$, by~$R_A$, and the former by~$R_{\i} : \mathbb{C} \rightarrow \mathbb{C}$. Using symmetry, we can check with the exact formulas~\eqref{eqn:Gform} and~\eqref{eqn:Fform_ren} that
$\mathcal{F}^\circ(R_A(w)) = \overline{\mathcal{F}^\circ(w)}$, and~$\mathcal{F}^\bullet(R_A(b)) = \overline{\mathcal{F}^\bullet(b)}$. Thus, for adjacent dual vertices~$v_1^*$ and~$v_2^*$ whose directed dual edge~$v_1^*-v_2^* = (wb)^*$, we have
$$\mathcal{T}(R_A(v_2^*)) - \mathcal{T}(R_A(v_1^*)) = - \mathcal{F}^\circ(R_A(w)) \mathcal{F}^\bullet(R_A(b)) =   R_{\i}\left( \mathcal{F}^\bullet(b) \mathcal{F}^\circ(w) \right) =R_{\i}\left(\mathcal{T}(v_2^*) - \mathcal{T}(v_1^*) \right) . $$
This together with the fact that~$R_{\i}\mathcal{T}(v_1^*) = \mathcal{T}(v_2^*) $ implies the claim. 

Now, since each corresponding reflection symmetry of the reduced hexagon~$H_A'$ exchanges a pair of adjacent frozen regions, the claim follows from the convergence of~$\mathcal{T}(\lfloor \chi A \rfloor , \lfloor \eta A \rfloor)$ in one of the frozen regions as described in the first paragraph. 

Similar symmetry considerations apply to the sequence of origami maps~$\mathcal{O}_A$, and this allows us to deduce the desired convergence of origami maps in the frozen regions as well.
\end{proof}


\bibliographystyle{plain}
\bibliography{bibliotek}

\end{document}